\documentclass[reqno, 11pt]{amsart}
\usepackage{amsmath,amssymb,cite}

\usepackage{amsthm} 
\usepackage{hyperref}
\usepackage{enumitem}
\usepackage{amsfonts}
\usepackage{verbatim}
\usepackage{stmaryrd}
\usepackage{fancyhdr}
\usepackage{graphicx}
\usepackage{subfigure}
\usepackage[left=1.2in, right=1.2in]{geometry}
\usepackage[usenames, dvipsnames,table]{xcolor}
\usepackage{amssymb,version,graphicx,fancybox,mathrsfs,multirow}
\usepackage{booktabs}      

\newcommand{\Cannote}[1]{\textcolor{blue}{#1}}

\theoremstyle{plain}
\newtheorem{theorem}{Theorem}[section]
\newtheorem{proposition}[theorem]{Proposition}
\newtheorem{lemma}[theorem]{Lemma}
\newtheorem{corollary}[theorem]{Corollary}

\theoremstyle{definition}

\newtheorem{assumption}[theorem]{Assumption}

\theoremstyle{remark}
\newtheorem{note}[theorem]{Note}

\numberwithin{equation}{section}
\numberwithin{theorem}{section}

\DeclareMathOperator{\bigo}{O}

\newcommand{\Dt}{\tau}

\newcommand{\E}{\mathbb{E}}

\DeclareMathOperator{\dist}{dist}

\newcommand{\R}{\mathbb{R}}

\newcommand{\N}{\mathbb{N}}
\newcommand{\be}{\begin{equation}}
\newcommand{\ee}{\end{equation}}

\newcommand{\con}{\mathbf{C}}

\begin{document}

\title[Numerical schemes for SPDEs]{Non-asymptotic uniform in time error bounds for new and old  numerical schemes for SPDEs.}

\author[C. Huang]{Can Huang$^{(1)}$}
\address{$^{(1)}$School of Mathematical Sciences, Xiamen University, Xiamen 361005, Fujian, China}
\email{canhuang@xmu.edu.cn}

\author[M. Ottobre]{Michela Ottobre$^{(2)}$}
\address{$^{(2)}$Maxwell Institute for Mathematical Sciences and Mathematics Department, Heriot-Watt University, Edinburgh EH14 4AS, UK}
\email{m.ottobre@hw.ac.uk}

\author[G. Simpson]{Gideon Simpson$^{(3)}$}
\address{$^{(3)}$Department of Mathematics, Drexel Univesrity, Philadelphia, 19104, USA}
\email{grs53@drexel.edu}

\begin{abstract}
We study numerical schemes for Stochastic Partial Differential Equations (SPDEs).  We introduce a general  method of proof of non-asymptotic uniform in time error bounds on numerical integrators for SPDEs, ensuring the schemes capture both the transient and the long term dynamics faithfully.  We then consider SPDEs with non-globally Lipshitz nonlinearities, which include for example the stochastic Allen-Cahn equation and some stochastic  advection-diffusion equations. For the case of Allen-Cahn type SPDEs  we show that the classic semi-implicit Euler  time-discretization can exhibit finite time blow up. This motivates analysing other schemes which do not suffer from this blow-up problem. We consider three numerical schemes for  SPDEs with non globally Lipshitz nonlinearity: a fully implicit scheme and two tamed schemes.  For these schemes we prove non-asymptotic uniform in time error bounds by leveraging our general criterion, and provide numerical comparisons. While the main emphasis in this paper is on the properties of the time-discretization, the schemes we consider are full space-time discretization of the SPDE.   

	\vspace{5pt}
    
\vspace{5pt}
\noindent  
{\sc AMS Classification (MSC 2020). 65C20, 65C30, 60H10, 65G99, 47D07, 60J60, 60H15. } 
\end{abstract}

\maketitle

\section{Introduction}


This work is devoted to the analysis of numerical schemes for Stochastic Partial Differential Equations (SPDEs) under full discretization, in both time and space.  Much of the existing analysis has been focused on either proving finite time error bounds, which imply that the scheme captures the transient phase of the dynamics correctly,  or establishing long-time properties of the scheme.  This latter challenge, of long time behavior, is often in relationship to accurate sampling of an invariant measure, provided the SPDE possesses one.


We introduce a set of easy to interpret conditions on the SPDE and the numerical scheme, which, when satisfied, are sufficient to imply error bounds that are uniform in time (UiT) and non-asymptotic, i.e. error bounds of the form
\begin{align}\label{eqn: general uiT bound intro}
\sup_{ k \in \mathbb N } \mathbb E \| u(k\tau) - u_{h}(k\tau)\| \leq  C(\tau, h) \bar C(u_0) \, . 
\end{align}
We state such critera in Section \ref{sec: sec2}. In the above expression $u(t)$ is the solution at time $t$ of the SPDE under consideration, so that, for each $t\geq 0$,  $u(t)$ is a random variable taking values in a suitable function space, such as $L^2(\R^d; \R)$. For brevity, we omit explicit dependence on the space variable $x \in \R^d$ and on the realization $\omega$ in the underlying probability space. Next, $u_{h}(k\tau)$ is a space-time approximation of the solution of the SPDE, with $h$ characterizing the  spatial mesh size  and $\tau$ the time-step size; $u_{h}(k\tau)$ is the discretization after $k$ time steps (i.e. at time $t =k\tau$) of $u$. As customary,   $\| \cdot \|$ denotes the $L^2$ spatial norm (though also other norms can be considered,  see Section \ref{sec: sec2}). Both $u(t)$ and the discretization are started with initial datum $u_0$. 

Below and in the next section we will introduce more precise notation and give more detail on the class of SPDEs and corresponding  numerical schemes we will consider.
For now we highlight that the essential feature of bound \eqref{eqn: general uiT bound intro}  is that both of the constants, $C$ and $\bar C$, appearing on the right hand side, are independent of time  - i.e. independent of $k$.   For this reason we say that the estimate \eqref{eqn: general uiT bound intro} is {\em{uniform in time}} (UiT).  Of course $C$ depends  on $\tau$ and  $h$.  For any {\em consistent} scheme,  $C$ vanishes as both $\tau$ and $h$ tend to zero,  with rates that depend on the specific scheme. For the schemes we will consider in this paper we will have 
$$
C(\tau, h) =  h+\tau^{1/2} \,.
$$
The constant $\bar C$ depends on the initial datum $u_0$.  

{A consequence of \eqref{eqn: general uiT bound intro} is that the numerical scheme has the desired property of capturing both short and long time dynamics, with errors that do not grow in time. In particular, assuming the SPDE \eqref{eqn:SPDEintro} has an invariant measure -- which will be the case under our assumptions, see Proposition \ref{prop:exponential convergence of SPDE} -- the scheme  correctly samples such a measure. {Were we only interested in approximating the invariant measure of the SPDE, then an asymptotic bound (asymptotic in time) would be sufficient}.  In contrast, the bound \eqref{eqn: general uiT bound intro}  is {\em{non-asymptotic}} (in time) - we explain this difference in Section \ref{sec: numerical schemes}. Hence, the challenge of producing bounds  like \eqref{eqn: general uiT bound intro}, which from now on we will refer to as non-asymptotic Uniform in Time bounds (naUiT),  has  attracted significant attention in the (SDE and SPDE) literature; see e.g. \cite{glatt2024long, jiang2025uniform, del2025time, del2020backward, schuh2024conditions, Angeli} and the very recent work  \cite{alfonsi2025euler} which also allows for variable step size.

The criteria we state in Section \ref{sec: sec2} should be seen as providing a general approach to prove  naUiT estimates \eqref{eqn: general uiT bound intro}, which also helps to  reconcile (and partly unify) other methods adopted in the literature, see Note \ref{note: comments on main thm}. After stating such criteria, we demonstrate their applicability by using them to prove naUiT for numerical discretizations of stochastic semilinear parabolic equations of the form \eqref{eqn:SPDEintro} below. We do this for three full space-time discretizations, namely the fully implicit Euler Scheme \eqref{scheme:im} and two tamed schemes, scheme \eqref{scheme: GTEM scheme} and \eqref{scheme:taming1}. 

{The methods used to derive our bounds have broader application beyond numerical discretizations of SPDEs, and they are applicable to  approximations of SDEs/SPDEs obtained via particle methods, multiscale techniques, and more; see, for example,  \cite{del2018stability, heydecker2019pathwise, durmus2020elementary, flandoli2022eddy, mischler2013kac, akyildiz2024multiscale, barre2021fast}}. However, the study of naUiT  bounds for approximations (be them numerical or not)  of SPDEs is still at an early stage compared to analogous questions in finite dimensional SDE dynamics. A broader purpose of this paper is to contribute to developing this branch of literature for SPDEs.   We elaborate on this point in Section \ref{sec: sec2}.}

{Our criteria for naUiT bounds, are analogs to those formulated for SDEs,  in \cite[Proposition 2.2]{Angeli}, and subsequently improved upon in  \cite{schuh2024conditions}. Indeed, the perspective of \cite{schuh2024conditions} allows for greater flexibility and easier adaptability to the infinite dimensional setting than \cite{Angeli}.  We leverage this fact in the present work, building directly on the results of \cite{schuh2024conditions}. 

The criteria formulated here are applicable to a large class of suitably well-posed SPDEs, and they do not require any particular features of the  nonlinear terms.  However,  to fix ideas, in the present work we consider SPDEs of the form 
\begin{equation}\label{eqn:SPDEintro}
\left\{
\begin{array}{ll}
du= \left(\Delta u  + f(u)\right) dt+  dW(t), & t\geq 0, x\in \mathcal O\\
u(t,x)=0, & x \in \partial \mathcal O, t\geq 0\\
u(0,x)= u_0(x), & x \in \mathcal O\, \,.
\end{array}
\right.
\end{equation}
for the unknown $u=u(t,x), \, t \in \R_+, x\in \mathcal O$, where $\mathcal O$ is a bounded subset of  $\R$.  The operator  $\Delta$ is the Laplacian with homogneous boundary conditions and $W$ is a trace class Wiener noise with covariance operator $Q$. Denoting by $L^2(\mathcal O; \R)$ the space of real-valued functions on $\mathcal O$ that are square-integrable,  we assume $Q:L^2(\mathcal O;\R)\rightarrow L^2(\mathcal O;\R)$ to be a symmetric, positive definite and trace class operator, so that there exists an orthonormal basis $\{\tilde{e}_j\}_{j \in \N}$ of $L^2(\mathcal O;\R)$ such that 
\begin{equation}
    Q \tilde e_j = \gamma_j \tilde e_j
\end{equation}
 with $\gamma_j \geq  0$ for every $j \in \mathbb N$ and 
 $$
 \sum_{j \in \N} \gamma_j < \infty \,.
 $$
 This implies that for any $v \in L^2$ (from now on we use $L^2$ as a shorthand for $L^2(\mathcal O; \R)$),  we can write 
$$
Q v= \sum_{j=0}^{\infty}  \gamma_j v_j \tilde{e}_j \,, \quad v_j:=(v, \tilde{e}_j)
$$
with $(\cdot, \cdot)$ denoting scalar product in $L^2$. In this setup, the $Q$-Wiener process $W$ can be written as 
\begin{equation}\label{eqn: Q wiener process definition}
W(t) = \sum_{j \in \N} \sqrt{\gamma_j}\, \tilde{e}_j \, \beta_j(t), \quad t\geq 0  \, ,
\end{equation}
where $\{\beta_j(t)\}_j$ is a collection of one-dimensional independent standard Brownian motions. 
The results we develop in this paper cover non-globally Lipshitz  nonlinearities such as  $f(u)=u-u^3$ or $f(u) = -u|u|^{q-2}, \, q>2$, which correspond to  stochastic Allen-Cahn and to a reaction-diffusion equation,  respectively \cite{prevot2007concise}. 
A complete list of rather standard assumptions on $Q,f$ (and on the operator $\Delta$),   as well as  comments on them, is deferred to Section \ref{sec: proof of main results}. Here we just note that under our assumptions (condition \eqref{assump eqn: growth of f} and Assumption \ref{assumption: on the covariance operator Q} in particular),  the problem \eqref{eqn:SPDEintro} has a unique mild solution,  see \cite[Proposition 6.2.2]{Cerrai01} and Section \ref{sec: proof of main results}.  

{Throughout this work, for simplicity, many of our examples and computations will be for problems posed over a subset  $\mathcal{O}$ of $\R$, like $[0,1]$ with Dirichlet boundary conditions.   But we emphasize that our results can be easily extended to the higher dimensional setting with analogous boundary conditions.}
 We restrict to dimension one for conciseness.    Extensions to bounded domains with periodic boundary conditions are also possible. We don't explore this fully in this paper but we observe that in this case care needs to be taken to appropriately choose the space discretization. We make comments on this point in Section \ref{sec: numerics}.
 
 

For numerical approximations of finite dimensional SDEs, it is well known that when the drift of the SDE is globally Lipshitz, the explicit Euler method correctly approximates  the dynamics on both  finite and  infinite time horizons \cite{mattingly2002ergodicity, Angeli, gyongy2005discretization}. However when the drift is  non-globally Lipshitz then explicit Euler may explode in finite time for sufficiently large data; see \cite[Lemma 6.3]{mattingly2002ergodicity} and a more precise statement in Section \ref{sec: why taming}. We prove that the same is true in the context of SPDEs. To be more precise, a fair analogue of the explicit Euler method for SDEs is given, in the SPDE context, by  the semi-implicit Euler scheme, where the Laplacian operator is treated implicitly while the nonlinearity is discretised explicitly, see \eqref{scheme:explicit Euler}.  Finite time error bounds for the semi-implicit Euler method for SPDEs of the form \eqref{eqn:SPDEintro} with globally Lipshitz  nonlinearity $f$ have  been provided in \cite{debussche2011weak}. We show that finite time blow up does occur when $f$ is non-globally Lipshitz, provided the initial datum is  sufficiently large. This fact, beyond intrinsic interest, also motivates the study of the fully implicit Euler scheme and of tamed schemes in this paper. 

In Section \ref{sec: numerics} we provide numerical experiments, comparing the schemes we analyse in this paper with others in the literature and making practical recommendations on which one to use.

The paper is organised as follows. In Section \ref{sec: sec2} and Section \ref{sec: numerical schemes}  we present our main results, and relate them to the existing literature. In Section \ref{sec: sec2} we state our criteria for naUiT convergence of numerical schemes for SPDEs, see Theorem \ref{thm: UiT theorem in our case} and Theorem \ref{thm: more general UiT theorem}, that is,  criteria to deduce naUiT bounds.  In Section \ref{sec: numerical schemes} we introduce the numerical schemes we consider and provide naUiT error bounds for each of them, see Theorem \ref{thm:Uit for fully implicit scheme}, Theorem \ref{thm: UiT estimate for GTEM} and Theorem \ref{thm: Uit bounds on scheme tamed with grad f}.  The proof of Theorem \ref{thm: UiT estimate for GTEM} - which contains naUiT bounds for one of the tamed schemes we study, namely  scheme \eqref{scheme: GTEM scheme} -  is particularly short. Indeed, we reiterate the point already made in \cite{schuh2024conditions}, that  the criteria of Section \ref{sec: sec2} make it easy to leverage  results which already exist in the literature.  A full list of assumptions as well as  the backbone of the proofs of our main results  are contained in Section \ref{sec: proof of main results}, with the proofs of specific estimates being deferred to Section \ref{sec:proof of uniform moment bounds} and Appendix \ref{sec: proofs of finite time error bounds}. In Section \ref{sec: why taming} we prove finite time blow up of the semi-implicit Euler scheme and in Section \ref{sec: numerics} we provide numerical experiments. The purpose of the numerical experiments of Section \ref{sec: numerics} is to  confirm theoretical results and to provide practical observations on the performance of the schemes we consider. Indeed, the three schemes we consider all satisfy the naUiT bounds \eqref{eqn: general uiT bound intro} so they all approximate correctly both the transient state and the invariant measure of the dynamics. However, when other considerations are taken into account (such as computational cost) then we find that,  a variant of the tamed scheme  \eqref{scheme: GTEM scheme}, namely the tamed scheme \eqref{scheme:taming by abs val f' pointwise standard form}, which we name {\em Truncated Tamed Scheme for SPDEs} is the one to be preferred. Further detail on this in Section \ref{sec: numerical schemes} and Section \ref{sec: numerics}.

\section{Uniform in time convergence of numerical schemes for SPDEs}\label{sec: sec2}

The main results of this section are Theorem \ref{thm: UiT theorem in our case} and a slight generalization of it, Theorem \ref{thm: more general UiT theorem}. These theorems  provide a set of sufficient conditions  which imply  uniform in time non-asymptotic bounds on the numerical error, i.e. bounds like \eqref{eqn: general uiT bound intro}.  Both  theorems are a simple adaptation to SPDEs of the approach of \cite{schuh2024conditions}. 

For the purposes of this section it is useful to introduce a slightly more detailed notation, namely
\begin{subequations}
\begin{align}
    u(t;u_0), & \quad \text{solution at time $t\geq 0$ of \eqref{eqn:SPDEintro} with initial datum $u_0$}, \label{detnot1}\\
    \begin{split}
    u_h(\ell \tau;u_0), &\quad \text{associated space-time discretization at time $t_{\ell}= \ell \tau$, $\ell \in \N$,}\\
    & \quad \text{with initial datum $u_0$}. \label{detnot2}
    \end{split}
\end{align}
\end{subequations}
In later sections we will just write $u_h^{\ell}$ for the space-time discretization at time $t_{\ell}=\ell \tau$, but here we maintain the more detailed notation \eqref{detnot1} - \eqref{detnot2}.


Below and throughout $H^1$ is a short hand for the Sobolev space $H^1(\mathcal O;\R)$ and $\|\cdot\|_1$ is the $H^1$ norm. 

\begin{theorem}\label{thm: UiT theorem in our case}
    With  the notation introduced so far, suppose the following three assumptions hold:
    \begin{enumerate}[label=\textnormal{[H\arabic*]},ref={[H\arabic*]}]
\item\label{Contractivity of SPDE}{\bf Contractivity of the dynamics.} There exist constants  $C, \lambda>0$ such that for any two initial data $u_0, v_0 \in H^1$ of  problem \eqref{eqn:SPDEintro} the following holds
\begin{equation}
    \|u(t;u_0) - u(t;v_0)\| \leq C e^{-\lambda t} \|u_0 - v_0 \| \,,\qquad  a.s. 
\end{equation}

\item\label{Finite time estimate}{\bf Finite time error bound.}
For every $\ell \in \mathbb N$ s.t. $\ell\tau \leq 1$ and for any $u_0 \in H^1$ the following bound holds
\begin{equation}\label{eqn: Finite time estimate}
    \sup_{\ell: \tau \ell \leq 1} 
    \mathbb E \|u(\tau \ell;u_0) - u_h( \tau \ell;u_0)\| \leq \tilde{C}(h, \tau) \mathbb E \|u_0\|_1^{\beta}
\end{equation}
for some constant $\tilde C= \tilde C (\tau, h)$ which depends only on $\tau$ and $h$, and for some $\beta>0$.

\item\label{Uniform in time moment bound}{\bf Uniform in time bound on the moments of the numerical scheme.}
For every $u_0 \in H^1$, there exists a constant $\hat C=\hat C(u_0)$, which depends only on $u_0$, such that
$$
\sup_{\ell \in \N} \mathbb E \|u_h(\tau\ell; u_0)\|_1^{\beta} \leq \hat C(u_0) \,.
$$
\end{enumerate}

Then there exists a constant $\bar C$ depending only on $\lambda$ and $u_0$, but not  $h, \tau $ or $ \ell$, such that 
\begin{equation}\label{UiT bound main thm}
    \sup_{\ell \in \mathbb N} 
    \mathbb E \|u(\tau \ell;u_0) - u_h(\tau \ell;u_0)\| \leq \bar C \, \tilde{C}(h, \tau) \,.
\end{equation}
\end{theorem}

Before proving the theorem, we briefly comment on its assumptions, with additional remarks in  Note \ref{note: comments on main thm}. Assumption \ref{Contractivity of SPDE} is an assumption purely on the SPDE dynamics \eqref{eqn:SPDEintro}, while \ref{Uniform in time moment bound} is an assumption on the numerical scheme. Assumption \ref{Finite time estimate} requires control on the finite time approximation error (in this case up to time one, but of course any finite time will do) and so this is an assumption on both  the dynamics and its numerical approximation.

\begin{proof}[Proof of Theorem \ref{thm: UiT theorem in our case}] This closely follows the result in \cite[Proof of Theorem 3.1]{schuh2024conditions}, so we only provide a sketch.  For any $\ell \in \N$,  let us  write $\ell \tau = k+m$, where $k \in \N$ and $m \in [0,1)$.   Then we can write \footnote{Strictly speaking so far we have only defined the discretization at times which are integer multiples of $\tau$. If $s> 0$ is not an integer multiple of $\tau$ then $u_h(s; \cdot)$ is understood to be a continuous time interpolation of $u_h(\tau n; \cdot)$, for example a piecewise linear interpolation. }
\begin{align*}
& \mathbb E \|u(\tau \ell;u_0) - u_h(\tau \ell; u_0)\| \\
= & \, \mathbb E \|u_h(m;u_h(k;u_0)) - u(m;u(k;u_0))\|\\
 \leq &  \, 
 \mathbb E  \| 
 u_h(m;u_h(k;u_0)) - 
 u(m; u_h(k;u_0))
 \| + \mathbb E \| u(m; u_h(k;u_0)) - u(m;u(k;u_0))
 \| \,.
\end{align*}
Let $I_1$ and $I_2$ be, respectively, the first and second addend on the RHS of the above.  Then, since $m \in [0,1)$, combining \ref{Finite time estimate} and \ref{Uniform in time moment bound}, for $I_1$ we have a straightforward estimate, 
$$
I_1 \stackrel{ \ref{Finite time estimate}}{\leq} \tilde{C}(h, \tau) \sup_{k \in \N} 
\mathbb E \|u_h(k; u_0)\|_1^{\beta} \stackrel{\ref{Uniform in time moment bound}}{\leq} \tilde C \hat C(u_0) \,. 
$$
As for the term $I_2$, 
\begin{align*}
I_2 & \leq \sum_{i=0}^{k-1} \mathbb E \|u(m+i; u_h(k-i;u_0)) - 
u(m+i; u(1;u_h(k-i-1;u_0))) \|\\
& \stackrel{\ref{Contractivity of SPDE}}{\leq}
C \sum_{i=0}^{k-1}  e^{-\lambda (m+i)} 
\mathbb E \| u_h (k-i;u_0) - 
u(1; u_h(k-i-1;u_0))\|\\
& = \,  C \sum_{i=0}^{k-1}  e^{-\lambda (m+i)} \mathbb E
\|u_h(1;u_h(k-i-1;u_0)) - u(1;u_h(k-i-1;u_0))\|\\
& \stackrel{\ref{Finite time estimate}\ref{Uniform in time moment bound}}{\leq} 
C \, \tilde{C}(\tau, h) \hat C(u_0)  \sum_{i=0}^{k-1}  e^{-\lambda (m+i)} \leq 
C \, \tilde{C}(\tau, h) \hat C(u_0)   \frac{1-e^{-\lambda k}}{1-e^{-\lambda}}  \,.
\end{align*}
This concludes the proof. 
\end{proof}


\begin{note}\label{note: comments on main thm}  

The conditions of Theorem \ref{thm: UiT theorem in our case} are completely analogous to those in the main theorem of \cite{schuh2024conditions}; here we just restate them in a way which makes them suited to the infinite dimensional context at hand. Other observations on the above theorem, in no particular order: 
\begin{itemize}
\item  The role of the SPDE dynamics $u$ and of its approximation $u_h$ are completely symmetric in the statement of the theorem; the theorem also holds if the contractivity of  Hypothesis \ref{Contractivity of SPDE} is enforced on the approximation and the uniform in time bound of Hypothesis \ref{Uniform in time moment bound} is proven for the dynamics $u(t; u_0)$. The work of \cite{glatt2024long}, which  is also devoted to providing an approach to prove  naUiT error bounds on SPDEs and their approximations,  can be interpreted in this spirit (modulo choosing the Wasserstein distance instead of the strong norm we choose here). That is, the approach of \cite{glatt2024long} can be obtained from Theorem \ref{thm: UiT theorem in our case}, after swapping the roles of the dynamics and its approximation. 

\item {Theorem \ref{thm: UiT theorem in our case} provides a general framework  for producing naUiT bounds of the form \eqref{UiT bound main thm}. As noted also in \cite{schuh2024conditions, glatt2024long}, the simplicity of this type of theorems hides the potential difficulty of actually proving  Hypotheses \ref{Contractivity of SPDE}, \ref{Finite time estimate}, \ref{Uniform in time moment bound}.  This merits additional comment.}

{Finite time error bounds of the form \eqref{eqn: Finite time estimate} are well studied in the literature, both for SDEs and SPDEs. Indeed, for Allen-Cahn type SPDEs, see, amongst others, \cite{brehier2020weak, brehier2018analysis, cui2019strong, liu2021strong, breit2024weak, cai2021weak, becker2023strong, ma2024strong}, with no claim to completeness of references. As for the contractivity property \ref{Contractivity of SPDE},  this is straightforward for Allen-Cahn SPDEs; see Lemma \ref{lemma:contractivity of SPDE}. {But for other evolutions this can be lengthy to prove, depending on both the evolution and the metric of interest. Again, various approaches have been proposed, using both probabilistic methods, e.g.  coupling techniques, or analytic approaches, e.g. Strong Exponential Stability, see  \cite{Angeli, crisan2022poisson, schuh2024global,liu2025l2, leimkuhler2024contraction, eberle2016reflection, glatt2024long, barrera2024ergodicity} and references therein.}  

The uniform in time bound \ref{Uniform in time moment bound} on the numerical scheme  usually requires an approach tailored to the specific dynamics. Despite the difficulties highlighted so far in proving the conditions required to apply Theorem \ref{thm: UiT theorem in our case},   it is also true that the literature is mature enough that, in some cases, many of the estimates are readily available. This will be the case for one of the schemes we study here, see Subsection \ref{subsec:truncated tamed} on this point.} 

\item Bound \eqref{UiT bound main thm} is a strong error bound. Similar bounds are usually obtained in weak form or in Wasserstein distance, \cite{glatt2024long, Angeli, crisan2021uniform, schuh2024conditions}. The key reason why we obtain strong naUiT bounds for our schemes is that Allen-Cahn type equations are contractive in a strong sense, see Lemma \ref{lemma:contractivity of SPDE}, and also because finite time strong error bounds are either easy to derive or already established in the literature. 

\item As we have repeatedly emphasized, the error bound \eqref{UiT bound main thm} is a {\em non-asymptotic} UiT bound.  In contrast, the UiT bounds commonly found in the literature are asymptotic (in time), see e.g. \cite{mattingly2002ergodicity, bou2023mixing, schuh2024convergence, duncan2017using, brehier2014approximation, cui2021weak}.  These bounds are aimed at understanding how well the numerical scheme approximates the invariant measure, denoted $\pi$, of the process $\{u_t\}_{t\geq 0}$ (when such an invariant measure exists).   Generally speaking, by ``asymptotic'' bounds we mean bounds of the form  
\begin{equation}\label{asymptboundintro}
\left|  \mathbb E\left(g\left( u_h(\tau k; u_0)\right)\right) - \pi(g)  \right| \leq a\,  e^{-\lambda k} + c \,  (\tau^{\hat \zeta} + h^{\zeta})
\end{equation}
for some constants  $a, \lambda, c, \zeta, \hat \zeta>0$ which are independent of time $k$ (though some of them might depend on $\tau$).  

These bounds are usually obtained by considering the invariant measure of the discretization $u_h(\tau k; u_0)$, $\pi^{\tau, h}$,  so that, by the triangle inequality, the first term on the RHS of \eqref{asymptboundintro} comes from studying exponential convergence of $u_h(\tau k; u_0)$ to  $\pi^{\tau,h}$ while the second comes from bounding the asymptotic bias, i.e. the distance between $\pi^{\tau, h}$ and the invariant measure $\pi$ of $\{u_t\}_{t\geq 0}$. Asymptotic bounds such as \eqref{asymptboundintro} are sufficient if one is {\em only} interested in sampling from the invariant measure $\pi$, but they are not necessarily informative of how the scheme approximates the dynamics for finite time.

\item When applying Theorem \ref{thm: UiT theorem in our case}, the detailed properties of the spatial discretization do not play a major role (except of course in \eqref{eqn: Finite time estimate}),  as the emphasis is on the properties of the time-discretization.

\item Theorem \ref{thm: UiT theorem in our case} is expressed in terms of  $L^2$ and $H^1$ norms, for readability and to relate it straight away to the type of norms that will appear in our bounds.  However, other norms for which these properties can be proven can be made to work and Theorem \ref{thm: more general UiT theorem} below is a slight generalisation of Theorem \ref{thm: UiT theorem in our case} in this sense. 

\item Because of the techniques we use to prove \ref{Finite time estimate}, the $H^1$ norm of the initial condition appears on the RHS of the finite time error estimates, like \eqref{eqn: Finite time estimate}. Had we used Malliavin calculus methods, as in \cite{cui2019strong}, then we could have obtained estimate in terms of the $L^2$ norm of the data. This would have had  the advantage that, instead of needing to bound the $H^1$ norm of the numerical scheme, as in \ref{Uniform in time moment bound}, it would suffice to bound the  $L^2$ norm and then use Theorem \ref{thm: more general UiT theorem} below (with $\mathbb M$ the expected value of the $L^2$ norm of $u_0$).  \end{itemize}
\end{note}
The following theorem is a straightforward generalization of Theorem \ref{thm: UiT theorem in our case}.

\begin{theorem}\label{thm: more general UiT theorem}
    Let $\| \cdot\|_B, \| \cdot\|_{\tilde B}$  be  norms  on the space $S$ of functions $u: \mathcal O \rightarrow \R$.\footnote{The space $S$ equipped with either norms should of course be good enough that the evolution \eqref{eqn:SPDEintro} is well posed in $B$ and that the LHS of \ref{Finite time estimate U} makes sense. } With  the notation introduced so far, suppose the following three assumptions hold:
    
    \begin{enumerate}[label=\textnormal{[A\arabic*]},ref={[A\arabic*]}]

\item\label{Contractivity of SPDE U}{\bf Contractivity of the dynamics.} There exist constants  $C, \lambda>0$ such that for any two initial data $u_0, v_0 \in S$  the following holds
\begin{equation}\label{eqn:contractivity norm B}
    \|u(t;u_0) - u(t;v_0)\|_{B} \leq C e^{-\lambda t} \|u_0 - v_0 \|_{\tilde B} \,, \quad a.s. 
\end{equation}

\item\label{Finite time estimate U}{\bf Finite time error bound.}
There exists a function $\mathbb M: B \rightarrow \R$ such that for every $\ell \in \mathbb N$ s.t. $\ell\tau \leq 1$ the following bound holds
\begin{equation}\label{eqn: Finite time estimate U}
    \sup_{\ell: \tau \ell \leq 1} 
    \mathbb E \|u(\tau \ell;u_0) - u_h( \tau \ell;u_0)\|_{\tilde B} \leq \tilde{C}(\tau, h) \mathbb M(u_0)
\end{equation}
for some constant $\tilde C= \tilde C (\tau, h)$ which depends only on $\tau$ and $h$.  

\item\label{Uniform in time moment bound U}{\bf Uniform in time control on the numerical scheme.}
There exists a constant $\hat C=\hat C(u_0)$, which depends only on $u_0$, such that
$$
\sup_{\ell \in \N} \mathbb M (u_h(\tau\ell; u_0)) \leq \hat C(u_0) \,.
$$
\end{enumerate}

Then there exists a constant $\bar C$ which depends on $\lambda$ and $u_0$ but does not depend on $N, \tau $ or $ \ell$ such that 
\begin{equation}
    \sup_{\ell \in \mathbb N} 
    \mathbb E \|u(\tau \ell;u_0) - u_h(\tau \ell;u_0)\|_B \leq \bar C \, \tilde{C}(h, \tau) \,.
\end{equation}

\end{theorem}
We omit a  proof of the above theorem as it is a minor variation on Theorem \ref{thm: UiT theorem in our case}.  The result can be easily obtained by adapting the proof of  \cite[Theorem 1.1]{schuh2024conditions}. Further generalizations of this theorem can be  obtained by exploiting the flexibility of the framework of \cite{schuh2024conditions}.


\section{Description of numerical schemes and main results on their convergence}\label{sec: numerical schemes}
In this section we introduce the three numerical schemes that we consider in this paper (and also further related schemes). We state here the main results we prove about each of these schemes, proofs are deferred to later sections.  

The space discretization is the same for all the schemes, and it is presented first, in Subsection \ref{subsec:space discretization} below. The different time-discretizations, and hence the three different schemes, are then described in Subsection \ref{subsec:fully implicit scheme} ({\em fully implicit Euler scheme}), Subsection \ref{subsec:truncated tamed} ({\em truncated tamed scheme}) and Subsection \ref{subsec:taming by grad fu} ({\em taming by $\nabla f$}). The fully implicit scheme is well-known in the literature, but only finite time error bounds have been shown so far;  the main contribution here is proving that naUiT bounds hold for this scheme. 
The two tamed schemes merit more comment, so for each of them we will explain the  rationale, we relate them to existing literature, and also consider some variations on them. We do not prove naUiT bounds for each of these variations, but we will numerically compare all of them in Section \ref{sec: numerics}.

\subsection{Space discretization}\label{subsec:space discretization} 
All the theoretical results that we present in this section refer to space- time discretizations where the space - discretization is given by the 
standard linear continuous piecewise Finite Element Method (FEM) \cite{Brenner, QuarteroniV97}, which we briefly recall below.  However in Section \ref{sec: numerics} we will show numerical experiments where the space discretization is produced by FEM and others where we use a Galerkin approximation, to numerically study the extent to which our theoretical findings are robust to different choices of space-discretization.  So below we briefly recall some essential notation for both FEM and Fourier spectral Galerkin approximations.  

\subsubsection{Finite Element Method (FEM)}
\label{s:fem}


In the first discretization we consider, let  $V_h$ be the space of continuous piecewise linear polynomials on the domain $[0,1]$ with Dirichlet boundary conditions over the uniform mesh
\begin{equation}
\label{e:femmesh}
    0=x_0<x_1<\ldots<x_{N} = 1, \quad h = 1/N \,.
\end{equation}
Equivalently, the space $V_h$ can be described as $V_h = \mathrm{span}\{\varphi_j\}_{j=1}^{N-1}$ where the $\varphi_j(x)$  are the ``hat functions":
\begin{equation}
\label{e:hatfunction}
    \varphi_j(x) = \begin{cases} \frac{x-x_{j-1}}{x_j - x_{j-1}}& x\in [x_{j-1}, x_j],\\
    1 - \frac{x_{j+1} - x}{x_{j+1} - x_{j}}& x\in (x_{j}, x_{j+1}], \qquad j=1, \dots, N-1 \\
    0 & \text{otherwise.}
    \end{cases}
\end{equation}
The space $V_h$ defined as in the above is sometimes also referred to as P1 FEM space (with the P1 referring to the fact that the polynomials are piecewise linear).  

Elements $u_h\in V_h$  admit the representation
\begin{equation}
    \label{e:FEM}
    u_h(x) = \sum_{j=1}^{N-1} u_{h,j} \varphi_j(x)\,,
\end{equation}
and the $\{u_{h,j}\}_{j=1}^N$ are determined either via $L^2$  projection or interpolation (the difference between the two choices being order $h^2$ in $L^2$,\cite{Brenner}) .  Regardless of how they are obtained, once the coefficients are set, $u_h(x_j) = u_{h,j}$.

The discrete Laplacian $\Delta_h$ can be defined accordingly,
\begin{align}\label{def:dislap}
(-\Delta_h v_h, w_h)=(\nabla v_h, \nabla w_h), \quad \mbox{for every } v_h, w_h\in V_h.
\end{align}
Furthermore,
$$
\|(-\Delta)^{\frac{1}{2}}v_h\|=\|\nabla v_h\|=\|(-\Delta_h)^{\frac{1}{2}}v_h\|, \ \ v_h\in V_h.
$$
In the sequel, we shall use the  orthogonal projection $P_h: L^2\to V_h$ defined as 
$$
(P_hv, w_h)=(v, w_h), \quad \mbox{for every } v\in L^2, \, w_h\in V_h. 
$$
Since we assume that the domain partition is uniform \cite{QuarteroniV97},
\begin{align}
&\|P_hv\|_1\leq c\|v\|_1, \\
&\|v-P_hv\|\leq ch\|v\|_1,\  \ \forall v\in H^1,
\end{align}
where $c$ is a constant independent of $h$.

\subsubsection{Fourier Pseudo Spectral-Galerkin method}
\label{s:fft}
While our theoretical results are given for the Dirichlet domain problem, we will also provide simulations in the periodic setting on the interval $[0, 2\pi]$,  where it is natural to use Fourier spectral-Galerkin methods and exploit the Fast Fourier Transform (FFT). In this periodic setting we will always consider the basis $V_h^{SG} = \{e^{ikx} \}_{k=-N/2+1}^{N/2}$, 
and fix  mesh size to $h=2\pi/ N$, with $N$ of the form $N=2^{n}$.  The choice of $N = 2^n$ is known to maximise computational efficiency of the FFT.   The method is pseudo spectral as nonlinearites are evaluated pointwise in real space.

For elements $u_h \in V_h^{SG}$,
\begin{equation}
\label{e:specgalerkin}
    u_h(x) = \sum_{j=-N/2+1}^{N/2} \hat{u}_{h}(j) e^{i j x}.
\end{equation}
$\hat{u}_{h}(j)$ is the $j$-th Fourier coefficient of $u_h$.   

\subsubsection{Notation for Time Discretization}

Before moving on to the various time-discretizations let us recall that from here on, instead of the more detailed notation of Section \ref{sec: sec2}, we use the shorthand  notation $u_h^{k}$ for a space-time discretization of \eqref{eqn:SPDEintro} at time $t_k = k\tau$,  where $\tau$ is the (fixed) time-step, so that $t_{k+1}-t_k = \tau$ for every $k\in\N$, and again $h$ refers to the mesh size of the space-discretization.  


\subsection{Fully and Semi-Implicit Euler scheme}\label{subsec:fully implicit scheme}
The Fully Implicit Euler (FIE) scheme for \eqref{eqn:SPDEintro} reads as follows: given $u_h^k$, the approximation $u_h^{k+1}\in V_h$ at time $t_{k+1}$ is given by
$$
u_h^{k+1} = u_h^{k} + \tau \,\Delta u_h^{k+1}  +
\tau \,P_h [f(u_h^{k+1})] + P_h (\delta W^{k+1}) \, 
$$
where $\delta W^{k+1} = W(t_{k+1})- W(t_k)$, recalling that $\{W(t)\}_{t\geq 0}$ is a $Q$-Wiener process, see \eqref{eqn: Q wiener process definition}. We iterate that in the above scheme and throughout we always consider FEM space discretization, unless otherwise stated. 
The FIE scheme can be equivalently  written in weak form as follows: 
\begin{equation}\label{scheme:im}
(u_h^{k+1}, \psi) = (u_h^{k}, \psi) + \tau (\Delta u_h^{k+1}, \psi)  +
\tau(f(u_h^{k+1}), \psi) + (\delta W^{k+1}, \psi), \quad \mbox{for every } \psi\in V_h.
\end{equation}
This time-discretization is certainly not new. For example, the properties of the time dicretization itself have been considered in \cite{breit2024weak,CuiH19,KovacsLL15} and the full space-time discretization in \cite{JentzenKW11, cui2021weak, liu2021strong, QiW19,FengLZ17}, along with references therein. These works produce  finite time error bounds, with the exception of  \cite{cui2021weak} where a UiT asymptotic bound is proved, see \cite[Corollary 1]{cui2021weak} -- in the context of that paper  the asymptotic bound is adequate as  the authors' goal is the approximation of the invariant measure. In this paper we prove the following (with detailed statement of assumptions postponed to Section \ref{sec: proof of main results}). 
\begin{theorem}\label{thm:Uit for fully implicit scheme}
    Suppose the SPDE \eqref{eqn:SPDEintro} satisfies Assumption \ref{assump1},  Assumption \ref{assump:coercivity assumption on overall drift} with $m=2$, as well as Assumption \ref{assump3}, Assumption \ref{assump2} and Assumption \ref{assumption: on the covariance operator Q}.  
    Suppose moreover that there exist constants $K_1 \in \R, K_2 \in \R_+$ such that 
    \begin{align}
    & (f(u) - f(v))(u-v) \leq K_1 |u-v|^2 \label{C21} \\
    & |f'(u)| \leq K_2 (1+ |u|^{p-1}) \label{C22}
    \end{align}
    and let $u_h^k$ be the sequence produced by the Fully Implicit Euler scheme \eqref{scheme:im}. Then the following holds 
     \begin{equation}\label{eqn: UiT estimate fully implicit}
       \sup_{k \in \N } \mathbb E \|u_h^k - u(t_k)\| \leq C (h+ \tau^{1/2}) \con(u_0)
   \end{equation}
   where $C$ is a generic constant (which will depend on $p$ but does not depend on time) and
   $\con(u_0)$ is a constant depending on the initial datum. 
\end{theorem}
The semi-implicit Euler (SIE) scheme for \eqref{eqn:SPDEintro} is instead given by
\begin{equation}\label{scheme: general SIE}
u_h^{k+1} = u_h^{k} + \tau \,\Delta u_h^{k+1}  +
\tau \,P_h [f(u_h^{k})] + P_h (\delta W^{k+1}) \, 
\end{equation}
We will show finite time blowup for this scheme in Section \ref{sec: why taming}.


\subsection{Truncated tamed scheme: taming by $f'(u)$}\label{subsec:truncated tamed}
We consider a tamed scheme and two variants of it. We describe them first and then motivate and discuss them in Subsection \ref{subsec: discussion of taming schemes}.


\subsubsection{Primary Taming Scheme: Truncated Tamed scheme for SPDEs, pointwise version.}
The main tamed scheme that we wish to discuss is
\begin{equation}\label{scheme:taming by abs val f' pointwise standard form}
u_h^{k+1} = u_h^{k} +  \Delta u_h^{k+1}  \tau +
P_h\left(\frac{f(u_h^{k})}{1+\tau |f'(u_h^k)|}\right) \tau + P_h(\delta W^{k+1}) \,,
\end{equation}
or, in weak form 
\begin{equation}\label{scheme:taming2derivativeabsolute value f'}
(u_h^{k+1}, \psi) = (u_h^{k}, \psi) +  (\Delta u_h^{k+1}, \psi) \tau +
\left(\frac{f(u_h^{k})}{1+\tau |f'(u_h^k)|}, \psi\right)\, \tau + (\delta W^{k+1}, \psi) \,, \quad \mbox{for every } \psi \in V_h\,. 
\end{equation}
In the above expression, $| \cdot |$ denotes absolute value and by $f'(u)$ we mean the derivative of $f$ (viewed as a real valued function of real variable) with respect to its own argument, calculated at $u$.  So for example if $f(u)=\frac u2 - u^3$ then $f'(u)=\frac 12 -3 u^2$. 
If the function $f$ grows at most like a degree $p$ polynomial,  then $f'$ is a polynomial of degree $p-1$ so this scheme essentially corresponds to taming as follows:  
\begin{equation}\label{scheme:taming2derivativeabsolute value}
u_h^{k+1} = u_h^{k} +  \Delta u_h^{k+1} \tau +
P_h \left(\frac{f(u_h^{k})}{1+\tau \alpha |u_h^k|^{p-1}}\right) \tau + P_h (\delta W^{k+1}) \,,  
\end{equation}
for some constant $\alpha>0$ to be chosen. The taming in either one of the above schemes is local, in the sense that it only depends on the value of $f'(u_h^k)$ (as opposed to the scheme we discuss next, where cthe taming depends on the norm of $f'(u)$). The essential feature of these taming schemes is that in the large argument limit (i.e. when $u_h^k$ is large), the drift grows linearly. The choice of $\alpha$ has been discussed in \cite{Angeli}, here we just observe that 
if $f$ is a polynomial with leading coefficient $a_p$, then one can choose $\alpha= p\, a_p$.

\subsubsection{Truncated tamed scheme: global version.}

The first variant of \eqref{scheme:taming by abs val f' pointwise standard form} that we consider tames the nonlinearity by a global quantity; that is, the nonlinear term $f(u)$ is tamed by the  $L^2$ norm of $f(u)$ rather than by the pointwise absolute value of $f(u)$:
\begin{equation}\label{scheme:tame by norm of f'(u)}
u_h^{k+1} = u_h^{k} +  \Delta u_h^{k+1} \tau + P_h\left(
\frac{f(u_h^{k})}{1+\tau \|f'(u_h^k)\|}\right)  \tau +  P_h(\delta W^{k+1}) \,.
\end{equation}
Again, if $f$ is a polynomial of degree $p$, then a similar result can be achieved by the following scheme
\begin{equation}\label{scheme:taming2}
u_h^{k+1} = u_h^{k} +  \Delta u_h^{k+1} \tau +
P_h\left(\frac{f(u_h^{k})}{1+\tau \alpha \|u_h^k\|^{p-1}}\right) \tau +  P_h(\delta W^{k+1}) \,. 
\end{equation}

\subsubsection{Alternate taming: GTEM}
Another variant of \eqref{scheme:taming by abs val f' pointwise standard form} that we consider has been introduced  in \cite{liu2025geometric} and it corresponds to 
\begin{equation}\label{scheme: GTEM scheme}
    u_h^{k+1} = u_h^k + \Delta u_h^{k+1} \tau + P_h\left( \frac{f(u_h^k)}{\sqrt{1+\tau |u_h^k|^{2(p-1)}}} \right)\tau + P_h (\delta W^{k+1}) \,.  
\end{equation}
in the polynomial case, and
\begin{equation}\label{scheme: GTEM scheme f}
    u_h^{k+1} = u_h^k + \Delta u_h^{k+1} \tau + P_h\left( \frac{f(u_h^k)}{\sqrt{1+\tau  |f'(u_h^k)|^2}} \right)\tau + P_h (\delta W^{k+1}) \,,  
\end{equation}
for more general $f$. This scheme was called GTEM in \cite{liu2025geometric}. 
The growth of the tamed nonlinear term is still linear but, with respect to \eqref{scheme:taming by abs val f' pointwise standard form}, in the denominator there is morally a $\sqrt{\tau}$ rather than a $\tau$.  

Before we move on to discussing these schemes, for the reader convenience we recap them in Table \ref{tab:scheme} below, together with other schemes that will be considered later, for the purpose of numerical comparison. 

\begin{table}[ht]
\caption{Summary of numerical schemes explored and compared in this manuscript.}
\label{tab:scheme}
\small
\centering

\rowcolors{2}{gray!10}{white}

\begin{tabular}{|l|l|c|}
\hline
\textbf{Scheme Name} & \textbf{Scheme} & \textbf{Eq. No.} \\
\hline
Fully Implicit Euler (FIE) & $
u_h^{k+1} = u_h^{k} + \tau \,\Delta u_h^{k+1} +
\tau \,P_h [f(u_h^{k+1})] + P_h (\delta W^{k+1})
$ & \eqref{scheme:im} \\
Semi-Implicit Euler (SIE) & $
u_h^{k+1} = u_h^{k} + \tau \,\Delta u_h^{k+1} +
\tau \,P_h [f(u_h^{k})] + P_h (\delta W^{k+1})
$ & \eqref{scheme: general SIE} \\
Truncated pointwise taming & Drift tamed by $1+\tau |f'(u)|$ & \eqref{scheme:taming by abs val f' pointwise standard form} \\
Truncated global taming & Drift tamed by $1+\tau \|f'(u)\|$ & \eqref{scheme:tame by norm of f'(u)} \\
GTEM & Drift tamed by $\sqrt{1+\tau |f'(u)|^2}$ & \eqref{scheme: GTEM scheme f} \\
Global gradient taming & Drift tamed by $1+\tau \|\nabla f(u)\|^2$ & \eqref{scheme:taming1} \\
Gy\"ongy pointwise taming & Drift tamed by $1+\tau|f(u)|$ & \eqref{e:gyongyspde} \\
\hline
\end{tabular}
\end{table}

\subsubsection{Discussion of Taming Schemes}\label{subsec: discussion of taming schemes}


The inspiration for taming scheme \eqref{scheme:taming by abs val f' pointwise standard form}, and its variants, comes from analogous challenges and methods in the finite dimensional setting.  Indeed, consider stochastic evolution equations of the  form 
\begin{equation}\label{generalSDE}
dX_t = b(X_t) dt+ \sigma(X_t) d\beta_t \,,
\end{equation}
where $X_t \in \R^d$ and, for the purposes of this discussion,   the drift $b:\R^d\rightarrow \R^d$ is assumed smooth, the diffusion matrix $\sigma: \R^d \rightarrow \R^{d\times d}$ is smooth and bounded  and $\beta_t$ is a $d$-dimensional standard Brownian motion.  As already noted, when the drift $b$ grows super-linearly the explicit Euler scheme may diverge from the actual solution in finite time \cite{mattingly2002ergodicity}. This motivated the introduction of tamed schemes, such as   
\begin{equation}\label{SDEclassicaltaming}
X^{n+1}=X^n+ \frac{b(X^n)}{1+\tau |b(X^n)|} \tau + \sigma(X_n) \delta W^n \,,   
\end{equation}
which is sometimes referred to as the {\em Gyongy method}. In the above $X^n$ is the discretization at time step $t_n=n\tau$. 
This scheme is convergent over finite time intervals \cite{gyongy2005discretization} and  it samples correctly the invariant measure \cite{brehier2022approximation}.  But bounds like \eqref{eqn: general uiT bound intro} have not yet been established for this scheme, mostly because it is not clear whether the scheme satisfies \ref{Uniform in time moment bound}; see \cite{brehier2022approximation} and \cite[Section 5]{Angeli} for a more extensive discussion of this.


On the other hand, it was proven in \cite[Section 5]{Angeli} that the scheme
\begin{equation}\label{SDEgradienttaming}
X^{n+1}=X^n+ \frac{b(X^n)}{1+\tau |b'(X^n)|} \tau + \sigma(X^n) \delta W^n \,, 
\end{equation}
has better properties, in two respects.  First it is possible to prove naUiT bounds of the form \eqref{eqn: general uiT bound intro} for this scheme.  Second, this scheme respects the Lyapunov structure of the SDE, assuming the SDE does have one.  If $b$ is a polynomial of degree $p$, instead of taming by $(1+ \tau |b'(x)|)$ we can just tame by the (euclidean norm) of $(1+ \tau \alpha |x|^{p-1})$. In this case it is not too difficult to see that  if the function $|x|^{p-1}$ is a Lyapunov function for the SDE then it is a Lyapunov function for the numerical scheme too, so the scheme has moments of order $p$ which are uniformly bounded in time (see \cite[Note 5.2]{Angeli} for more precise remarks and assumptions). Intuitively, when we use taming \eqref{SDEclassicaltaming} we replace the original drift $b(x)$ with a bounded drift;  when we use the taming \eqref{SDEgradienttaming} we replace the original drift $b(x)$ with a linear drift (which is still ``pushing'' in the correct direction).


Scheme \eqref{scheme:taming by abs val f' pointwise standard form} is the SPDE analog of \eqref{SDEgradienttaming}. In this infinite dimensional context one may then wonder whether it is wiser to tame pointwise, as in \eqref{scheme:taming by abs val f' pointwise standard form}, or tame by using the norm of $\|f'(u)\|$, as in \eqref{scheme:tame by norm of f'(u)}. The latter scheme is more expensive because at each iteration one needs to calculate the norm $\|f'(u_h^k)\|$. However, one may also argue that the increase in cost is only linear (in $h$) in many standard settings. We numerically compare these two schemes in Section \ref{sec: numerics}.

Recently the paper \cite{liu2025geometric} has proposed the scheme \eqref{scheme: GTEM scheme}, which is a slight variant of \eqref{scheme:taming by abs val f' pointwise standard form}, providing a more careful and detailed analysis of the Lyapunov structure of the scheme than the one described in \cite{Angeli}. In particular the work \cite{liu2025geometric} contains results on finite time convergence and geometric ergodicity, the latter  without explicit rates,   of \eqref{scheme: GTEM scheme}. In this paper we prove bounds of the form \eqref{eqn: general uiT bound intro} for the same scheme, see Theorem  \ref{thm: UiT estimate for GTEM} below.  One of our points is that Theorem \ref{thm: UiT theorem in our case} expedites the proof of naUiT bounds by leveraging results that are often already in the literature. Indeed, for the specific case of the scheme \eqref{scheme: GTEM scheme}, all needed estimates are already available.


We do not prove analogous bounds also for \eqref{scheme:taming by abs val f' pointwise standard form} and \eqref{scheme:tame by norm of f'(u)} - though it seems straightforward that they should hold, and indeed we provide evidence that this is the case and  numerically compare the schemes introduced in this section (and other tames schemes) in Section \ref{sec: numerics}.


\begin{theorem}\label{thm: UiT estimate for GTEM}
Suppose the SPDE \eqref{eqn:SPDEintro} satisfies  Assumption \ref{assump:coercivity assumption on overall drift} with $m=2$ and let $u_h^k$ be the sequence produced by the scheme \eqref{scheme: GTEM scheme}. 

Suppose furthermore that the initial datum $u_0$ is such that $\|-\Delta^{1/2} u_0\|< \infty$ and $u_0 \in L^{2(3p-2)}$ and that the function $f$ satisfies the following:
    \begin{enumerate}[label=\textnormal{[C\arabic*]},ref={[C\arabic*]}]
\item\label{pointwise cond 1} There exist positive constants $C_1, C_2, C_3, C_4$ such that
$$
f(u) \leq C_1+C_2 |u|^p
$$
and 
$$
u f(u) \leq C_3 - C_4 |u|^{p+1}\,.
$$
\item\label{pointwise cond 2}
There exist positive constants $K_1,K_2$ such that \eqref{C21} and \eqref{C22} are satisfied and a positive constants $K_3$ such that the following holds
\begin{align}
 [1-\tau |u|^{2(p-1)}] f'(u)
+ 2q\tau |u|^{2(p-2)} u f(u) \leq K_3 (1+ \tau |u|^{2(p-1)})^{3/2} \,.\nonumber
\end{align}
\end{enumerate}
   Then there exists a constant $C$, independent of $\tau, h$ and $k$ such that 
   \begin{equation}\label{eqn: UiT estimate for GTEM}
       \sup_{k \in \N } \mathbb E \|u_h^k - u(t_k)\| \leq C   (h+ \tau^{1/2}) \con(u_0), 
   \end{equation}
   where $\con(u_0)$ is again a constant which depends on the initial datum $u_0$. 
\end{theorem}


\subsection{Taming by the gradient of $f(u)$}\label{subsec:taming by grad fu}
The last tamed scheme we introduce is  (in weak form) as follows
\begin{equation}\label{scheme:taming1}
(u_h^{k+1}, \psi) = (u_h^{k}, \psi) +  (\Delta u_h^{k+1}, \psi) \tau +
\frac{\tau}{1+\alpha\tau \|\nabla f(u_h^k)\|^2}(f(u_h^{k}), \psi) + (\delta W^{k+1}, \psi), \ \psi\in V_h \,.
\end{equation}
For the sake of clarity, this time we are taming by the full gradient of $f(u)$, i.e. $\nabla f(u) = f'(u) \nabla u$. This is the most computationally intensive scheme of those proposed so far, as it requires both norm and gradient calculations at each step. 
For now we just present it and state below the main theorem we will prove on this scheme. In Section \ref{sec: numerics} we will compare it numerically to other schemes presented here and make further observations. 

\begin{theorem}\label{thm: Uit bounds on scheme tamed with grad f}
Suppose the SPDE \eqref{eqn:SPDEintro} satisfies the same assumptions as in Theorem \ref{thm:Uit for fully implicit scheme} and let $u_h^k$ be the sequence produced by the scheme \eqref{scheme:taming1}. Then the following holds 
     \begin{equation}
       \sup_{k \in \N } \mathbb E \|u_h^k - u(t_k)\| \leq C (h+ \tau^{1/2}) \con(u_0) \, 
   \end{equation}
   where $C$ is a generic constant (which will depend on $p$ but does not depend on time) and
   $\con(u_0)$ is a constant depending on the initial datum. 
\end{theorem}

\begin{corollary}
    Let $u(t)$ denote the solution to \eqref{eqn:SPDEintro} and $u_h^k$ the numerical approximation of $u(t)$ produced by the scheme \eqref{scheme:im} (\eqref{scheme: GTEM scheme}, \eqref{scheme:taming1}, respectively). Suppose that the assumptions of Theorem \ref{thm:Uit for fully implicit scheme} (Theorem \ref{thm: UiT estimate for GTEM}, Theorem \ref{thm: Uit bounds on scheme tamed with grad f}, respectively) are satisfied. Then the SPDE \eqref{eqn:SPDEintro} admits a unique invariant measure $\pi$ and  the following holds
    $$
  \left \vert \frac1K \sum_{k=0}^{K}  \mathbb E \varphi(u_h^k) - \int_{L^{1}} \varphi(u) d\pi(u) \right \vert \leq C\left( \frac{1}{K} + (h^2+\tau)^{1/2} \right) \con(u_0) \,,
    $$
    for any Lipshitz function $\varphi:L^1 \rightarrow \R$,  where in the above $C$ is a generic constant, independent of $t, h$ and $\tau$, while $\con(u_0)$ is a constant depending on the initial datum.
\begin{proof}
    The proof of this corollary is standard, using Proposition \ref{prop:exponential convergence of SPDE} below and following e.g. \cite[Corollary 3.6]{crisan2021uniform}. 
\end{proof}
    
\end{corollary}


\section{Assumptions and proofs of main results}\label{sec: proof of main results}
In this section we first provide a list of our main assumptions and comment on them; then in Subsection \ref{subsec: 4.1} we provide the backbone of the proof of   Theorem \ref{thm:Uit for fully implicit scheme}, Theorem \ref{thm: UiT estimate for GTEM} and Theorem \ref{thm: Uit bounds on scheme tamed with grad f} (proof of specific estimates needed in those proofs are deferred to later sections.) 

We enforce on $f$  the following. 

\begin{assumption}\label{assump1}
The nonlinearity  
$f$ is the Nemytskii operator associated with a real-valued $C^2$ function $\tilde{f}:\R\rightarrow \R$. The function $\tilde f$ is such that
\begin{equation}\label{assump eqn: growth of f}
    \sup_{y \in  \R} \frac{|\tilde{f}(y)|+ |\tilde{f}'(y)|+|\tilde{f}''(y)|}{1+|y|^p} < \infty
\end{equation}
for some $p\geq 2$. {\it From here on, with abuse of notation, we use $f$ for both the operator and the function.}
Moreover, $f$ satisfies
\begin{align}
&(f(u),u)\leq b_1 \, ,\label{eqn:coercive drift}\\
&(\nabla f(u), \nabla u)\leq b_2\|\nabla u\|^2 \, , \label{eqn:assump on gradient of drift}
\end{align}
for every $u \in H^1$, and for some $ b_1\geq 0$ and $b_2<1$. When dealing with the scheme \eqref{scheme:taming1}, we additionally require $b_2<\min\{1, \lambda_1^2\}$, where $\lambda_1$ is the smallest non-zero eigenvalue of $-\Delta$ on $\mathcal{O}$. 
\end{assumption}

\begin{assumption}\label{assump:coercivity assumption on overall drift}
The nonlinearity and the operator $\Delta$ are  such that
\begin{equation} 
    (\Delta(u-v), (u-v) |u-v|^{m-2}) + (f(u)-f(v), (u-v) |u-v|^{m-2}) \leq -  \lambda \|u-v\|_{L^m}^m, 
\end{equation}
for some $\lambda>0$ and for every $u,v \in L^m$ for which the LHS of the above is well defined and where $m$ is either equal to  $2$ or to $p$ (with $p$ as in \eqref{assump eqn: growth of f}). 
\end{assumption}

\begin{note}[Examples of drifts satisfying Assumption \ref{assump1} and Assumption \ref{assump:coercivity assumption on overall drift}.]
The Allen-Cahn type equation $f(u)=b_2u-b_3u^3\ (b_3>0)$ satisfies Assumption \ref{assump1} with $p=3$ since $(f(u),u)\leq \frac{2b_2}{b_3}|\mathcal{O}|-\frac{b_3}{2}\|u\|_{L^4}^4$ and $(\nabla f(u),\nabla u)\leq b_2\|\nabla u\|^2-b_3\|u\nabla u\|^2$. Assumption \ref{assump:coercivity assumption on overall drift} is easily verified for this example (with $m=2$). 

Another example is provided by the drift  
$f(u)= -u |u|^{q-2} $ with $q> 2$. For this nonlinearity also the bounds \eqref{eqn:coercive drift} \eqref{eqn:assump on gradient of drift} are satisfied, as it is straightforward to check (to verify the latter bound just use the fact that $\nabla f(u) = - (q-1) |u|^{q-2} \nabla u$). To verify that  Assumption \ref{assump:coercivity assumption on overall drift} holds with $m=2$, 
\begin{align*}
    (\Delta(u-v), u-v) + (f(u)-f(v), u-v) & = (\Delta(u-v), u-v) 
    + \int [- u |u|^{p-2} + v |v|^{p-2}] (u-v)\\
    &\leq (\Delta(u-v), u-v) \leq -c \|u-v\|_{L^2}^2
\end{align*}
having used Poincare' inequality in the last inequality and, in the penultimate inequality, we have used the fact that the function $-a|a|^{q-2}$ is decreasing for every $a \in\R$. 
\end{note}


\begin{assumption}\label{assump3}
 The initial condition $u_0$ is such that 
$
\mathbb{E}\|u_0\|_{1}^q<\infty$,  for every  $q\geq 2.$
\end{assumption}

Our assumptions on the covariance operator  $Q$ are as follows. 
\begin{assumption}\label{assump2}
The covariance operator $Q$ satisfies
\begin{align}
\|(-\Delta)^{\frac{1}{2}}Q^{\frac{1}{2}}\|_{HS}<C,
\end{align}
where $C$ is a constant that depends only on $Q$ and $\Delta$ and $\| \cdot \|_{HS}$ denotes Hilbert Schmidt norm.  
\end{assumption}
We recall the definition of the stochastic convolution $Z(t)$, namely
$$
Z(t) = \int_0^t e^{(t-s) \Delta} dW(s) = \sum_j  \sqrt{\gamma_j} \int_0^t e^{(t-s) \Delta} \tilde e_j \,d\beta_j(s) \,.
$$

\begin{assumption}\label{assumption: on the covariance operator Q} The covariance operator $Q$ is such that the stochastic convolution $Z$ is well defined in $L^{\infty}$ and moreover, the following holds
$$
\sup_{t\geq 0} \mathbb E \|Z(t)\|^\ell_{L^{\infty}} < C, \quad \mbox{ for all } \ell \in \mathbb N \, 
$$
where the constant $C$ may depend on $\ell$ but does not depend on $t$. 
\end{assumption}
It should be noted that, in dimension $d=1$, if Assumption \ref{assump2} is satisfied then also Assumption \ref{assumption: on the covariance operator Q} holds. We state them separately so their role in the proofs is more clear (in view of potential future extensions).  Indeed, because of the Sobolev embedding $H^1\hookrightarrow L^\infty$,  which holds in dimension one, the Poincare inequality and the Burkholder-Davis-Gundy inequality, we have 
\begin{align}
\mathbb E \|Z(t)\|_\infty^{\ell} &\leq C\mathbb E\|Z(t)\|_1^{\ell}\leq C\mathbb E\|(-\Delta)^{\frac{1}{2}}Z(t)\|^{\ell} \notag\\
&\leq C\mathbb E \bigg\|\int_0^t (-\Delta)^{\frac{1}{2}}\mathrm{e}^{(t-s)\Delta}dW(s)\bigg\|^{\ell}\notag\\
&\leq C_{\ell} \bigg(\int_0^t \|(-\Delta)^{\frac{1}{2}}Q^{\frac{1}{2}}\|_{HS}^2 \mathrm{e}^{-(t-s)\lambda_1^2}ds\bigg)^{\frac{\ell}{2}}\notag\\
&\leq C_{\ell}  \|(-\Delta)^{\frac{1}{2}}Q^{\frac{1}{2}}\|_{HS}^{\ell} \frac{1}{\lambda_1^{\ell}}.
\end{align}

It is well known that under our assumptions, the SPDE \eqref{eqn:SPDEintro} admits a unique invariant measure $\pi$ and it converges exponentially fast to it. Namely, the following holds. 
\begin{proposition}[Proposition 3.3 in \cite{brehier2022approximation} and Proposition 6.2.2 in \cite{Cerrai01}]\label{prop:exponential convergence of SPDE}
Suppose $f$ satisfies \eqref{assump eqn: growth of f} and Assumption \ref{assumption: on the covariance operator Q}. Then there exists a unique mild solution to the SPDE \eqref{eqn:SPDEintro} for initial data in $L^p$. 
    If also  Assumption \ref{assump:coercivity assumption on overall drift} holds,  then the SPDE \eqref{eqn:SPDEintro} admits a unique invariant measure $\pi$ which satisfies $\int \|u\|_{L^{\infty}}^k d\pi(u)< \infty$ for every $k \in \N$. Moreover, if $m \in \{2,p\}$ (where $p$ is as in  \eqref{assump eqn: growth of f}), then the following holds for every initial datum $u_0 \in L^m$ of \eqref{eqn:SPDEintro}: 
    $$
    \left \vert \mathbb E \varphi(u(t)) - \int_{L^{m}} \varphi(u) d\pi(u) \right \vert \leq  C(\varphi) e^{-\lambda t/2 } (1+\|u_0\|_{L^m}), \quad t\geq 0, 
    $$
    for every $\varphi:L^m\rightarrow \R$  which is Lipshitz continuous,  where $C$ is a constant which depends on the observable $\varphi$. 
\end{proposition}
A useful sketch of the proof of the above proposition can be found in \cite[Section 3]{brehier2022approximation}, see also \cite[Theorem 3.5.5]{berglund2019introduction}. The method of proof is detailed in  \cite[Theorem 6.3.3 and Chapter 4]{da1996ergodicity}.

\subsection{Proof of  Theorem \ref{thm:Uit for fully implicit scheme}, Theorem \ref{thm: UiT estimate for GTEM} and Theorem \ref{thm: Uit bounds on scheme tamed with grad f}}\label{subsec: 4.1}
Our main results will be proved by applying the frameworks of Section \ref{sec: sec2} to the schemes we presented. In particular we will use Theorem \ref{thm: UiT theorem in our case} to prove our naUiT bounds on the FIE scheme \eqref{subsec:fully implicit scheme} and on the tamed scheme \eqref{scheme:taming1}, i.e. to prove Theorem \ref{thm:Uit for fully implicit scheme} and Theorem \ref{thm: Uit bounds on scheme tamed with grad f}, respectively. We  use Theorem \ref{thm: more general UiT theorem} to prove Theorem \ref{thm: UiT estimate for GTEM}.   

When  using Theorem \ref{thm: UiT theorem in our case} (or Theorem \ref{thm: more general UiT theorem}) 
  the property required to hold in Assumption  \ref{Contractivity of SPDE}  of Theorem \ref{thm: UiT theorem in our case}  (or, analogously, in Assumption \ref{Contractivity of SPDE U} of Theorem \ref{thm: more general UiT theorem}) is a property of the SPDE itself and it does not depend on the scheme. So we prove it in Lemma \ref{lemma:contractivity of SPDE} below. Assumption \ref{Finite time estimate} and \ref{Uniform in time moment bound} depend on the scheme instead, and so we will need to prove them for each scheme separately. Such proofs are deferred to  Appendix \ref{sec: proofs of finite time error bounds} and Section \ref{sec:proof of uniform moment bounds}, respectively.

\begin{lemma}[Proposition 3.3 in \cite{brehier2022approximation}]\label{lemma:contractivity of SPDE}
    Let $m$ be as in Assumption \ref{assump:coercivity assumption on overall drift} and   $u(t;u_0)$ be the solution at time $t\geq 0$ of \eqref{eqn:SPDEintro}, with initial datum $u_0 \in L^m$. If Assumption \ref{assump:coercivity assumption on overall drift} holds with $m=2$, then the contractivity property \ref{Contractivity of SPDE} holds too.  If  Assumption \ref{assump:coercivity assumption on overall drift} holds with $m>2$ then the contractivity property \ref{Contractivity of SPDE} holds with the $L^2$ norm replaced by the $L^m$ norm on both sides of the inequality.
\end{lemma} 
\begin{proof}
    Just observe that $v(t) = u(t;u_0) - u(t; \tilde u_0)$ solves the equation 
    $$
    d v(t) = A v(t) + f(u(t; u_0)) - f(u(t; \tilde u_0)). 
    $$
    The conclusion follows by using Assumption  \ref{assump:coercivity assumption on overall drift}. 
\end{proof}

\begin{proof}[Proof of Theorem \ref{thm:Uit for fully implicit scheme} and of  Theorem \ref{thm: Uit bounds on scheme tamed with grad f}]
    To prove both results we use Theorem \ref{thm: UiT theorem in our case}. Condition \eqref{assump eqn: growth of f} and Assumption \ref{assumption: on the covariance operator Q} are required to ensure well-posedness of the SPDE. Moreover,  as we said, \ref{Contractivity of SPDE} is a property of the SPDE and it holds thanks to Lemma \ref{lemma:contractivity of SPDE} above. 
    The proof of the finite time error bound \ref{Finite time estimate} and of the of the bound \ref{Uniform in time moment bound} for the implicit scheme can be found in Subsection \ref{subsec:proof of H2 for fully implicit} and Subsection \ref{subsec:proof of H3 for implicit}, respectively. More precisely,  in Proposition \ref{Prop: finite tim error implicit} and  in Proposition \ref{prop:moment bound for implicit scheme}, respectively. 
    Similarly, the proof of the bounds \ref{Finite time estimate} and \ref{Uniform in time moment bound}  for the tamed scheme \eqref{scheme:taming1} are deferred to Subsection \ref{subsec: tamed H3} and Subsection \ref{subsec: tamed H2}, respectively. 
\end{proof}

\begin{proof}[Proof of Theorem \ref{thm: UiT estimate for GTEM}.]
To prove this result we use Theorem \ref{thm: more general UiT theorem} where in that theorem we choose  $\|\cdot \|_B$ to be  the $L^2$ norm  while  $\mathbb M(u_0) = (1+ \mathbb E \|u_0\|^{p(p+1)})$. So, with these choices,   we need to prove  Assumption \ref{Contractivity of SPDE U}, \ref{Finite time estimate U} and \ref{Uniform in time moment bound U}. From Lemma \ref{lemma:contractivity of SPDE}, if Assumption \ref{assump:coercivity assumption on overall drift} holds with $m=2$ then Assumption \ref{Contractivity of SPDE U} holds. Assumption \ref{Finite time estimate U} and \ref{Uniform in time moment bound U} have been proved in \cite{liu2025geometric}. Indeed notice that the scheme \cite[(3.11)]{liu2025geometric} is precisely scheme \eqref{scheme: GTEM scheme} when the diffusion coefficient is constant  (in that paper the authors consider non-constant diffusion coefficient, here for brevity we still consider constant diffusion). With this clarification\cite[Theorem 5.1]{liu2025geometric} (which in our constant diffusion coefficient setting holds under assumption \ref{pointwise cond 2}), gives precisely the finite time estimate 
$$
\sup_{\ell: \tau \ell \leq 1} \mathbb E\|u(\tau\ell; u_0) - u_{h}(\tau \ell; u_0) \|^2 \leq C (1+ \mathbb E \|u_0\|^{p(p+1)}) (h^2 + \tau), 
$$
where $C$ is independent of $\tau, h$ and $u_0$. 
The above estimate is just \cite[equation (5.5)]{liu2025geometric} with $\gamma=0, T=1$ and $q$ in that theorem replaced by $p-1$ in our notation. So  proving Assumption \ref{Uniform in time moment bound U} boils down to showing the following 
\begin{equation}\label{eqn: bounds r}
\sup_{\ell \in \N} \mathbb E \|u_h^\ell\|^{r} \leq \hat C
\end{equation}
when $r=p(p+1)$, for some constant $\hat C$ which does not depend on $\ell$. When $r=2$, the above bound is a consequence of \cite[Theorem 3.3]{liu2025geometric} (which holds under assumption \ref{pointwise cond 1}, which is again a rewriting, in our setting a notation, of the assumptions of  \cite[Theorem 3.3]{liu2025geometric}). Once this is done,  proving \eqref{eqn: bounds r} for $r\geq 2$ requires standard and lengthy calculations, which are completely analogous to computations we show in Section \ref{sec:proof of uniform moment bounds} for other schemes that we consider in this paper, see e.g. from equation \eqref{eq:pn1} on. So we omit this calculation here. 
    
\end{proof}

\section{Why taming: finite time explosion of the semi-implicit Euler Method}\label{sec: why taming}
In this section we show that, when the non-linearity $f$ in equation \eqref{eqn:SPDEintro} is not globally Lipshitz, the SIE scheme may blow up in finite time. As noted in the introduction, an analogous result is already well known for finite dimensional SDEs. Indeed, bearing in mind that a fair analogue of the SIE scheme for SPDEs is the Explicit Euler scheme for SDEs, we recall that   in \cite[Lemma 6.3]{mattingly2002ergodicity} the authors consider the Explict Euler approximation 
$\{X^n\}_{n\in \N}$ for the following SDE in $\R^d$
\be\label{SDE Mattingly}
dX_t= -X_t^3 \, dt  + \sqrt{2}\,  d\beta_t, 
\ee
(where $\beta_t$ is a $d-$ dimensional standard Brownian motion) and prove that the second moment of $X^n$ explodes if the initial datum is too large. More precisely, they show that if $\E |X_0|^2 \geq 2/\tau$ then $\E |X^n|^2 \geq \E |X_0|^2 + n\,  \tau $, so that $\E |X^n|^2$  tends to infinity as $n \rightarrow \infty$.

A similar result holds here. We consider the SPDE \eqref{eqn:SPDEintro} when $f(u)=-u^3$,   that is, 
\begin{equation}\label{SPDE Mattingly}
    du = (-u^3 + \Delta u) dt + \sqrt{2  } dW_t,
\end{equation}
and the associated SIE time-discretization, $\{u^n\}_{n \in \N}$, with time-step $\tau$, namely
\begin{equation}\label{scheme:explicit Euler}
    u^{n+1} = u^n - (u^n)^3 \Dt  + \Delta u^{n+1} \Dt  + \sqrt{2 \Dt}  \, \delta W^n .
\end{equation}
We conjecture that the above time-discretization will exhibit finite time blowup, irrespective of the chosen space-discretization. 
In Subsection \ref{subsec: Can blow up proof} we show blow up of a full space-time discretization of \eqref{SPDE Mattingly} with Dirichlet boundary conditions, where the time-discretization is given by the SIE scheme and space is discretized via FEM scheme; see Lemma \ref{lemma: Can explosion lemma}.  This is confirmed by numerical simulations, presented in Section \ref{sec: numerics}.

Additionally, via our simulations of Section \ref{sec: numerics}, we also observe blowup of the SIE scheme when we consider the SPDE \eqref{SPDE Mattingly} on the torus, and discretise space using pseudo-spectral Glerkin method. While a full proof of this fact is beyond the scope of the present work, we provide a calculation in Section \ref{s:periodicblowup} suggesting blowup.


\subsection{Blow up for space-time discretization}\label{subsec: Can blow up proof}

We consider the SIE sequence \eqref{scheme:explicit Euler} and further discretize space via the FEM discretization discussed in Section \ref{s:fem}. The resulting space-time discretization of \eqref{SPDE Mattingly}   solves
\begin{align}\label{sch1}
u_h^{n+1}-u_h^n=\Delta u_h^{n+1} \tau -(u^n_h)^3 \tau+\sqrt{2}\, \delta W^n \, ,\qquad u_h^n\in V_h \,.
\end{align}

\begin{lemma}\label{lemma: Can explosion lemma}
 Let $\{u_h^n\}$ be the sequence produced by \eqref{sch1}. There exists a constant $a_0>0$ dependent on $h$ and $\tau$ (we characterize this constant in the proof) such that if $\mathbb E{\|u_h^0\|^2} > a_0$ then 
 \begin{align}\label{eqn: explosion of euler scheme}
\mathbb{E}\|u_h^n\|^2\geq\mathbb{E}\|u_h^0\|^2+n\tau \to \infty, \text{ as} \ \ n\to\infty.
 \end{align}
 \end{lemma}
Before proving the above lemma we state and prove the next Lemma \ref{lemma: auxiliary lemma}.  
\begin{lemma}\label{lemma: auxiliary lemma}
Consider the following one-dimensional boundary value problem on the interval $(0,1)$: 
\begin{align}
&-\tau u''+u=v, \qquad u\in V_h, \notag\\
&u(0)=u(1)=0. 
\end{align}
Then we have 
 \begin{align}\label{eqn:lower bound for FEM}
 \|u\|^2\geq \frac{1}{1+2C_{\rm inv}^2\tau h^{-2}+C_{\rm inv}^2h^{-4}\tau^2}\|v\|^2:=C_0\|v\|^2.
 \end{align}
\end{lemma} 
 \begin{proof}[Proof of Lemma \ref{lemma: auxiliary lemma}]
 Recall the inverse inequality for FEM \cite[section 4.5]{Brenner}
 $$
 \|u\|_{W_p^l}\leq C_{\rm inv}h^{m-l+\min(0, \frac{d}{p}-\frac{d}{q})}\|u\|_{W_q^m}\,, \quad u \in V_h\, ,
 $$
 where  we are using the standard notation $W_p^\ell$ for the Sobolev space of $L^p$ functions with $\ell$ derivatives and $C_{\rm inv}$ is a positive constant. 
 Hence, 
 $$
 \|u'\|\leq C_{\rm inv}h^{-1}\|u\|, \ \ \text{and} \ \ \ \ \|u''\|\leq C_{\rm inv}h^{-2}\|u\|.
 $$
 We square both sides and use integration by parts, and obtain
 \begin{align*}
 \|v\|^2&=\|-\tau u''+u\|^2\\
 &=\|u\|^2+2\tau\|u'\|^2+\tau^2\|u''\|^2 \\
 &\leq \|u\|^2+2C_{\rm inv}^2\tau h^{-2}\|u\|^2+C_{\rm inv}^2\tau^2h^{-4}\|u\|^2 \, ,
 \end{align*}
 from which \eqref{eqn:lower bound for FEM} follows. 
 \end{proof}

 \begin{proof}[Proof of Lemma \ref{lemma: Can explosion lemma}]
 From the definition of the scheme \eqref{sch1}, we have
 $$
u_h^{n+1} - \tau\Delta u_h^{n+1} = u_h^n-(u_h^n)^3+\sqrt{2}\delta W^n \,.
$$
Since the sequence $\{u_h^n\}_{n \in \mathbb N}$ belongs to $V_h$,   we can use the lower bound \eqref{eqn:lower bound for FEM}, and write
 \begin{align}
\mathbb{E}\|u^{n+1}_h\|^2&=\mathbb{E}\|(I-\tau\Delta)^{-1}(u^n_h-\tau (u_h^n)^3+\sqrt{2}\delta W^n)\|^2\notag\\
 &\geq C_0\mathbb{E} \|u^n_h-\tau (u_h^n)^3+\sqrt{2}\delta W^n\|^2\notag\\
 &{=} C_0 [\mathbb{E}\|u^n_h-\tau (u_h^n)^3\|^2+2Tr(Q)\tau]\notag\\
 &{=}  C_0[\mathbb{E}\|u^n_h\|^2-2\tau\mathbb{E}\|u^n_h\|_{L^4}^4+\tau^2\mathbb{E}\|u^n_h\|_{L^6}^6+2Tr(Q)\tau]\notag\\
 &\geq C_0[\mathbb{E}\|u^n_h\|^2-2\tau\mathbb{E}\|u^n_h\|_{L^4}^4+\tau^2\mathbb{E}\|u^n_h\|_{L^4}^6+2Tr(Q)\tau]\notag\\
 &\geq C_0[\mathbb{E}\|u^n_h\|^2-2\tau\mathbb{E}\|u^n_h\|_{L^4}^4+\tau^2(\mathbb{E}\|u^n_h\|_{L^4}^4)^{3/2}+2Tr(Q)\tau]  \label{func tilde f}
 \end{align}
 where in the second equality we have used the fact that $u_h^n$ is independent of $\delta W^n$, and the last two lower bounds follow from Holder's inequality
 \begin{align}
 \|u\|_{L^4}^4=\int_0^1 u^4dx\leq \bigg(\int_0^1 u^6dx\bigg)^{2/3} \bigg(\int_0^1 1dx\bigg)^{1/3}=\|u\|_{L^6}^4, 
 \end{align}
 and from Jensen's inequality, respectively. 

 Let us set  $n=0$, and write \eqref{func tilde f} as 
 $$
 \mathbb E \|u^1_h\|^2 \geq C_0 \mathbb E \|u^0\|^2 + \tilde s(\mathbb E \|u^0_h\|_{L^4}^4)
 $$
 where 
 $\tilde s (x)$ is the function $\tilde s:\R_+ \rightarrow \R$ defined as 
 $$
 \tilde s(x) := C_0\tau^2 x^{3/2}-2\tau C_0 x+ 2Tr(Q) C_0\tau \,.
 $$
 Consider now the function
 $$
 s(x):= \tilde s(x) - (1-C_0) x-\tau = C_0\tau^2 x^{3/2}-(2\tau C_0+(1-C_0))x+2Tr(Q)C_0\tau-\tau, \ \ x\geq 0.
 $$
 We will show below that there exists a (unique) point $a_0\geq 1$ such that if $x\geq a_0 $ then $s(x)\geq 0$, i.e.
 \begin{equation}\label{use a_0}
     \tilde s(x) \geq (1-C_0) x +\tau, \quad \mbox{for } x\geq a_0 \,.
 \end{equation}
Once this is proved the proof is concluded. Indeed, take an initial datum $u^0$ such that $\mathbb{E}\|u_h^0\|^2\geq \sqrt{a_0}$.  Then  
$\mathbb{E}\|u_h^0\|_{L^4}^4\geq a_0$, because 
  $\mathbb{E}\|u_h^0\|_{L^4}^4\geq \mathbb{E}\|u_h^0\|^4\geq (\mathbb{E}\|u_h^0\|^2)^2\geq a_0$, and 
  we have
 \begin{align}
 \mathbb{E}\|u_h^1\|^2&\geq C_0\mathbb{E}\|u_h^0\|^2+(1-C_0)\mathbb{E}\|u_h^0\|_{L^4}^4+\tau\notag\\
 &\geq C_0\mathbb{E}\|u_h^0\|^2+(1-C_0)(\mathbb{E}\|u_h^0\|^2)^2+\tau\notag\\
 &\geq C_0\mathbb{E}\|u_h^0\|^2+(1-C_0)\mathbb{E}\|u_h^0\|^2+\tau\notag\\
 & = \mathbb{E}\|u_h^0\|^2+\tau \, ,
 \end{align}
 where the last inequality simply follows from the fact that $\mathbb{E}\|u_h^0\|^2 \geq 1$ (because $a_0\geq 1$.)
 By iteration over $n$ we obtain \eqref{eqn: explosion of euler scheme}. 

We are only left to show that there exists $a_0\geq 1$ such that \eqref{use a_0} holds. To this end, note that the derivative of the function $s$ is given by 
 $$
 s'(x)=\frac{3}{2}C_0\tau^2 x^{1/2}-(2C_0\tau+(1-C_0)).
 $$
 Thus, its minimum is attained at
 $$
 x^*=\bigg(\frac{4C_0\tau+2(1-C_0)}{3C_0\tau^2}\bigg)^2.
 $$
 When $x>x^*$, the function $s$ is increasing and $\lim\limits_{x\to\infty}s(x)=+\infty$. So, we can always find a $x_0$ (depending on $\tau$ and $h$) such that $s(x) \geq 0$
 when $x\geq \max\{1,x_0\}$. Setting $a_0:= \max\{1,x_0\}$ concludes the proof.  
  
 \end{proof}

\subsection{Blow up via the Mean} 
\label{s:periodicblowup}
We also simulate \eqref{SPDE Mattingly} on the one dimensional  torus $\mathbb T = [0,2\pi]$ with periodic boundary conditions.  As our simulations  in Section \ref{sec: numerics} will show, SIE has the potential for blowup when the data is sufficiently large.  Here, we note a potential blowup mechanism consistent with our simulations.

Let us decompose our solution as
\begin{equation}
    u^n  = \bar{u}^n + w^n, \quad \tfrac{1}{2\pi}\smallint w^n(x)dx =\langle w^n\rangle = 0,
\end{equation}
the mean and its fluctuations.  Substituting back into \eqref{scheme:explicit Euler}, and projecting onto the mean and the orthogonal complement:
\begin{subequations}
\begin{align}    
    \bar{u}^{n+1} &= \bar{u}^n - \tau ((\bar{u}^n)^3 + 3 \bar{u}^n \langle(w^n)^2\rangle  + \langle (w^n)^3\rangle )  + \sqrt{2}\delta \bar{W}^n\\
    \begin{split}
    (I - \tau \Delta)w^{n+1} &= w^n - \tau (3 (\bar{u}^n)^2 w^n + 3 \bar{u}^n ((w^n)^2-\langle (w^n)^2\rangle ) + ((w^n)^3-\langle (w^n)^3\rangle ) )\\
    &\quad + \sqrt{2}\delta W^n - \sqrt{2}\delta \bar{W}^n
    \end{split}
\end{align}
\end{subequations}
Thus, if $w^0$ is zero, we would anticipate it will remain small for some small number of iterates, and, to leading order, initially,
\begin{equation}
    \bar{u}^{n+1} \approx \bar{u}^n - \tau(\bar{u}^n)^3 + \sqrt{2}\delta \bar{W}^n,
\end{equation}
which exactly of type \eqref{SDE Mattingly}.  Thus, we anticipate that when $u_0(x) = \bar{u}^0$ is a constant, and  $\E[|\bar{u}^0|^2]\gtrsim \tau^{-1}$, blowup may occur.  {The divergence via this mode zero growth mechanism has been studied in detail in \cite{beccari_strong_2019}.}

\section{Numerical Experiments}\label{sec: numerics}

In this section we present numerical experiments to confirm our theoretical results and to provide a comparison amongst our proposed methods. One of the take-home messages from this section is that,   balancing concerns for empirical robustness and efficiency, we recommend use of the truncated pointwise taming scheme, \eqref{scheme:taming by abs val f' pointwise standard form}. Before proceeding we recall that all of our methods are summarized in Table \ref{tab:scheme}. 


For our experiments, we consider the following  SPDE of type \eqref{eqn:SPDEintro}, 
\begin{equation}
\label{e:testspde1}
    du = (- u^3  + \Delta u) dt  + dW(t).
\end{equation}
{We consider this equation in two settings, namely on the spatial domain $[0,1]$ equipped with Dirichlet boundary conditions and then on $[0,2\pi]$ with periodic boundary conditions. In the first case we use FEM space-discretization, in the second we use spectral Galerkin. In each case we show the effect of the different time-discretizations presented in Section \ref{sec: numerical schemes}.    
}



%



{For computational convenience in the Dirichlet FEM case, in our simulations, local nonlinearities are handled by interpolation through the nodal values:
\begin{equation}
\label{e:interp}
f(u_h)(x)\approx \sum_j f(u_h(x_j)) \varphi_j(x) = I_hf(u_h)(x),
\end{equation}
where $I_h$ denotes the interpolation operator.  This is a common choice, as the interpolation error for P1 elements is the same order as the discretization error, $\mathrm{O}(h^2)$; see \cite{Brenner,thomee2007galerkin,larsson_partial_2009}.
}

As \eqref{e:testspde1} has a cubic nonlinearity, for the spectral Galerkin discretization, we evaluate the nonlinearities in real space, and transform back.  To address the aliasing error, we  simulate with twice as many Fourier modes as desired, a standard de-aliasing strategy, \cite{canuto_spectral_2006,boyd_chebyshev_2001}.  Thus the projection of the nonlinear term is analytically handled for schemes \eqref{scheme:im} (FIE), \eqref{scheme:tame by norm of f'(u)} (global taming), and \eqref{subsec:taming by grad fu} (global gradient taming), along with  
 \eqref{scheme: general SIE}  (SIE). For the other schemes (\eqref{scheme:taming by abs val f' pointwise standard form}, \eqref{scheme:taming2derivativeabsolute value f'}, \eqref{scheme:taming2derivativeabsolute value}, \eqref{scheme: GTEM scheme f}), where the nonlinearity is a rational function at any $\tau>0$, there is an additional approximation, though for sufficiently large $N$, this will be manageable.  We also check our results across different $N$ for consistency.  When reporting that our spectral Galerkin simulations are with a particular value of $N$, the number of fully resolved modes is $(N \div 2)-1$ as we also zero out the Nyquist mode.  For consistency of notation, $P_h$ denotes projection onto this set of lowest modes, where, again, $N = 2\pi/h$ will be an integer.  
 

 Strictly speaking our results of Section \ref{sec: numerical schemes} do not apply to this spectral Galerkin setting (as all such results are state for FEM space-discretization, and Dirichlet boundary conditions). In particular, the Poincaré inequality,  which is repeatedly used in establishing the theorems of Section \ref{sec: numerical schemes},  {does not hold for spectral-Galerkin schemes.}  Nevertheless, as PDE models are frequently studied on the torus, we explore it here. {As we will see, while in most cases our results of Section \ref{sec: numerical schemes} remain robust to the choice of space-descretization,  the spectral Galerkin implementation will prove unreliable in some cases. We conjecture that this could be related  to lack of Poincare' inequality for this space-discretization. } 
 
 


In addition to the numerical methods we propose in the present work, we also compare against SIE, demonstrating the propensity for blowup, and the SPDE analog of Gy\"ongy's method \eqref{SDEclassicaltaming},  namely 
\begin{equation}
\label{e:gyongyspde}
    u_h^{k+1} = u_h^k + \Delta u_h^{k+1}\tau + P_{h}\left(\frac{f(u_h^k)}{1+\tau|f(u_h^k)|}\right)\tau + P_h(\delta W^{k+1}).
\end{equation}

The common structure of all of our experiments is to integrate the approximation of \eqref{e:testspde1} and compare the time series of observables, notably the first  moment of the $L^2$ norm of the approximation, $\E\|u_h^k\|\approx \E\|u(t_k)\|$. 


In looking at our results, the reader should look favorably on methods which accurately capture the evolution of the observable {\it throughout} the time of integration; indeed, that is the whole purpose of a naUiT scheme.  The fidelity to truth is judged by comparing simulations with different values of $\tau$ against one another.  

\subsection{Additional Details of Experiments}
We report results with $\tau = 0.1, 0.01, 0.001$.  All experiments are simulated with $10^4$ independent trials and, where they appear, shaded regions reflect the one standard deviation from the ensemble.



Our FEM results are performed with $N=100$, as in \eqref{e:femmesh} and \eqref{e:hatfunction} on $[0,1]$, such that we have $N-1=99$ coefficients.  The covariance of the noise process is taken to be  $Q = (-\Delta)^{-1}$, where the Laplacian is equipped with Dirichlet boundary conditions, making $Q$ trace class.  This is simulated within the FEM environment by initially computing the Cholesky factorization of the stiffness matrix, ${\bf K} = {\bf L}{\bf L}^T$, and solving $\mathbf{L}^T {\mathbf{w}} = {\bf z}$, with ${\bf z} \sim N(0,I_{N-1})$.  Letting $\mathbf{M}$ denote the mass matrix, the update schemes then take the form
\begin{equation*}
({\bf M} + \tau {\bf K}){\bf u}_h^{k+1} = {\bf M} {\bf u}_h^{k} + {\bf M} \cdot  \text{Drift Term} +\sqrt{\tau} {\bf M} {\bf w}^{k+1}
\end{equation*}

Our main results on spectral Galerkin are performed with $N = 256$, corresponding to 127 resolved modes.   For this case, the noise process has covariance $Q = (I - \Delta)^{-1}$ with periodic boundary conditions.  This is sampled in Fourier space where it diagonalizes with variance $(1 + j^2)^{-1}$ in component $j$.  On the Fourier side, the update scheme takes the form
\begin{equation*}
(1+j^2)\hat{u}^{k+1}_{h}(j) = \hat{u}^{k}_{h}(j)  + \mathcal{F}[\text{Drift Term}](j) + \sqrt{\tau} (1+j^2)^{-1/2} \hat{z}^{k+1}(j),
\end{equation*}
where  $\mathcal{F}$ denotes the Fourier transform.  The $\hat{z}^{k+1}(j)$ are such that this is a real valued periodic function; with $(\hat{z}^{k+1}(j))^\ast = \hat{z}^{k+1}(-j)\sim N_{\mathbb{C}}(0,1)$  for $j \neq 0$, and $\hat{z}^{k+1}(0)\sim N_{\mathbb{R}}(0,1)$

\subsection{Experimental Results}

We first report our results for the FEM case and then for spectral Galerkin.  We discuss these results afterwards, in Section \ref{sec:numericdisc}.

\subsubsection{FEM Results}
In our first set of experiments, we take $u_0$ to be a constant function, of different sizes, $u_0(x) = 0, 1, 100$.   The results, in increasing order of magnitude of the initial datum,  are shown in Figures \ref{fig:zerodata_fem}, \ref{fig:onedata_fem}, \ref{fig:hundreddata_fem}. 

When $u_0(x) = 0$, Figure \ref{fig:zerodata_fem} shows good agreement across methods, particularly as $\tau$ is reduced.  Arguably, the best naUiT method for this data is the fully implicit Euler method, as there is very little variation in performance between the coarsest and finest values of $\tau$.  Indeed, the fact  that the results, graphically, appear the same at all values of $\tau$, would suggest this was the most computationally efficient, allowing for the largest time step.  But this obfuscates the additional, per iterate, cost of the fully implicit method, as it requires a Newton type solver.  

Next, we look at ``small'' nonzero data, taking $u_0(x) = 1$.   Figure \ref{fig:onedata_fem} shows results that are quite similar to the case of zero initial data.  We have broad consistency across all cases, improvement as $\tau$ is reduced, and the best results amongst the naUiT methods from the fully implicit method.

For ``large'' data, $u_0(x) = 100$, Figure \ref{fig:hundreddata_fem} shows the first ``failure'' of a method. For SIE, \eqref{scheme: general SIE}, all results are absent; this is due to finite time blowup.  Data of size $u_0 = 10$ also revealed blowup for $\tau=0.1$; for brevity, we do not include those results here.


In Figure \ref{fig:osc10_10_fem}, we present the case of $u_0(x) = 10 \sin(10\pi x)$.  Again, the results are consistent across all methods, and we see convergence as $\tau$ is reduced.


\begin{figure}
    \centering

    \subfigure[SIE, \eqref{scheme: general SIE}]{\includegraphics[width=6.5cm]{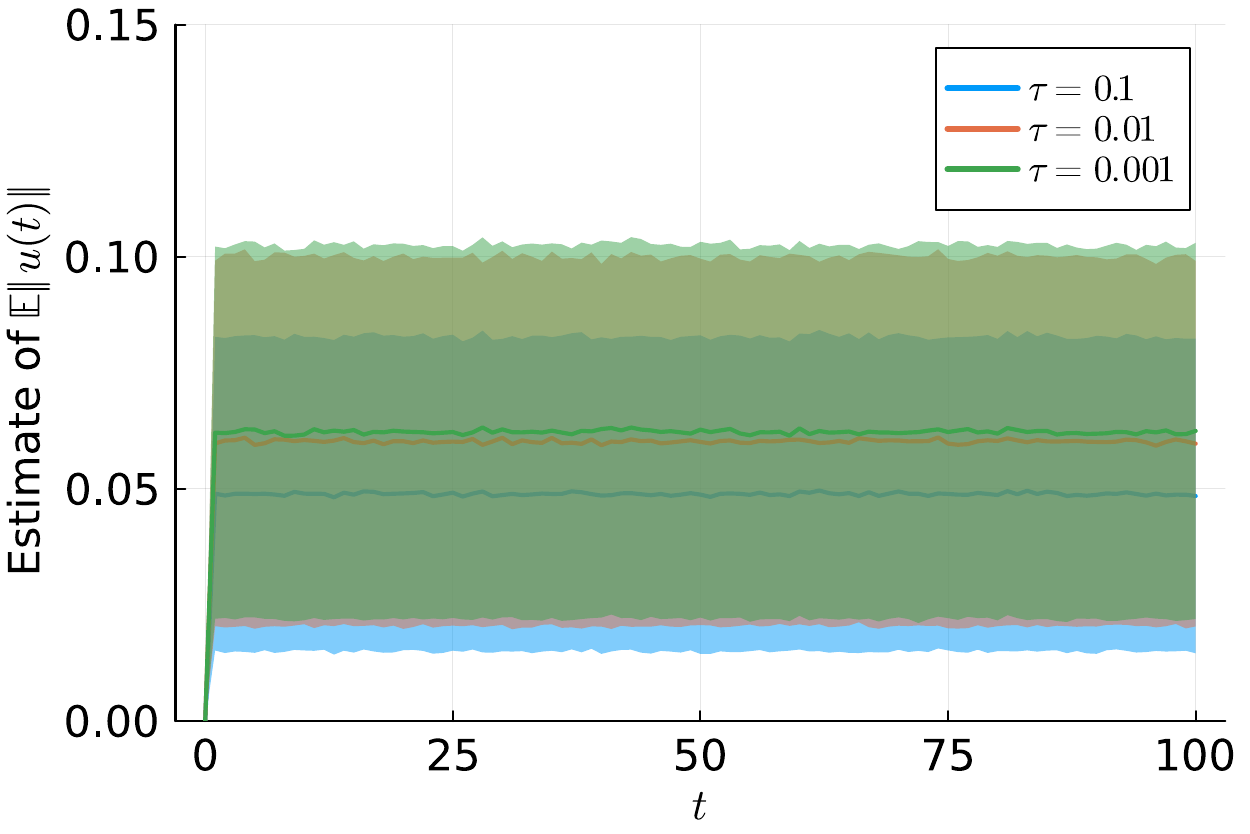}}
    \subfigure[FIE, \eqref{scheme:im}]{\includegraphics[width=6.5cm]{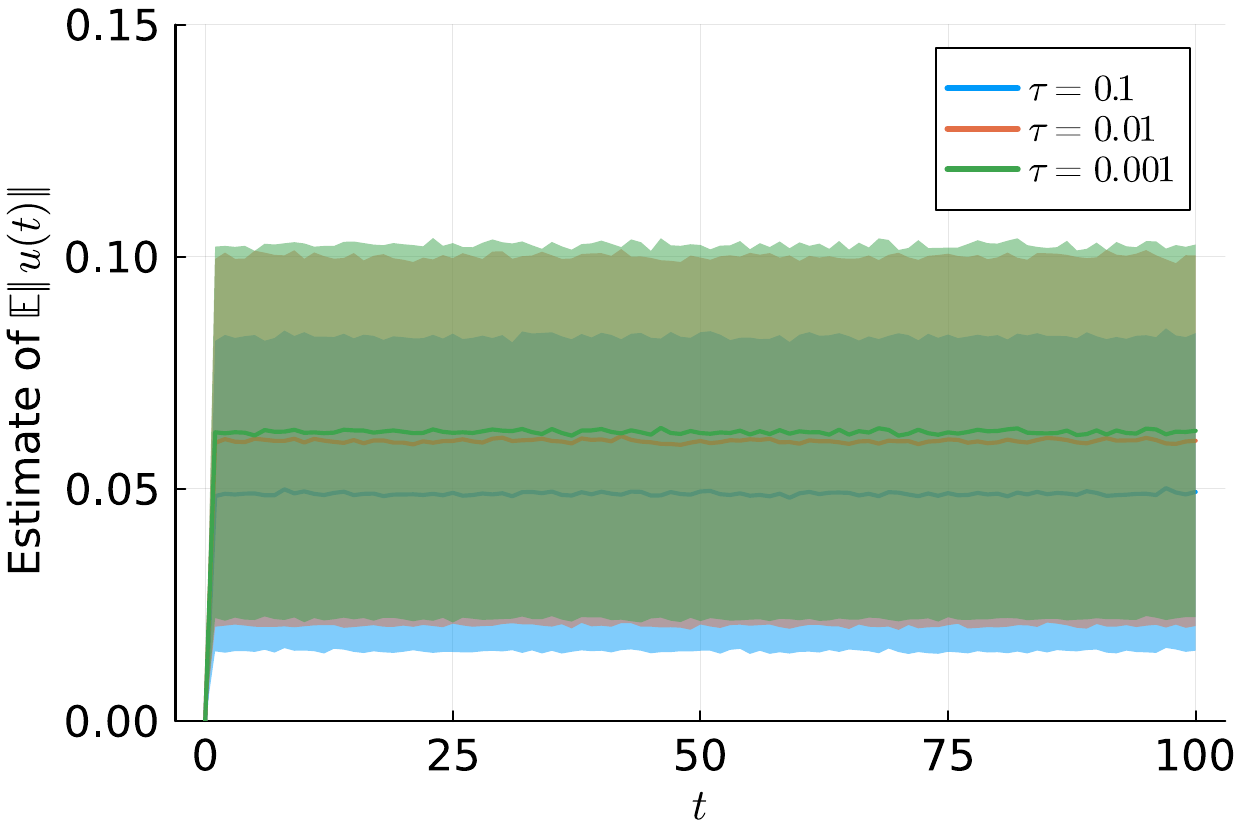}}
    \subfigure[Gy\"ongy's method, \eqref{e:gyongyspde}]{\includegraphics[width=6.5cm]{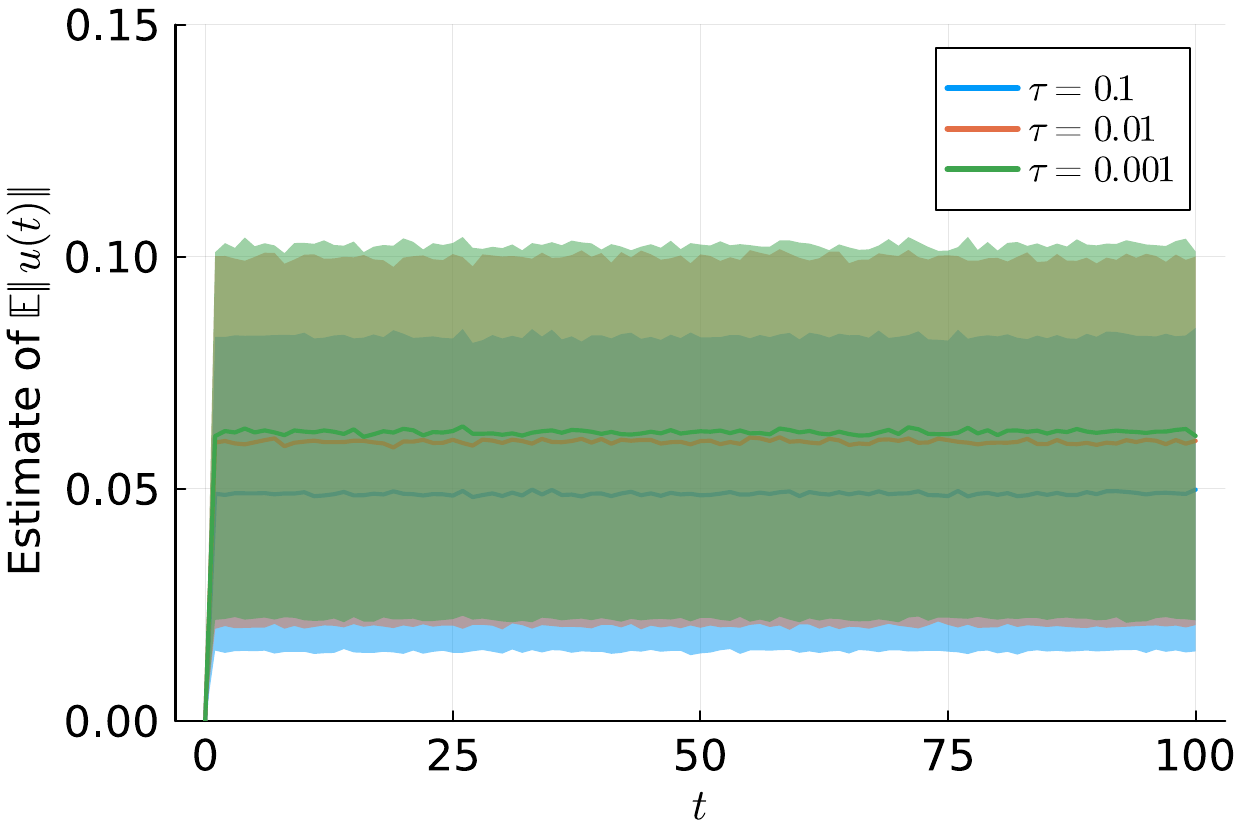}}
    \subfigure[Truncated pointwise taming,  \eqref{scheme:taming by abs val f' pointwise standard form}]{\includegraphics[width=6.5cm]{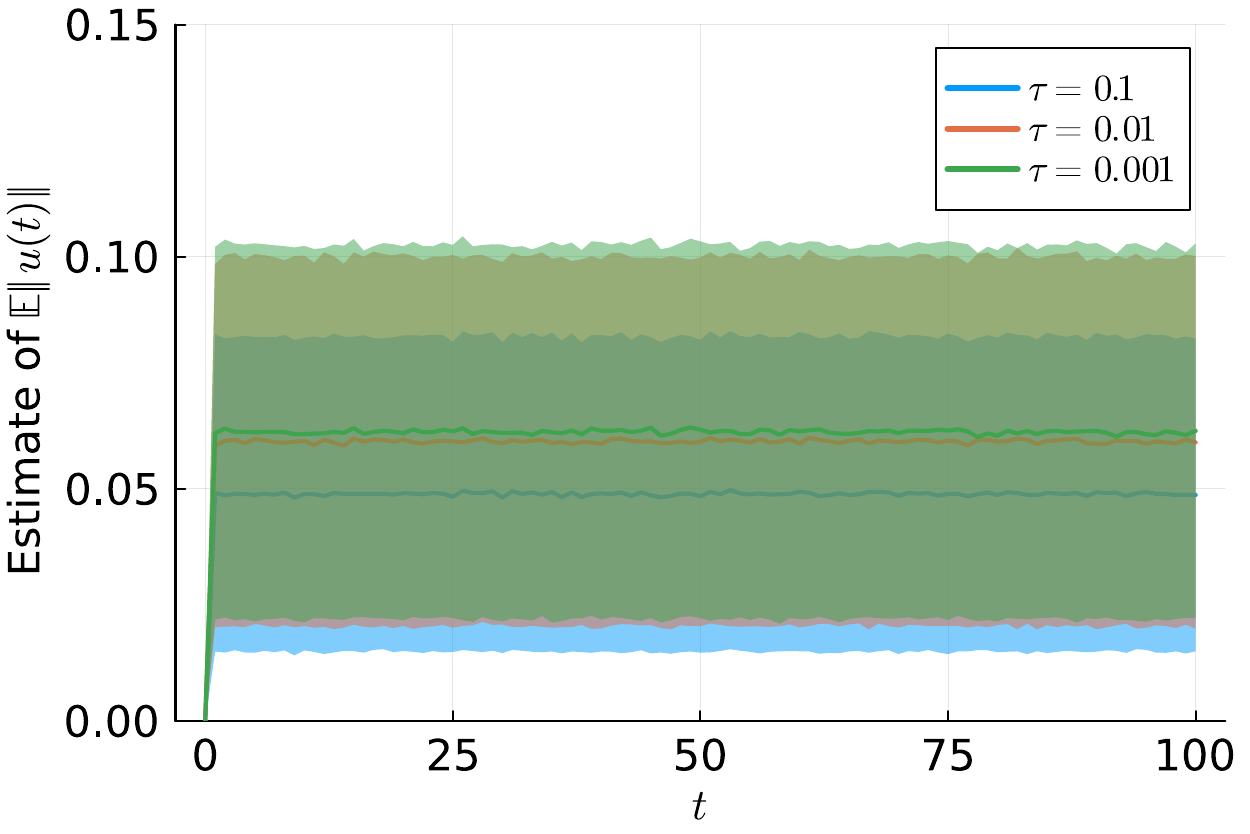}}
    \subfigure[GTEM, \eqref{scheme: GTEM scheme f}]{\includegraphics[width=6.5cm]{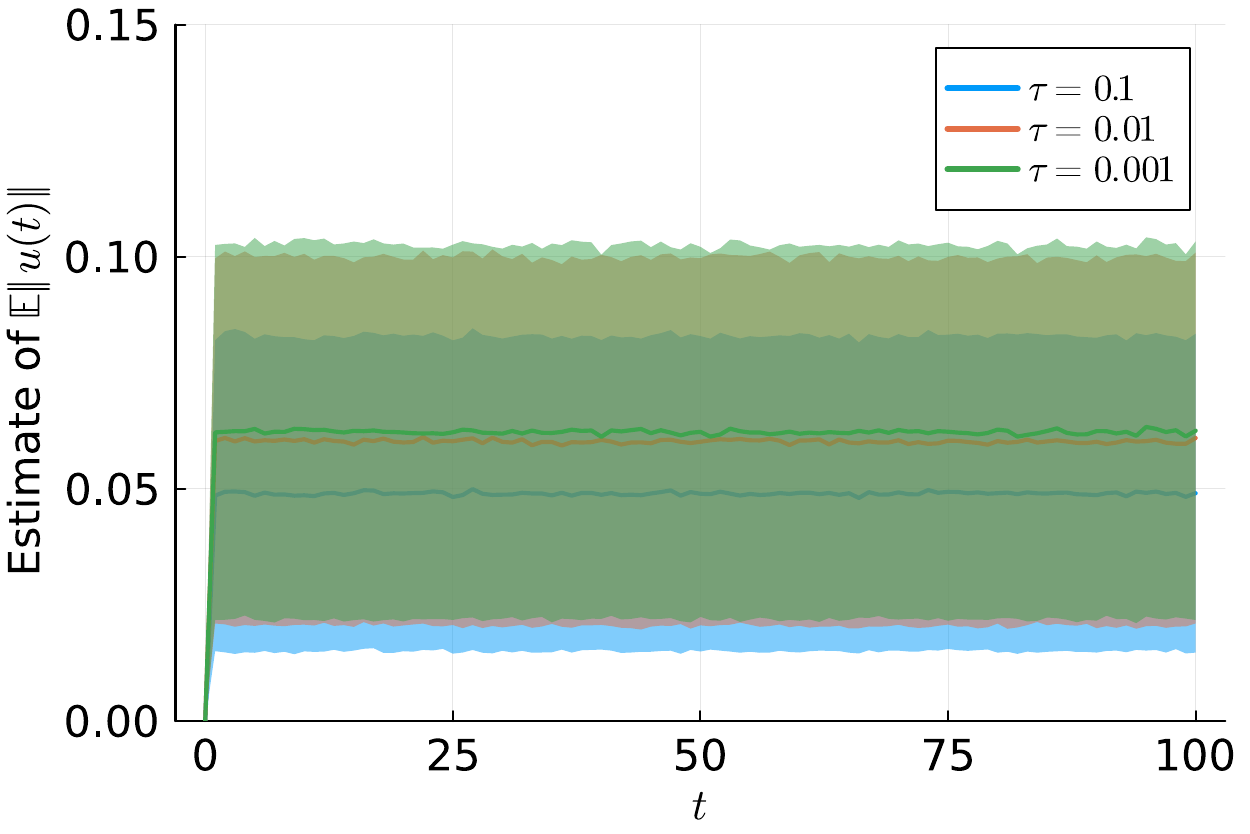}}
\subfigure[Global gradient taming, \eqref{scheme:taming1}]{\includegraphics[width=6.5cm]{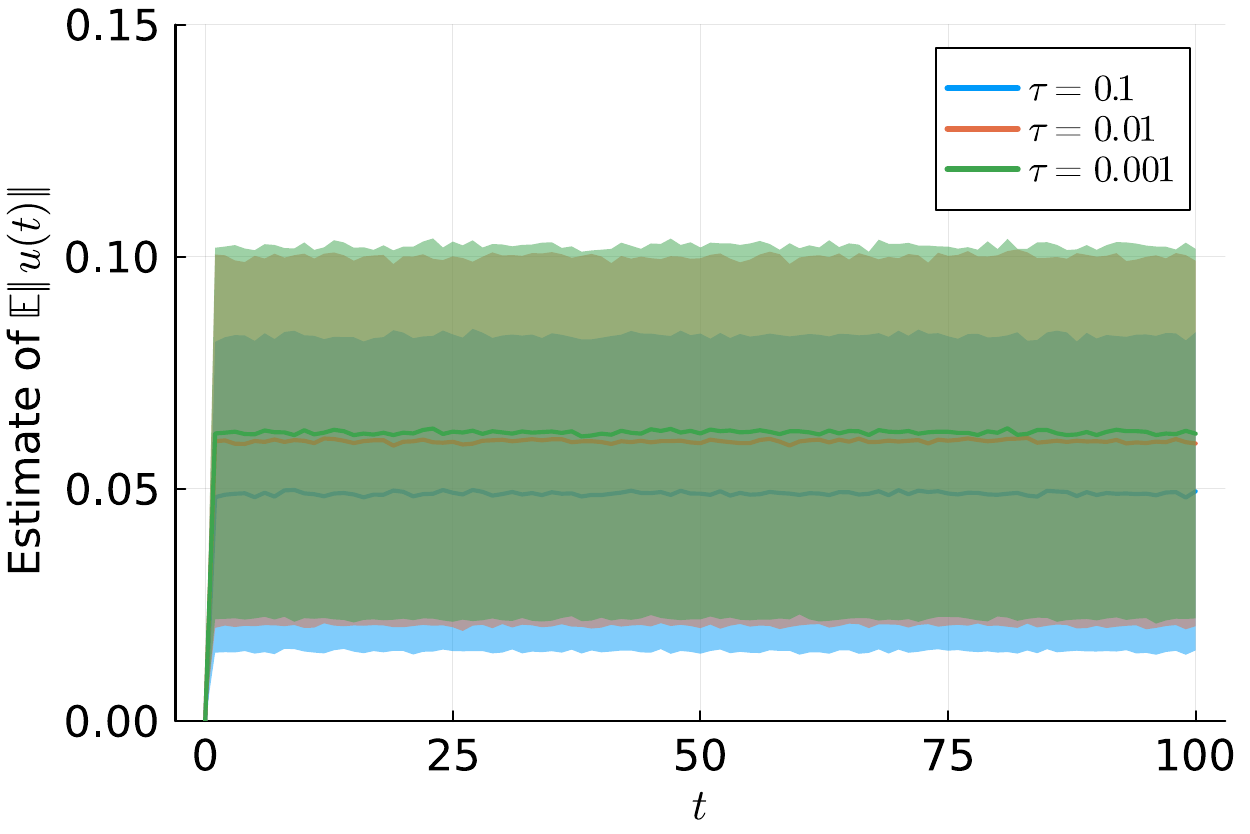}}
    \subfigure[Truncated global taming, \eqref{scheme:tame by norm of f'(u)}]{\includegraphics[width=6.5cm]{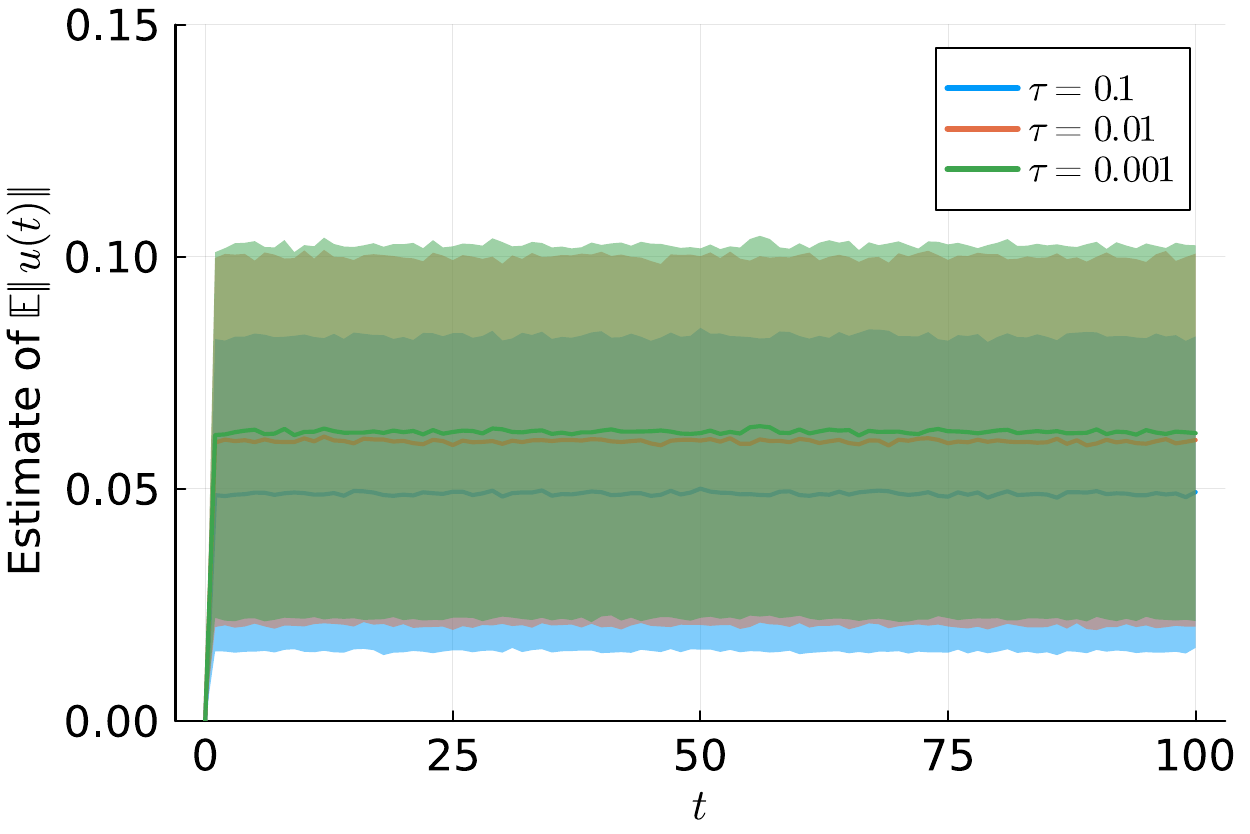}}
    
    \caption{Time series of FEM solutions for the first moment with initial condition $u_0(x) = 0$.  All schemes appear quite comparable.  Shaded regions reflect one standard deviation. }
    \label{fig:zerodata_fem}
\end{figure}

\begin{figure}
    \centering

    \subfigure[SIE, \eqref{scheme: general SIE}]{\includegraphics[width=6.5cm]{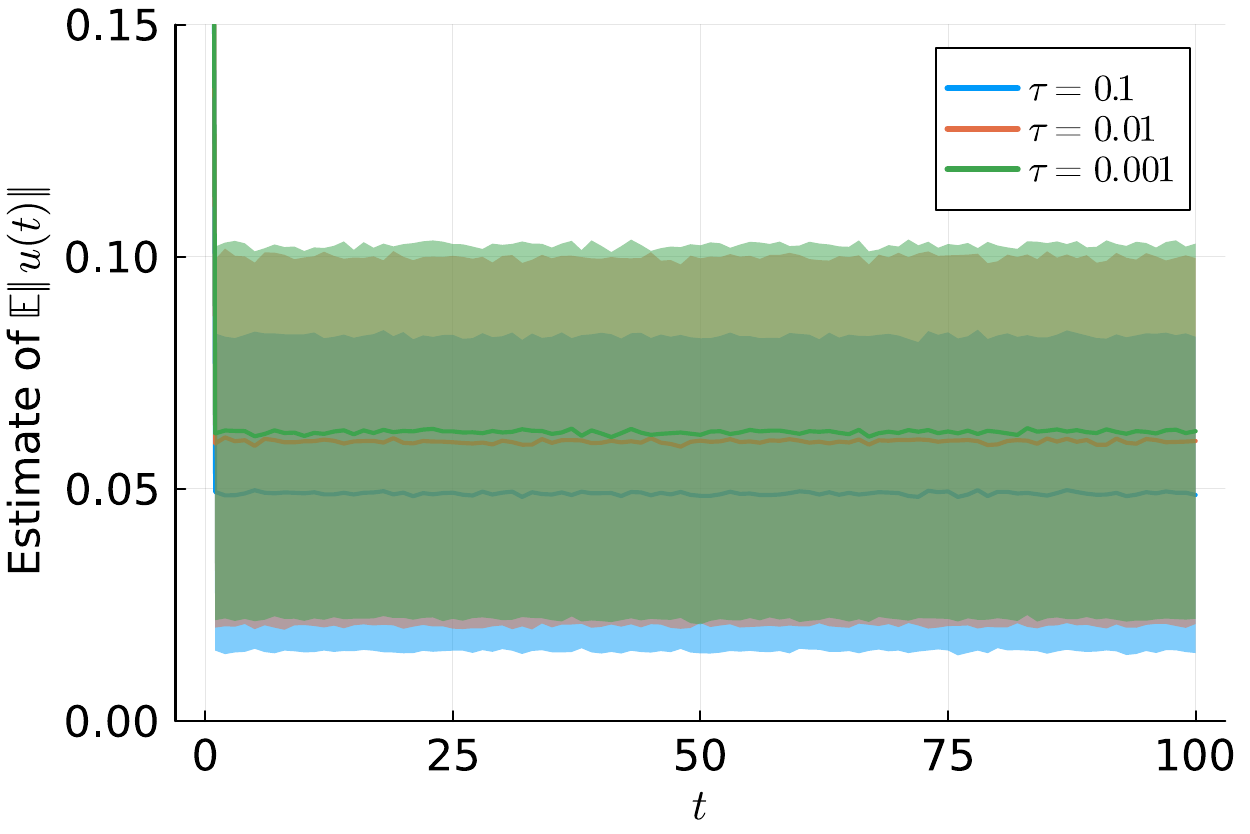}}
    \subfigure[FIE, \eqref{scheme:im}]{\includegraphics[width=6.5cm]{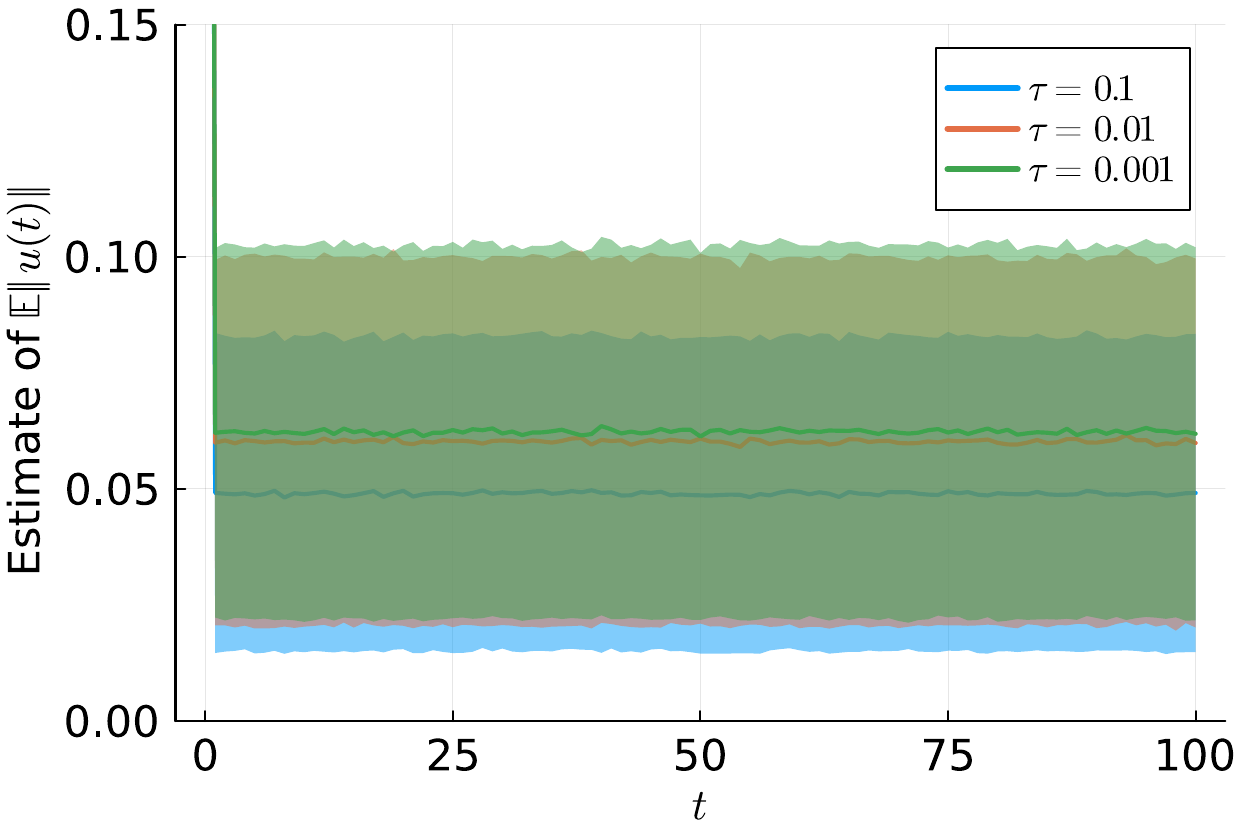}}
    \subfigure[Gy\"ongy's method, \eqref{e:gyongyspde}]{\includegraphics[width=6.5cm]{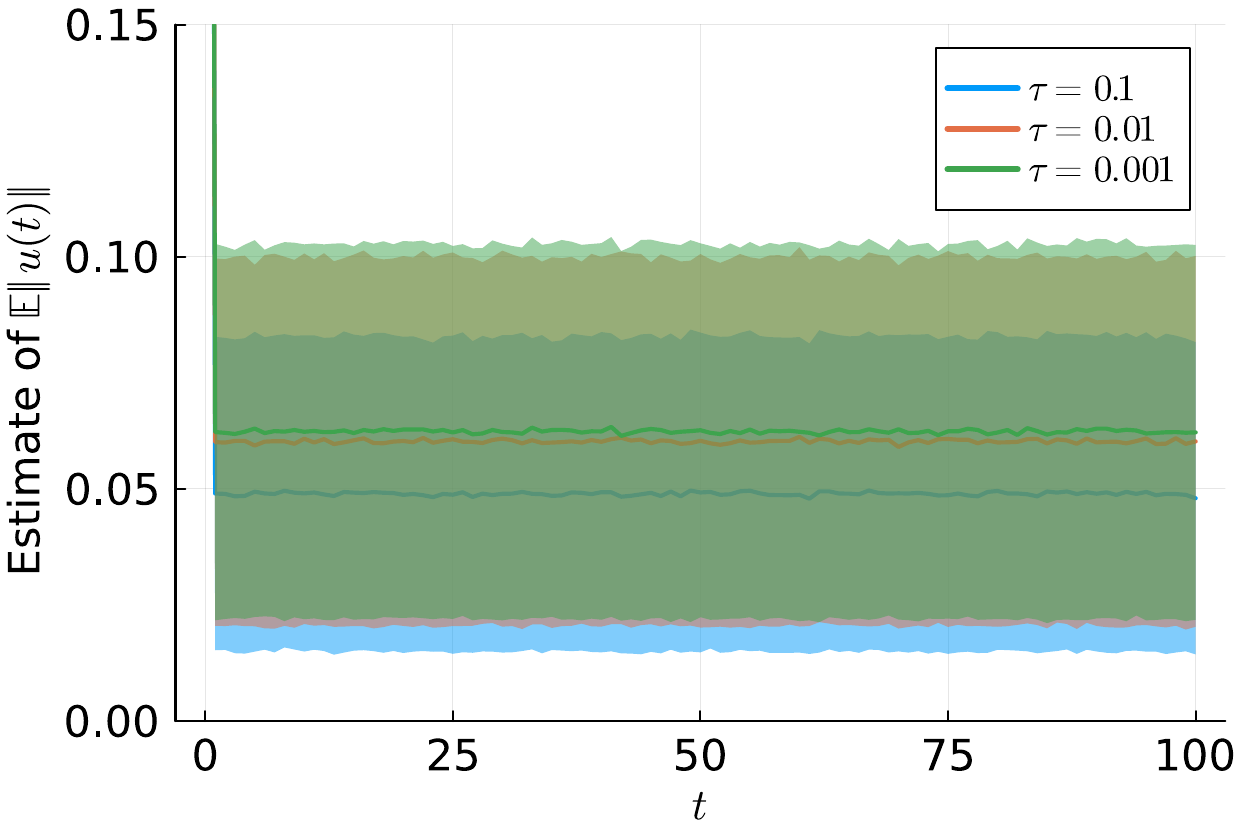}}
    \subfigure[Truncated pointwise taming,  \eqref{scheme:taming by abs val f' pointwise standard form}]{\includegraphics[width=6.5cm]{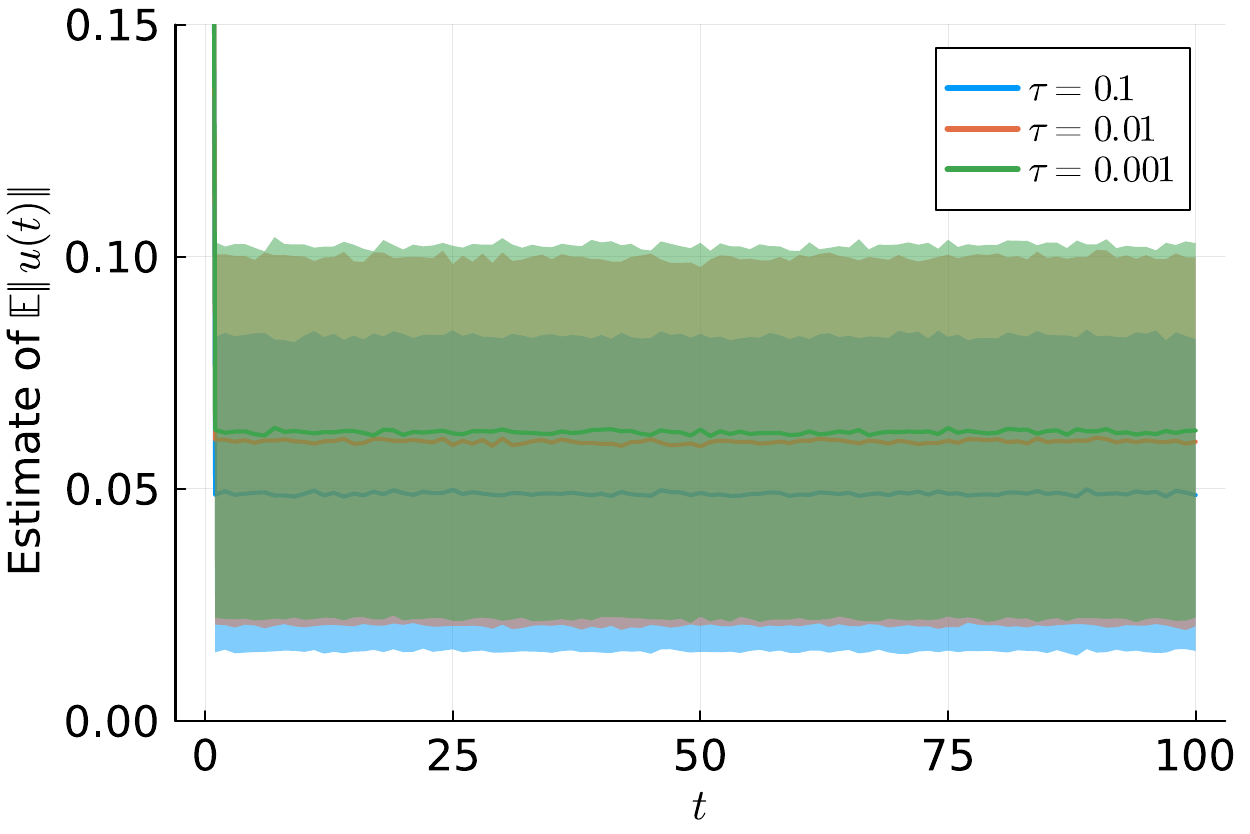}}
    \subfigure[GTEM, \eqref{scheme: GTEM scheme f}]{\includegraphics[width=6.5cm]{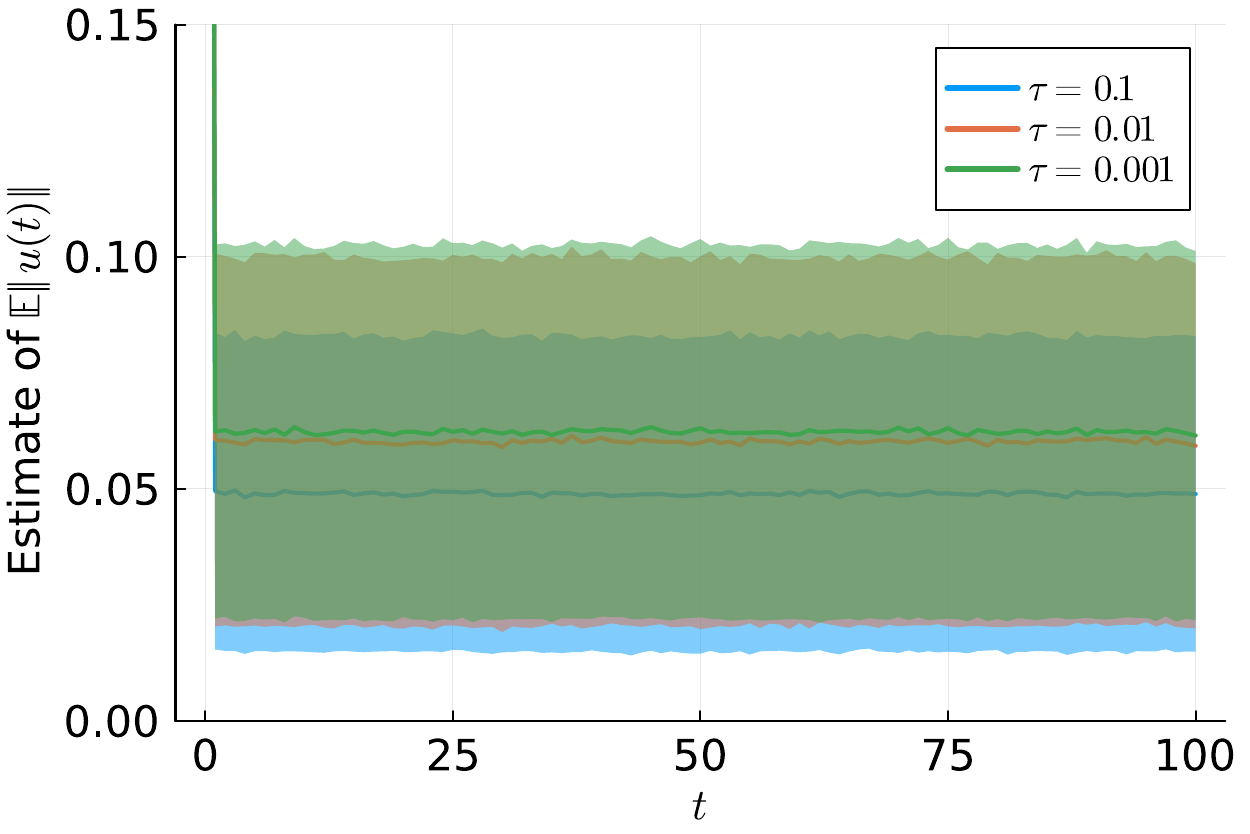}}
    \subfigure[Global gradient taming, \eqref{scheme:taming1}]{\includegraphics[width=6.5cm]{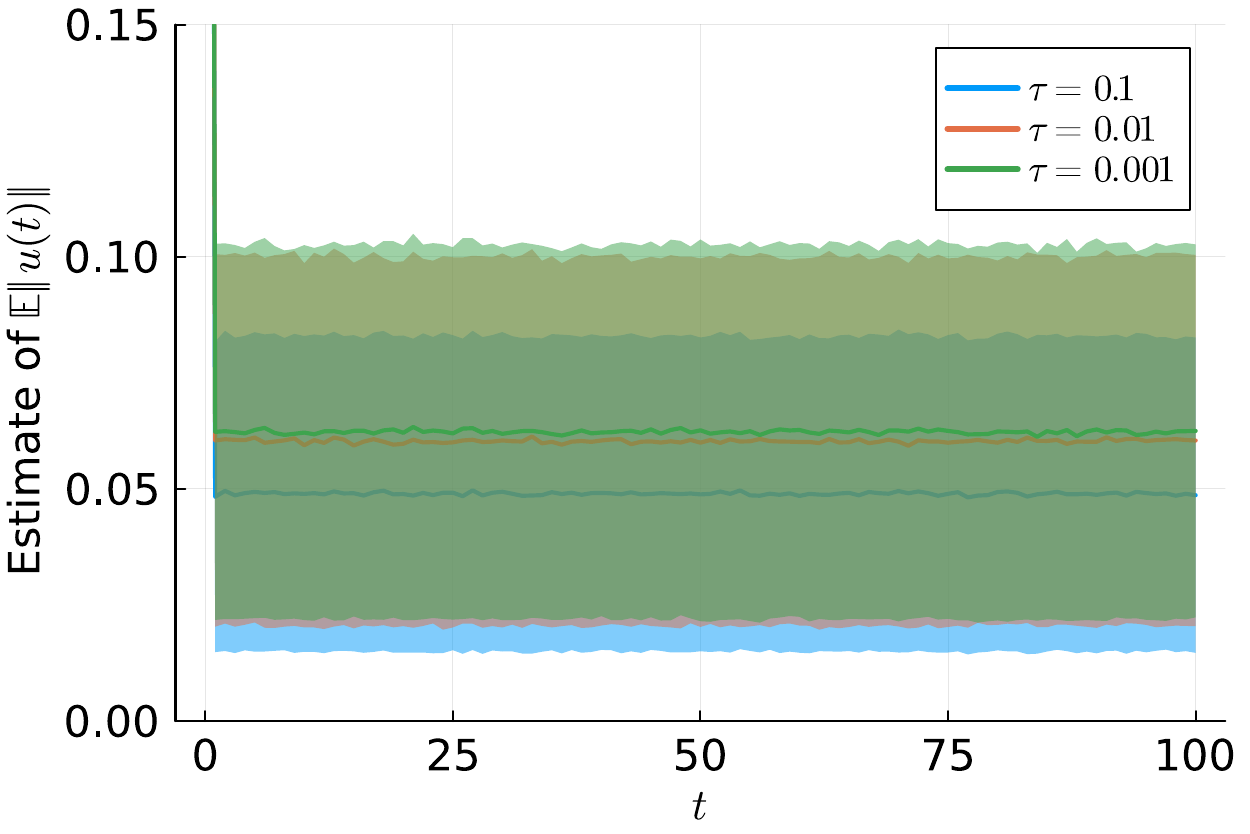}}
    \subfigure[Truncated global taming, \eqref{scheme:tame by norm of f'(u)}]{\includegraphics[width=6.5cm]{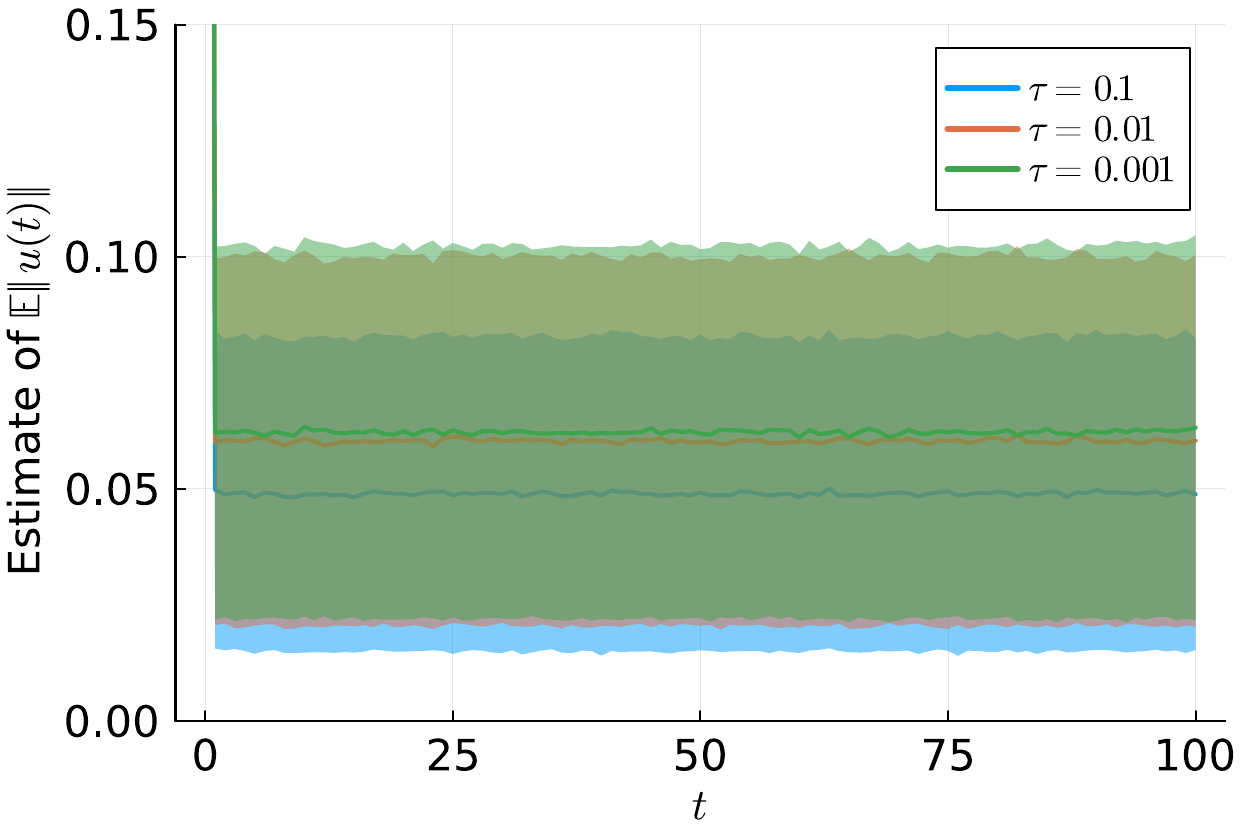}}
    \caption{Time series of FEM solutions for the first moment for initial condition $u_0(x) = 1$.  All schemes appear quite comparable.  Shaded regions reflect one standard deviation.}
    \label{fig:onedata_fem}
\end{figure}



    

    


\begin{figure}
    \centering

    \subfigure[SIE, \eqref{scheme: general SIE}]{\includegraphics[width=6.5cm]{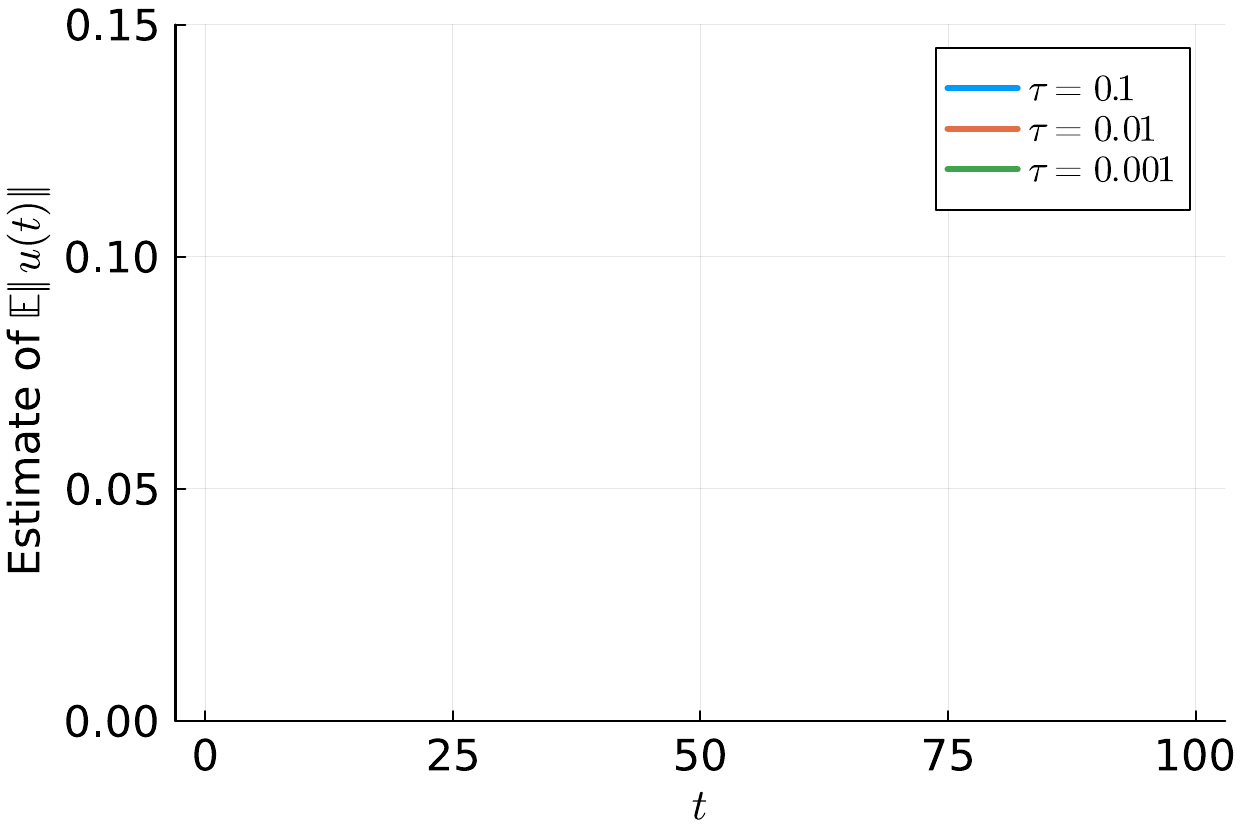}}
    \subfigure[FIE, \eqref{scheme:im}]{\includegraphics[width=6.5cm]{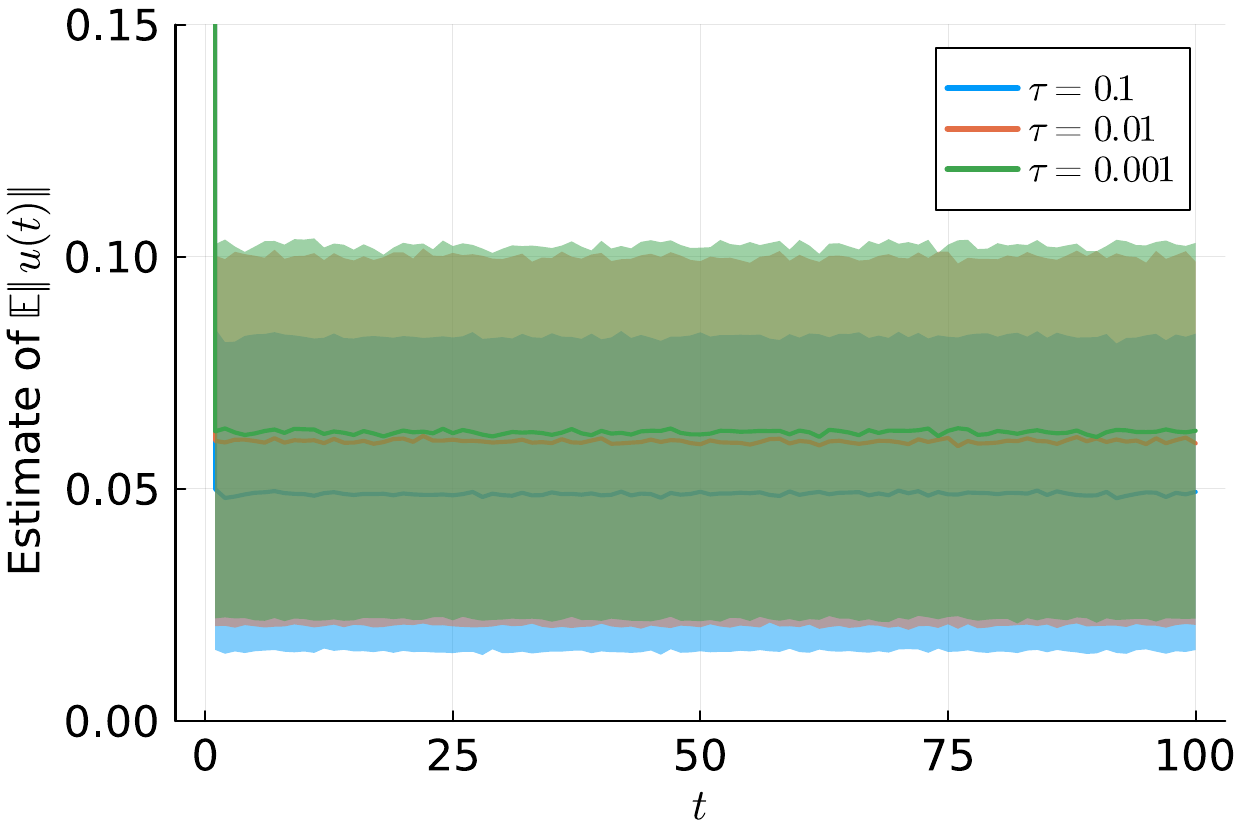}}
    \subfigure[Gy\"ongy's method, \eqref{e:gyongyspde}]{\includegraphics[width=6.5cm]{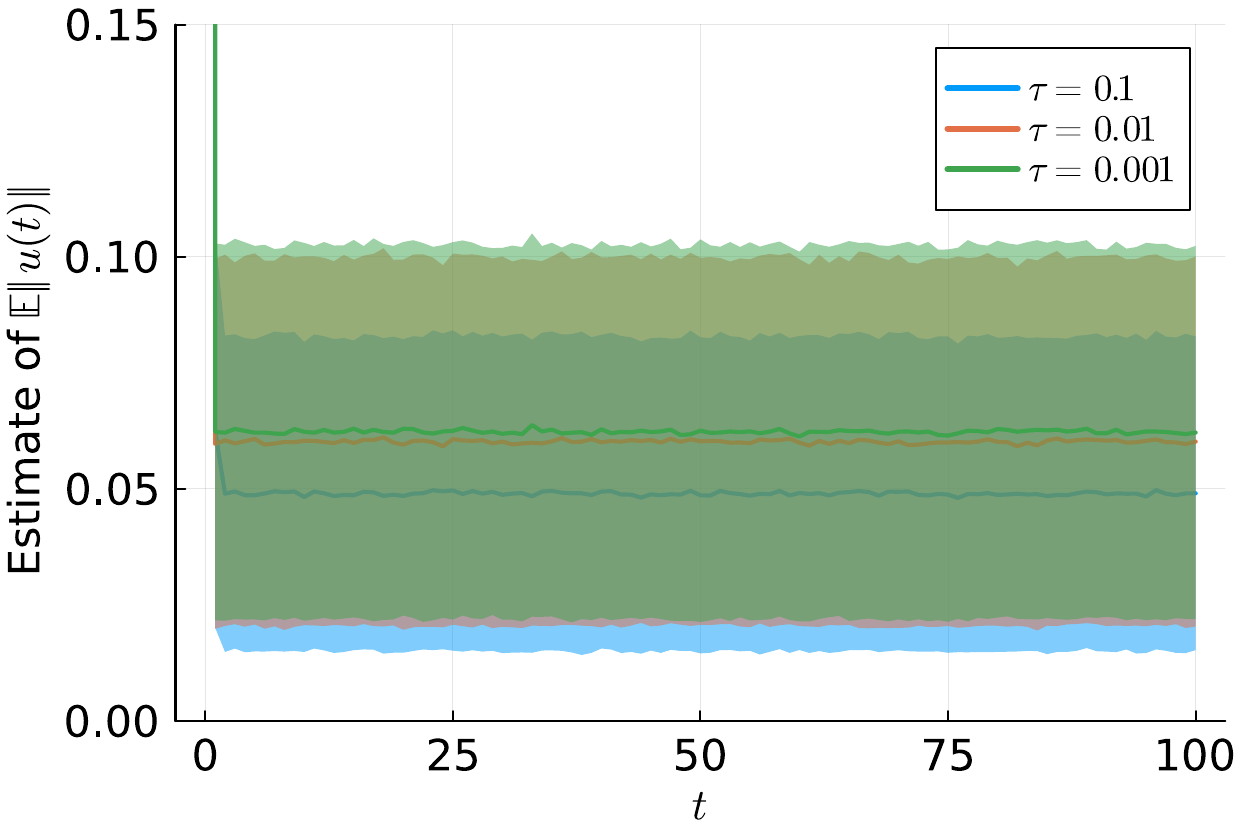}}
    \subfigure[Truncated pointwise taming,  \eqref{scheme:taming by abs val f' pointwise standard form}]{\includegraphics[width=6.5cm]{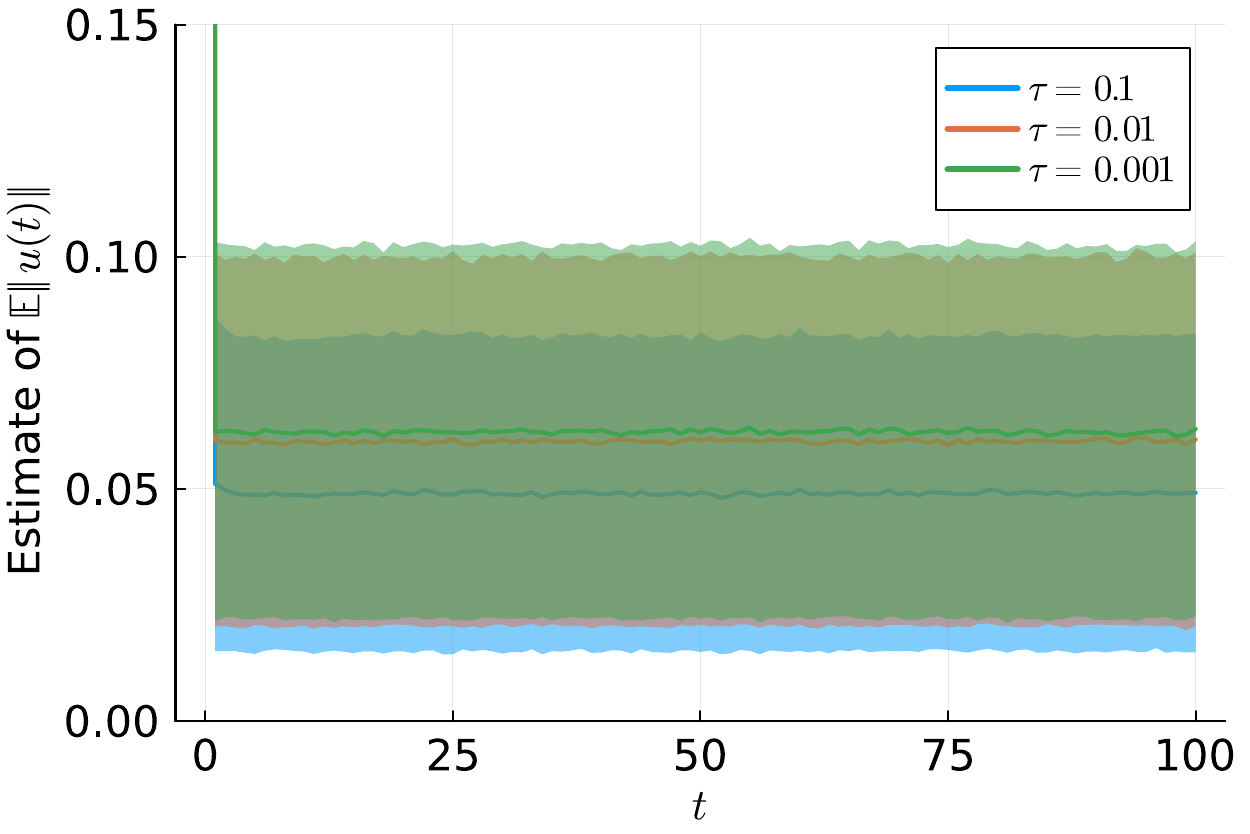}}
    \subfigure[GTEM, \eqref{scheme: GTEM scheme f}]{\includegraphics[width=6.5cm]{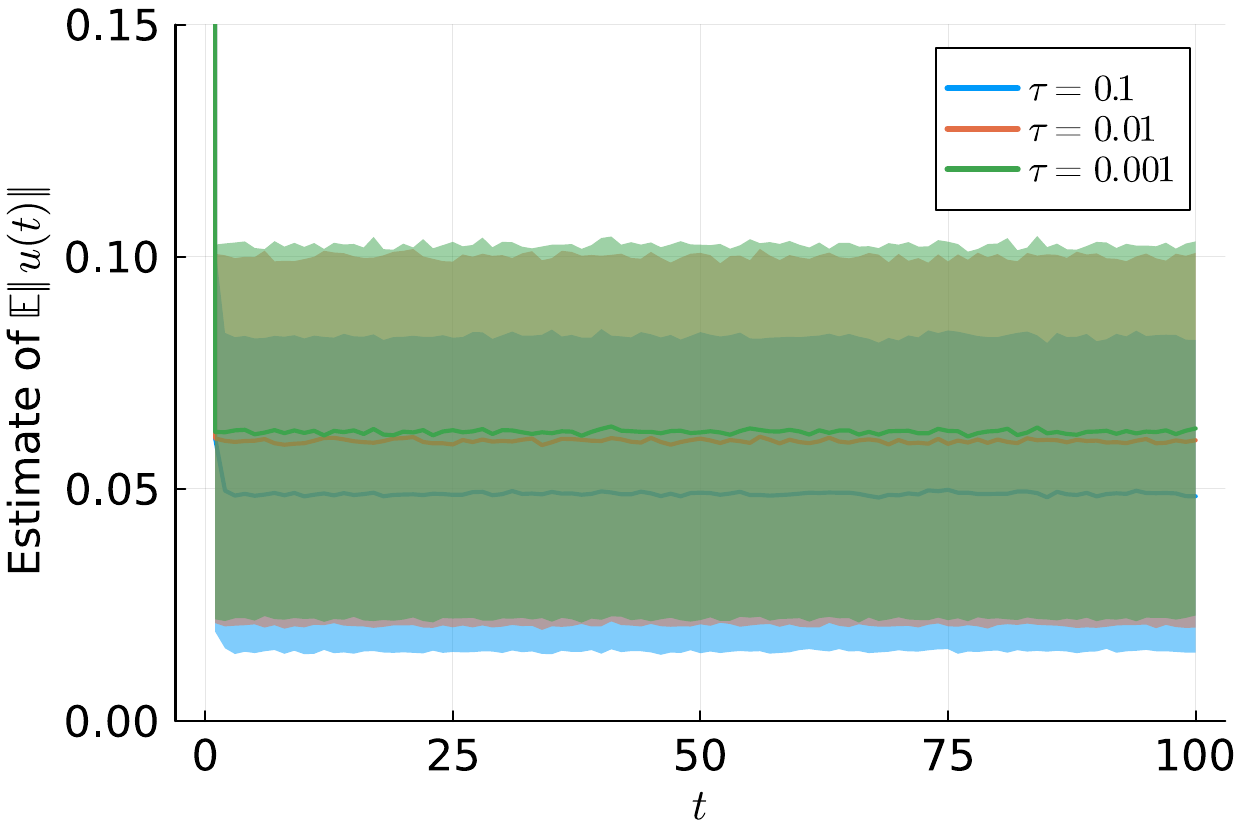}}
    \subfigure[Global gradient taming, \eqref{scheme:taming1}]{\includegraphics[width=6.5cm]{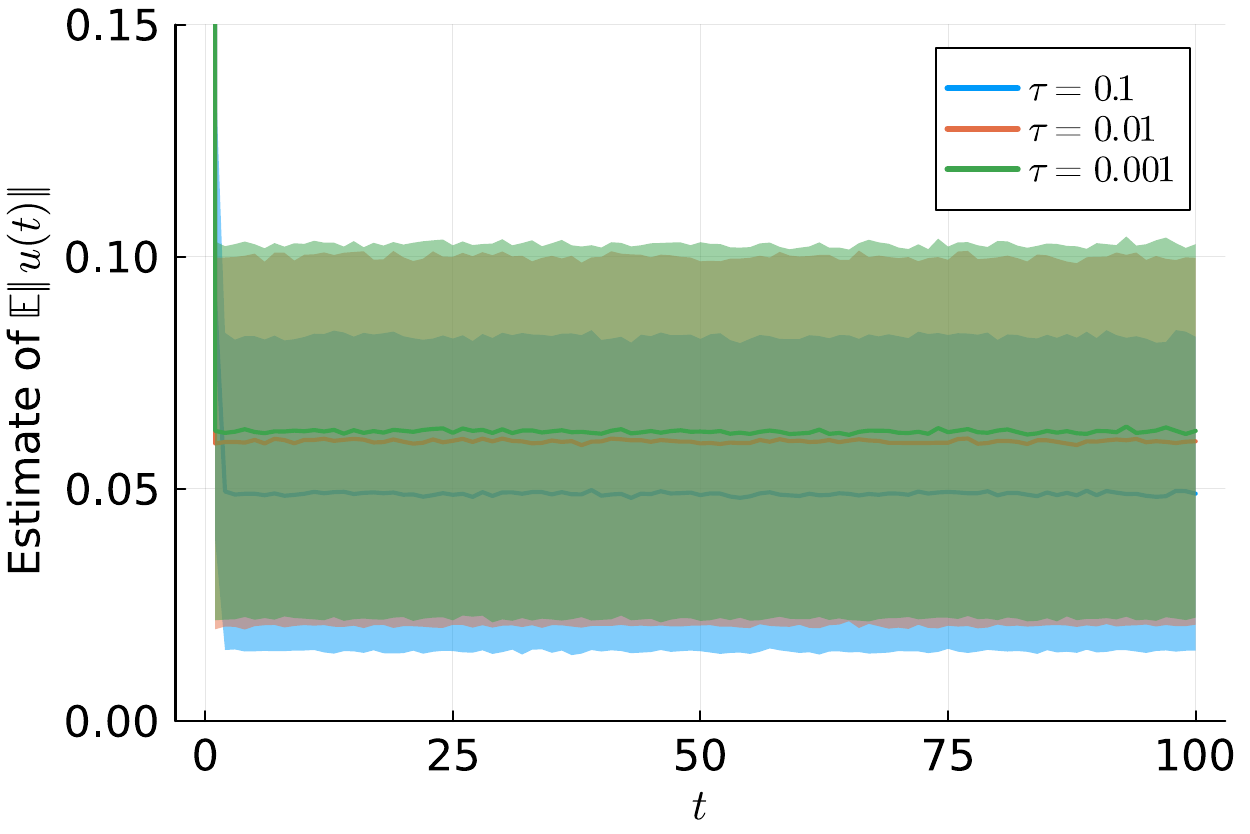}}
    \subfigure[Truncated global taming, \eqref{scheme:tame by norm of f'(u)}]{\includegraphics[width=6.5cm]{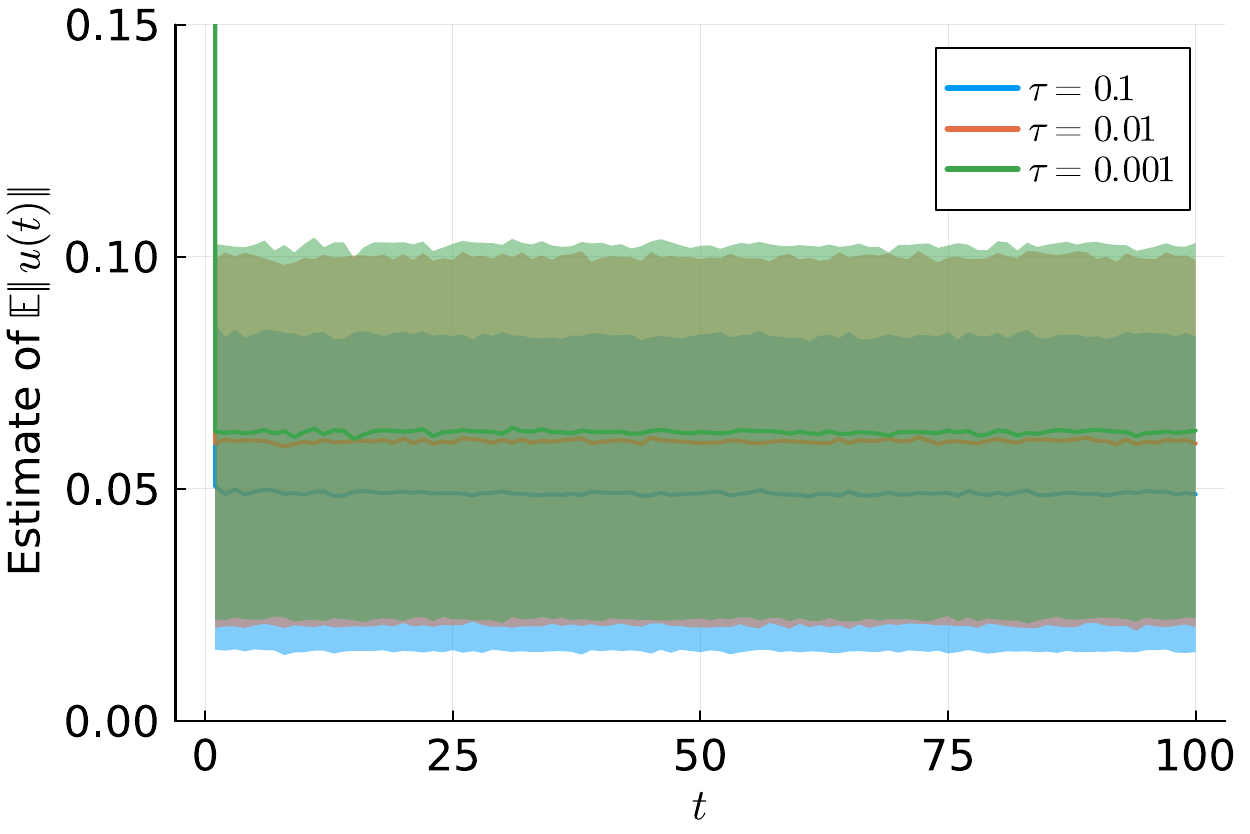}}
    
    \caption{Time series of FEM solutions for the first moment for initial condition $u_0(x) = 100$. For all computed values of $\tau$, SIE experiences blowup.  Otherwise, the schemes appear quite comparable.
    Shaded regions reflect one standard deviation.}
    \label{fig:hundreddata_fem}
\end{figure}


\begin{figure}
    \centering

    \subfigure[SIE, \eqref{scheme: general SIE}]{\includegraphics[width=6.5cm]{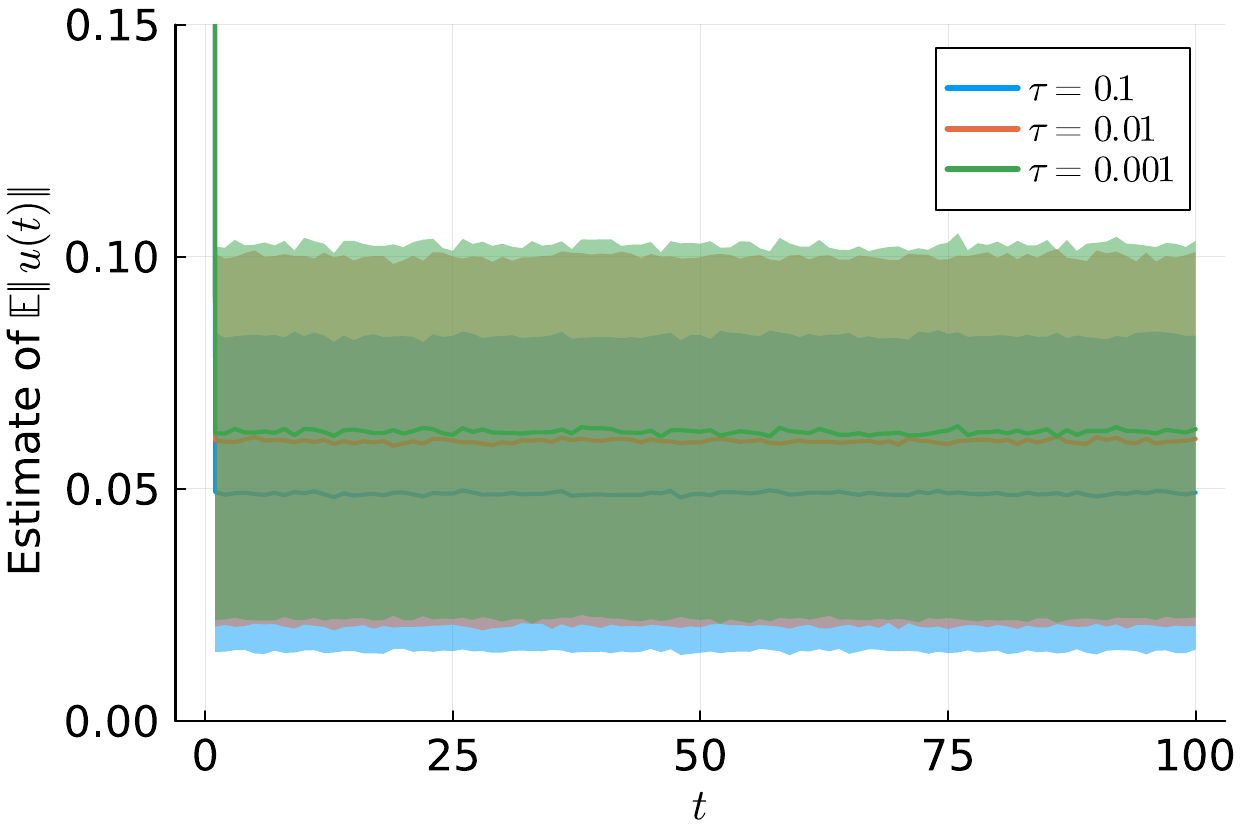}}
    \subfigure[FIE, \eqref{scheme:im}]{\includegraphics[width=6.5cm]{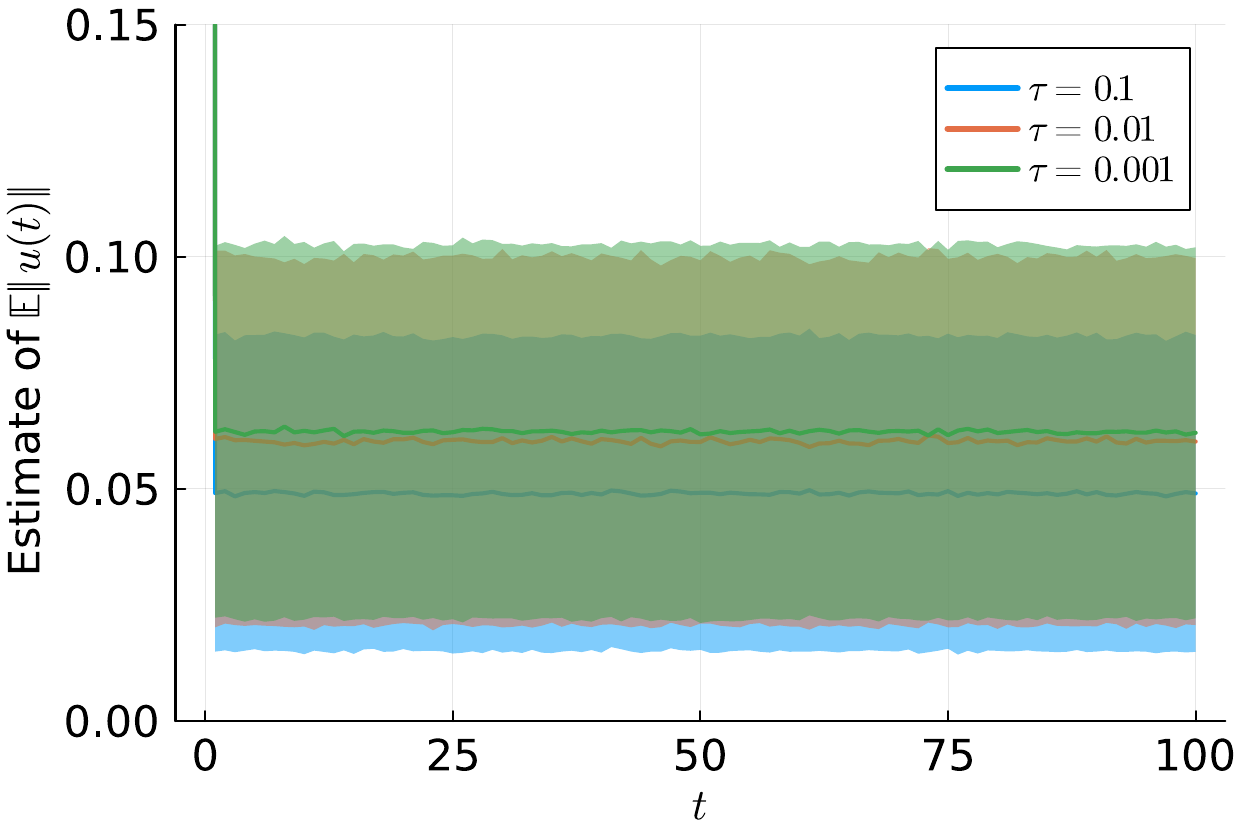}}
    \subfigure[Gy\"ongy's method, \eqref{e:gyongyspde}]{\includegraphics[width=6.5cm]{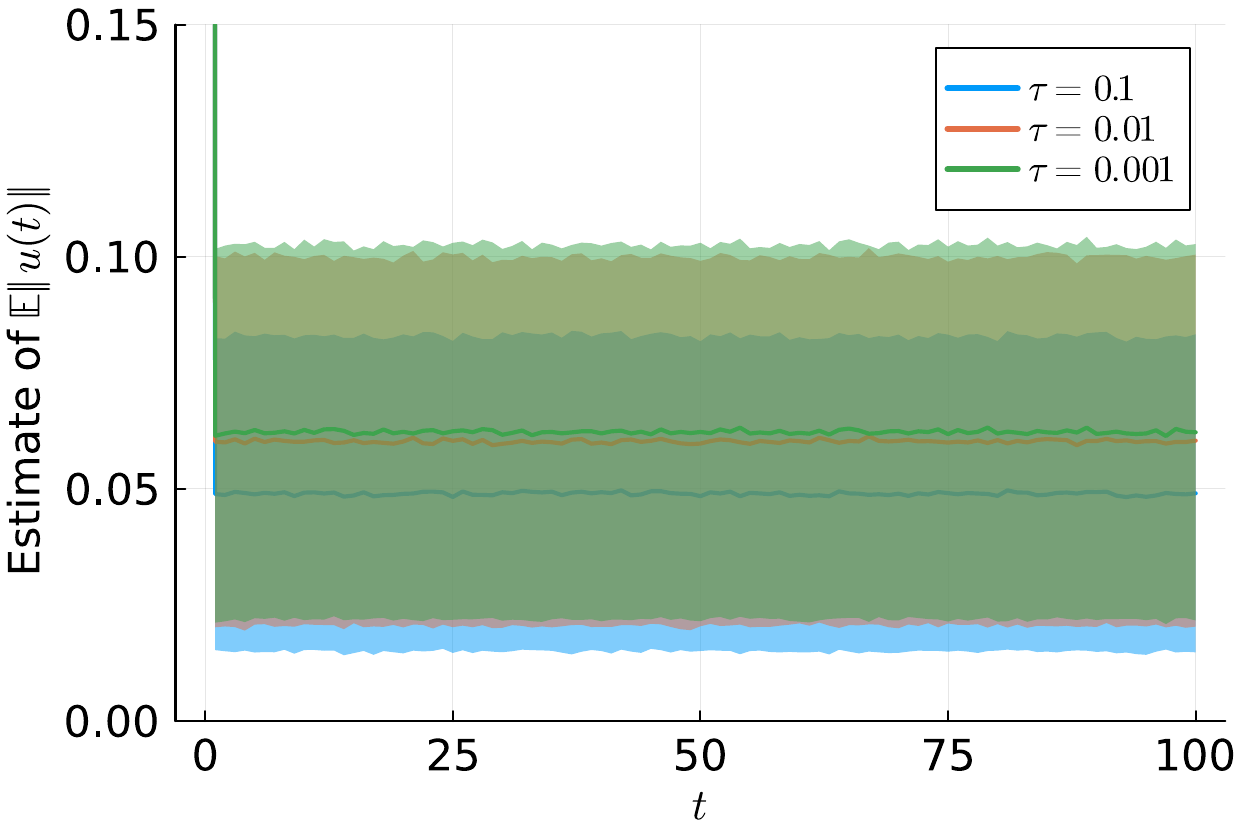}}
    \subfigure[Truncated pointwise taming,  \eqref{scheme:taming by abs val f' pointwise standard form}]{\includegraphics[width=6.5cm]{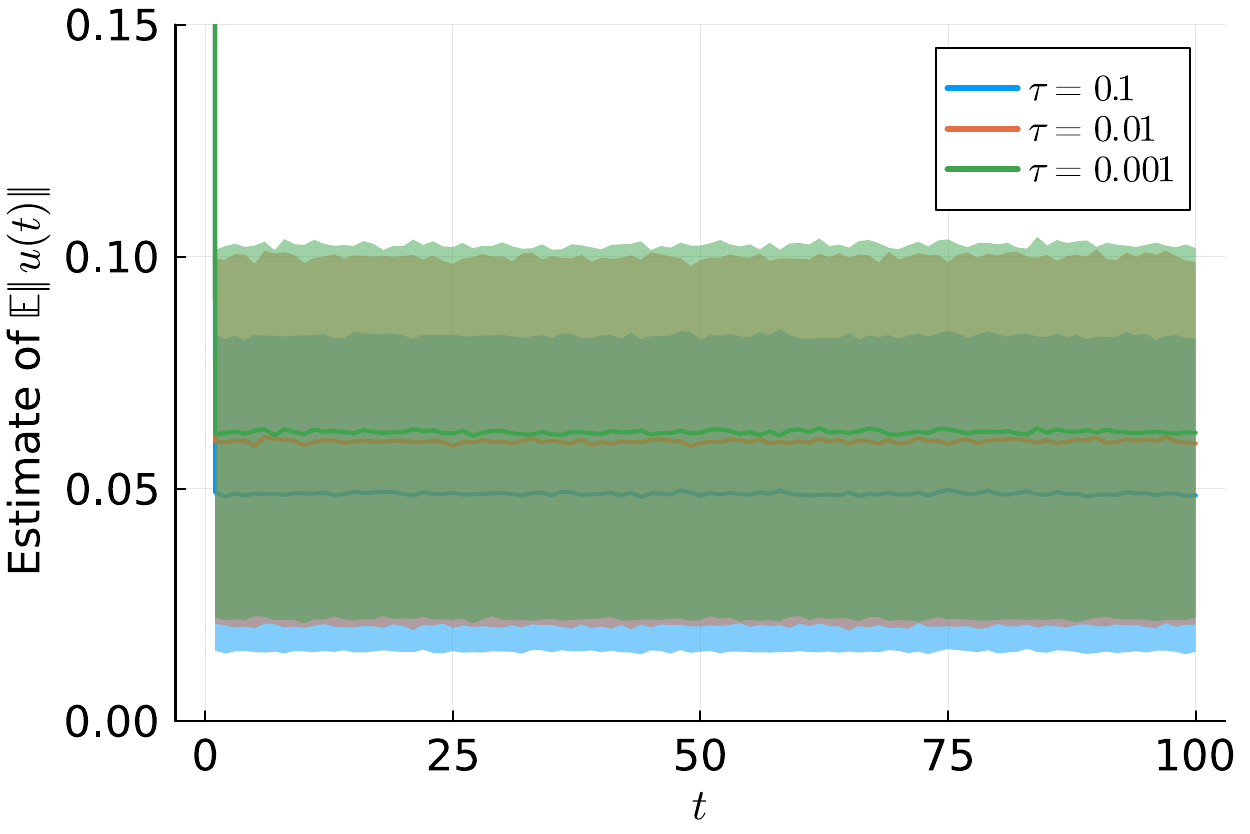}}
    \subfigure[GTEM, \eqref{scheme: GTEM scheme f}]{\includegraphics[width=6.5cm]{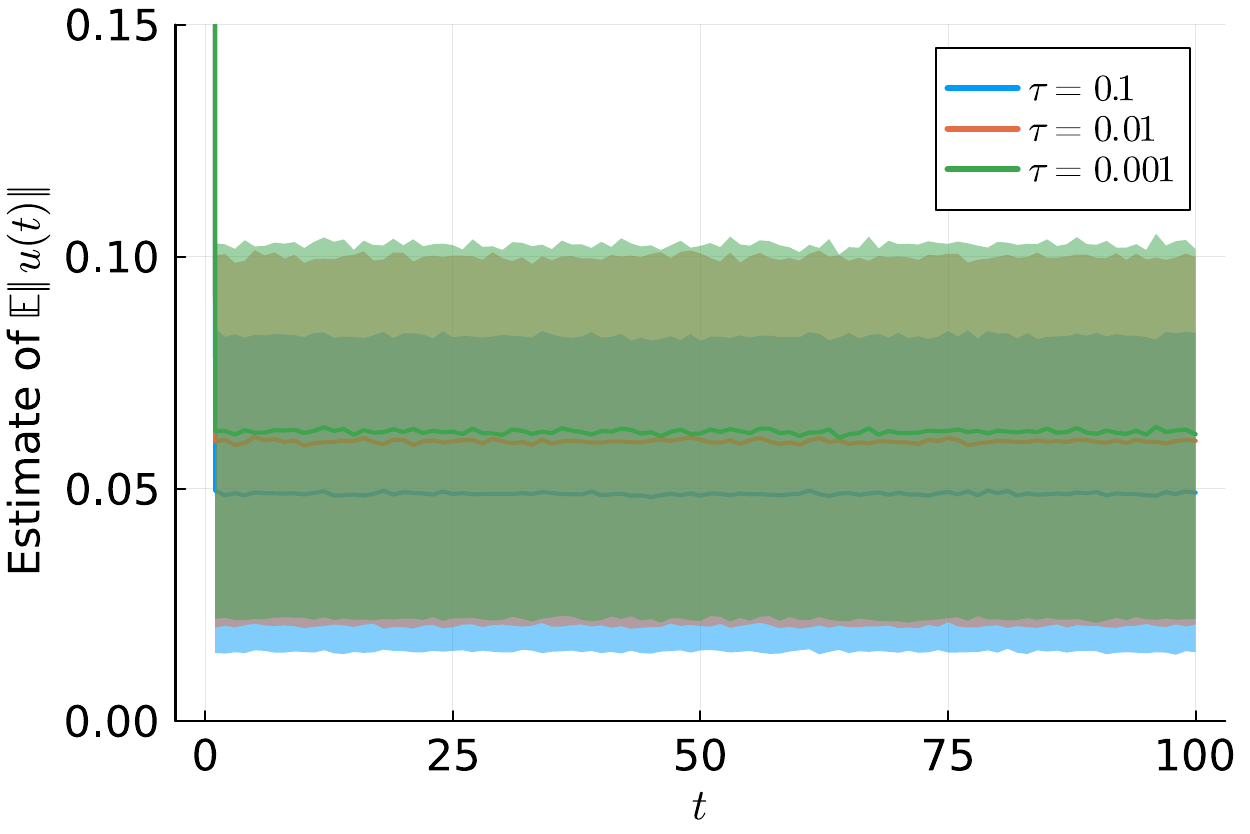}}
    \subfigure[Global gradient taming, \eqref{scheme:taming1}]{\includegraphics[width=6.5cm]{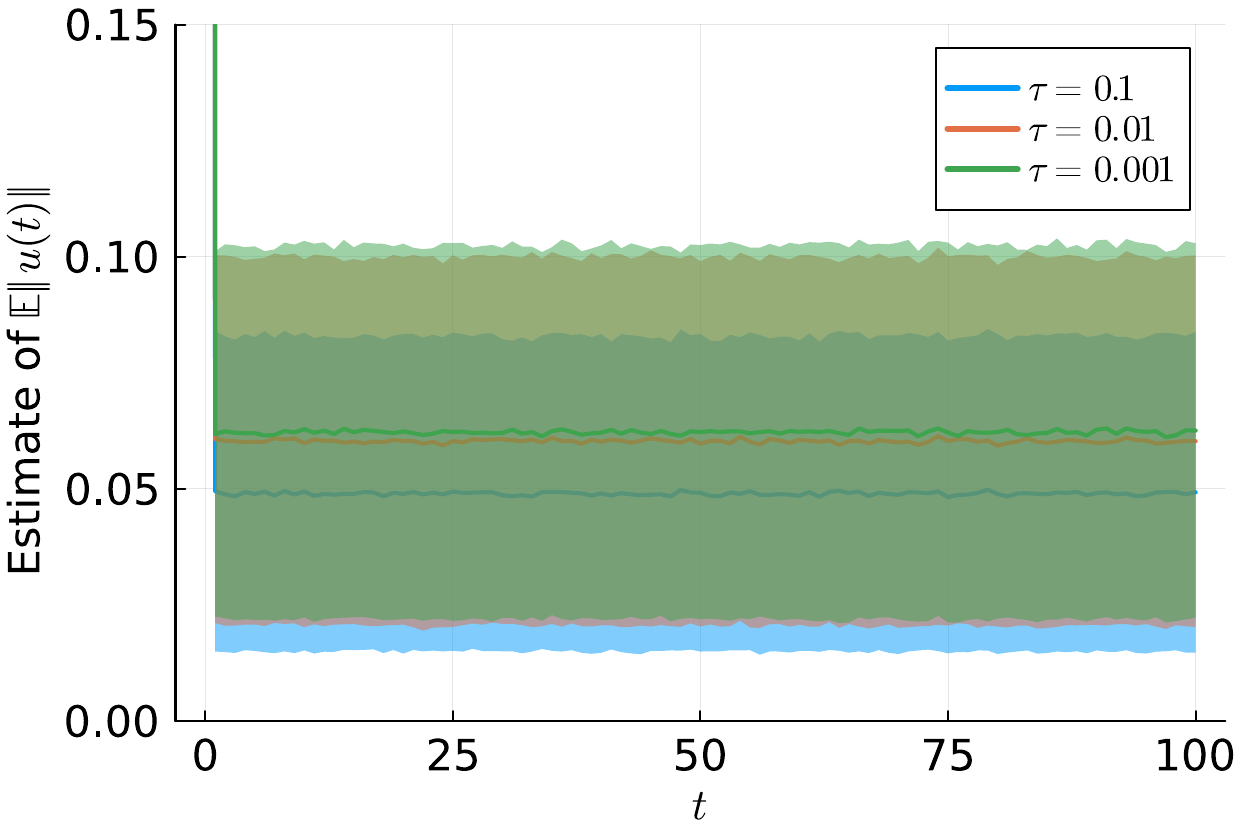}}
    \subfigure[Truncated global taming, \eqref{scheme:tame by norm of f'(u)}]{\includegraphics[width=6.5cm]{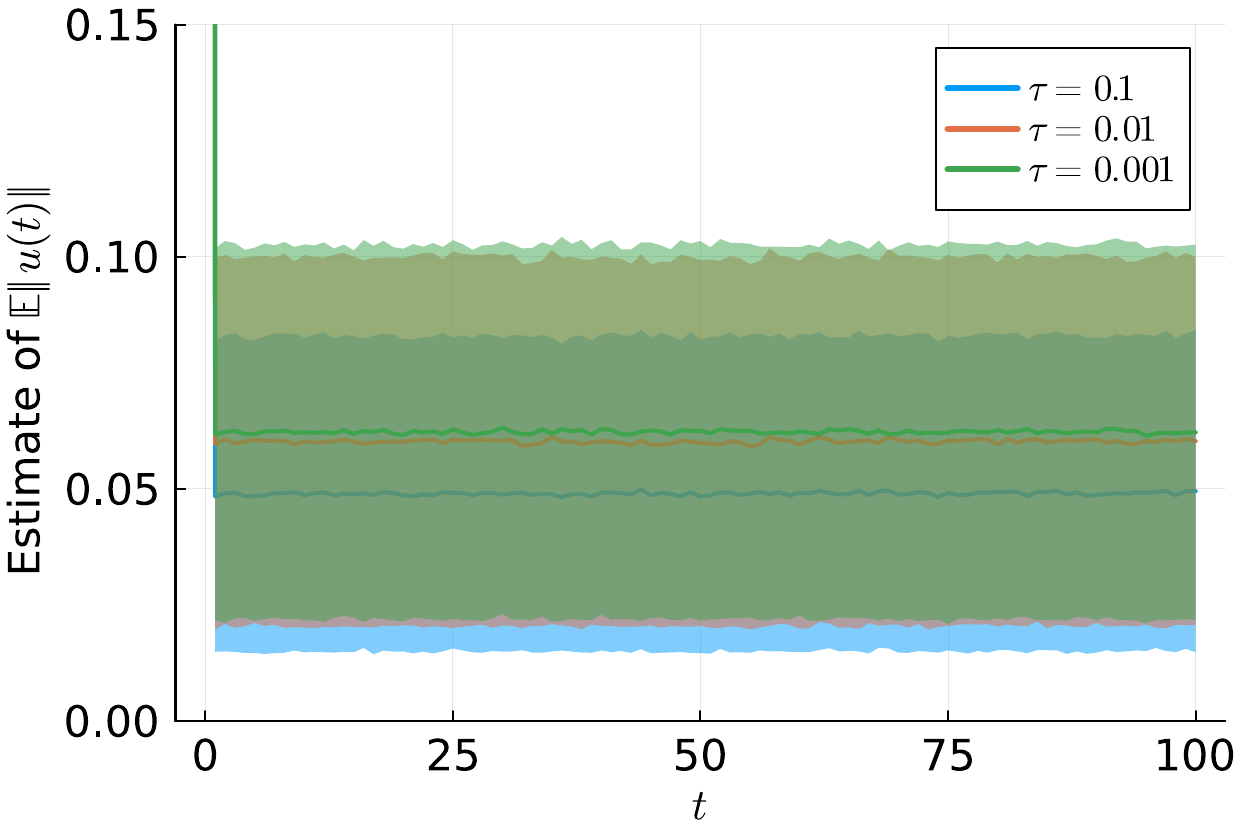}}
    
    \caption{Time series of FEM solutions for the first moment for initial condition $u_0(x) = 10\sin(10\pi x)$.  All methods appear quite comparable.
    Shaded regions reflect one standard deviation.}
    \label{fig:osc10_10_fem}
\end{figure}

\subsubsection{Spectral Galerkin Results}

For the spectral Galerkin problem, the results are more varied in two ways.  First, the  hypothesized blowup mechanism via the mean manifests itself, even when taming is used; we believe this is a consequence of the aforementioned lack of the Poincaré inequality, to be discussed more, below.  Second, during the initial transient phase, even for methods we expect to converge, empirically, some give better results than others. {We conjecture that this latter distinction is at least partially attributable to the ability of the system to relax throught the boundary in the Dirichlet case, something absent in the periodic case.}


The first set of results are, again, for constant initial conditions, and they appear in Figures \ref{fig:zerodata_fft}, \ref{fig:onedata_fft}, and \ref{fig:hundreddata_fft}.  As was the case of the FEM simulations, for zero data, all look in good agreement.  The same is true for $u_0(x) = 1$, the small data case.  Though in this small, nonzero, data case, we do see some discrepancy between methods at $\tau =0.1$ at shorter time scales ($t\lesssim 10$).



Figure \ref{fig:hundreddata_fft} depicts the large data results, where we observe two changes from the FEM case. First, the time series is absent for methods \eqref{scheme: general SIE}  and \eqref{scheme:taming1}; the first is due to blowup of the SIE method, but the latter is not. Indeed, for method \eqref{scheme:taming1}, the values are in excess of the vertical scale and this will be revisited below.  Second, for several methods, including  \eqref{scheme: GTEM scheme f}, \eqref{e:gyongyspde}, and \eqref{scheme:tame by norm of f'(u)}, there is a notable separation between the time series curves at short time ($t\lesssim 10$) for the larger values of $\tau$.   Again, in this example, the best results are found with the computationally expensive FIE method, while \eqref{scheme:taming by abs val f' pointwise standard form}, truncated pointwise taming, is in second place.



Lastly, we try the oscillatory initial condition $u_0(x) = 10 \sin(10x)$.  The results appear in Figure \ref{fig:osc10_10_fft}.  Again, there is broad consistency across methods, though at $\tau = 0.1$, there is some variation in performance.


\begin{figure}
    \centering

    \subfigure[SIE, \eqref{scheme: general SIE}]{\includegraphics[width=6.5cm]{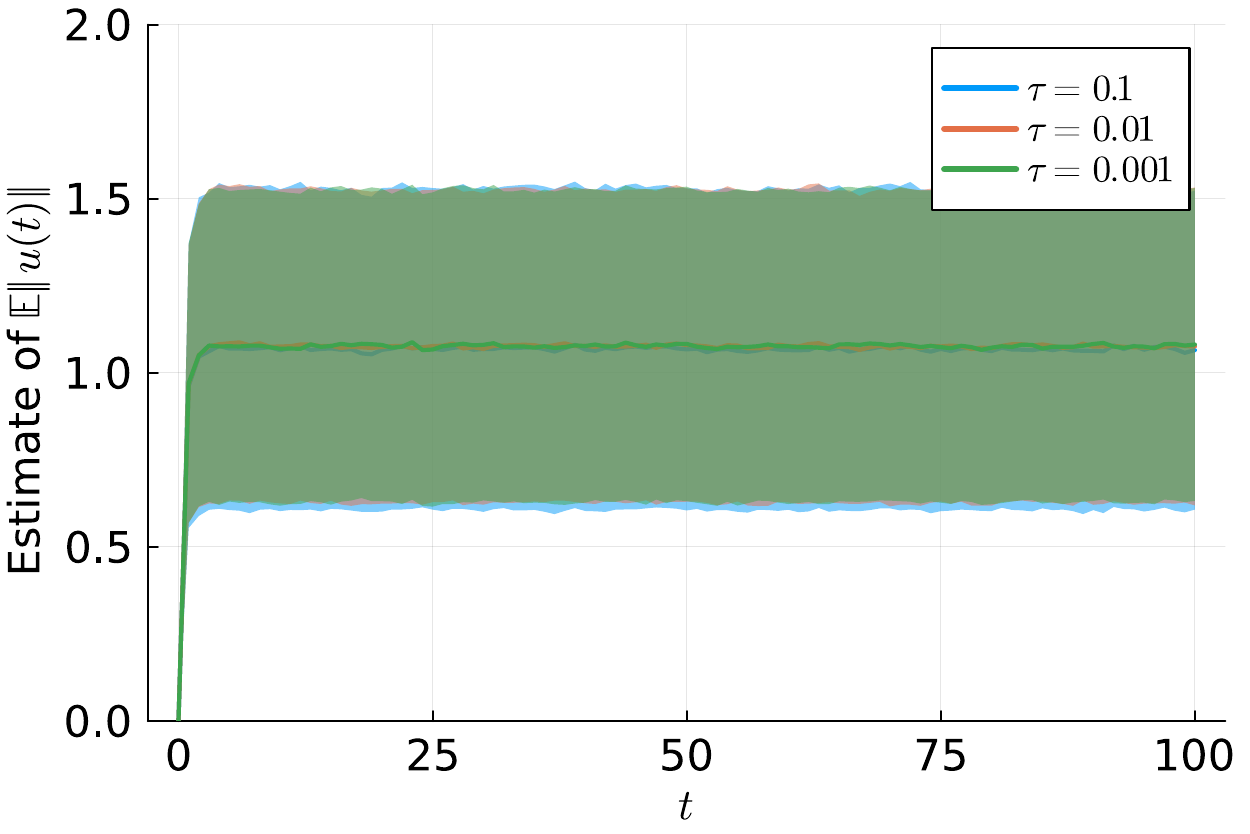}}
    \subfigure[FIE, \eqref{scheme:im}]{\includegraphics[width=6.5cm]{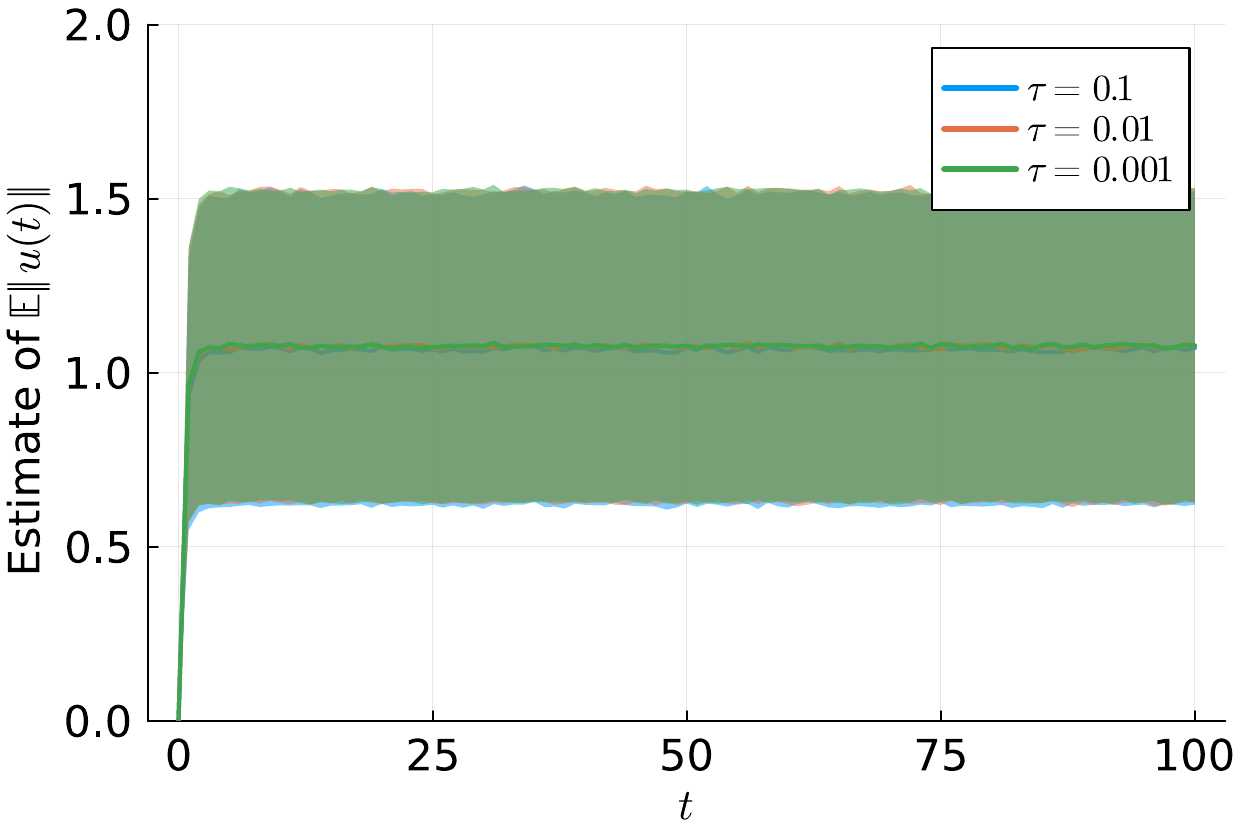}}
    \subfigure[Gy\"ongy's method, \eqref{e:gyongyspde}]{\includegraphics[width=6.5cm]{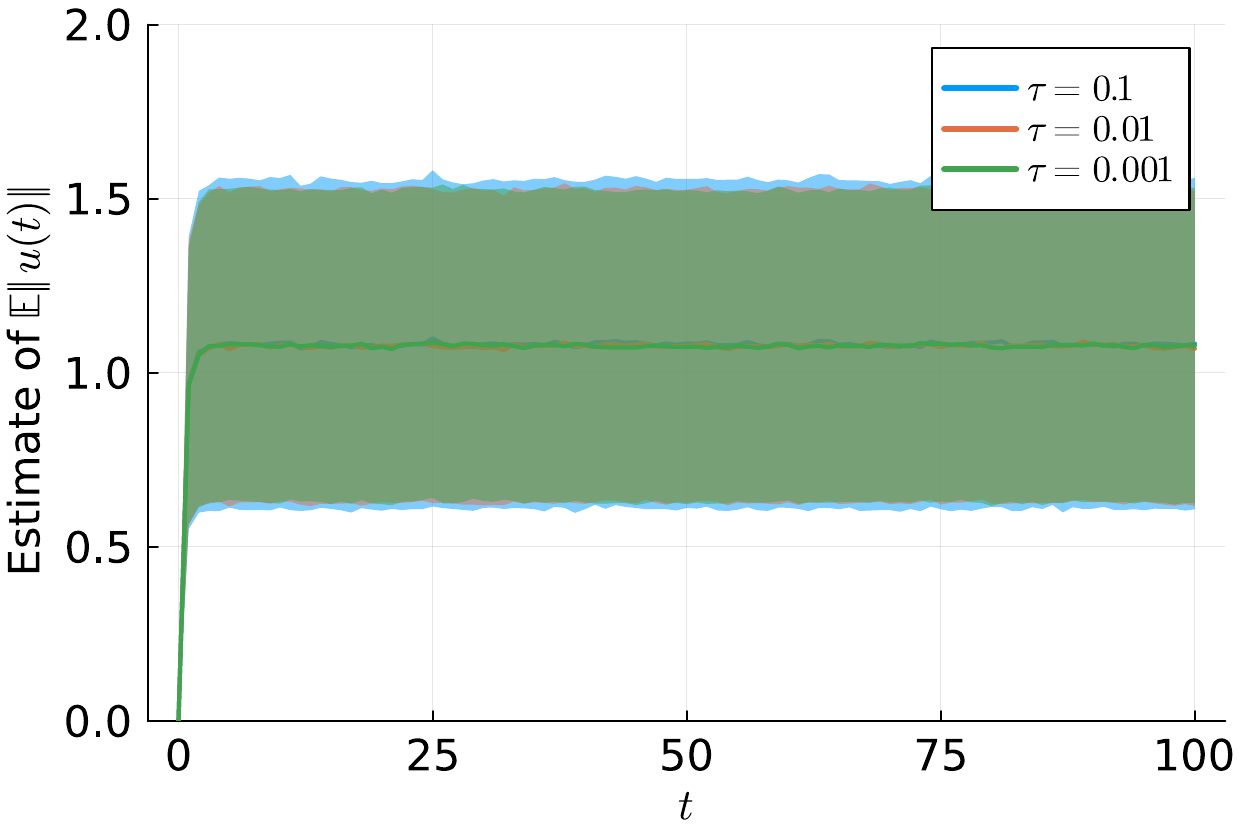}}
    \subfigure[Truncated pointwise taming,  \eqref{scheme:taming by abs val f' pointwise standard form}]{\includegraphics[width=6.5cm]{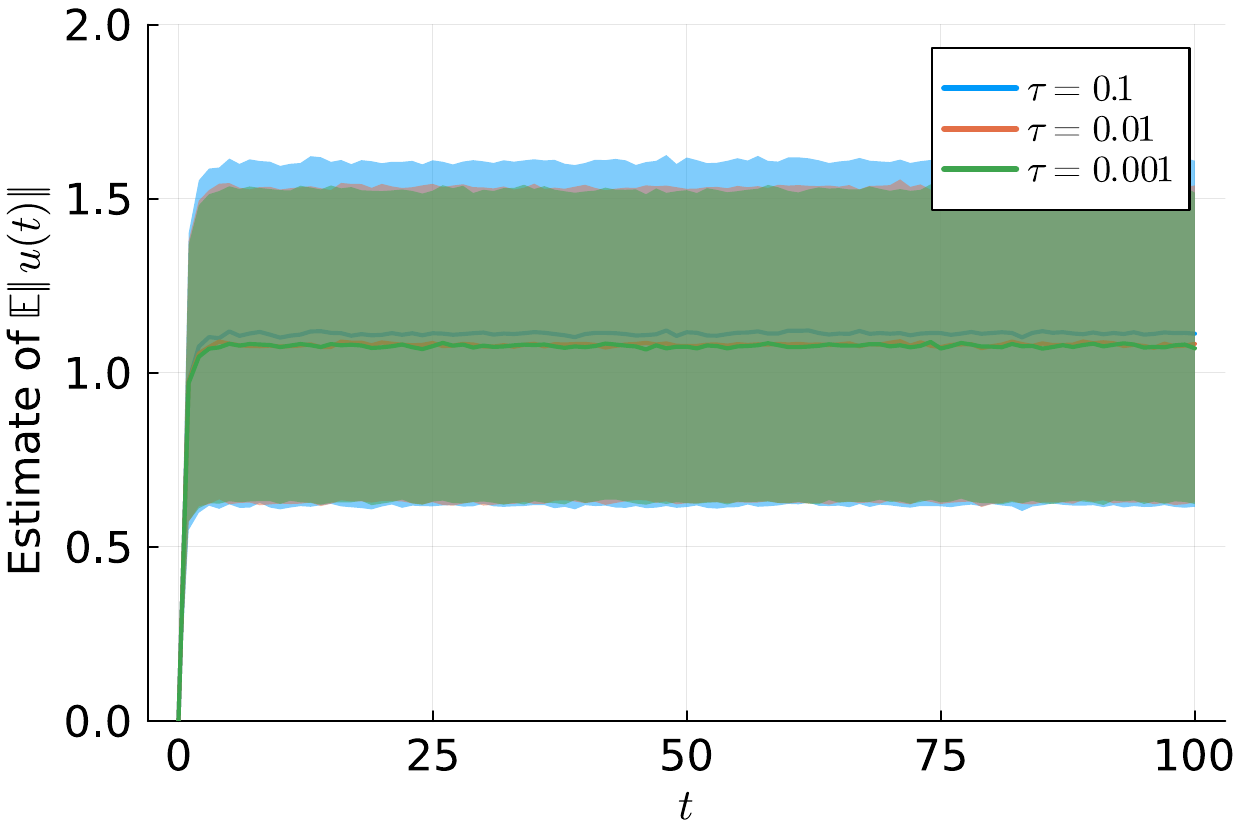}}
    \subfigure[GTEM, \eqref{scheme: GTEM scheme f}]{\includegraphics[width=6.5cm]{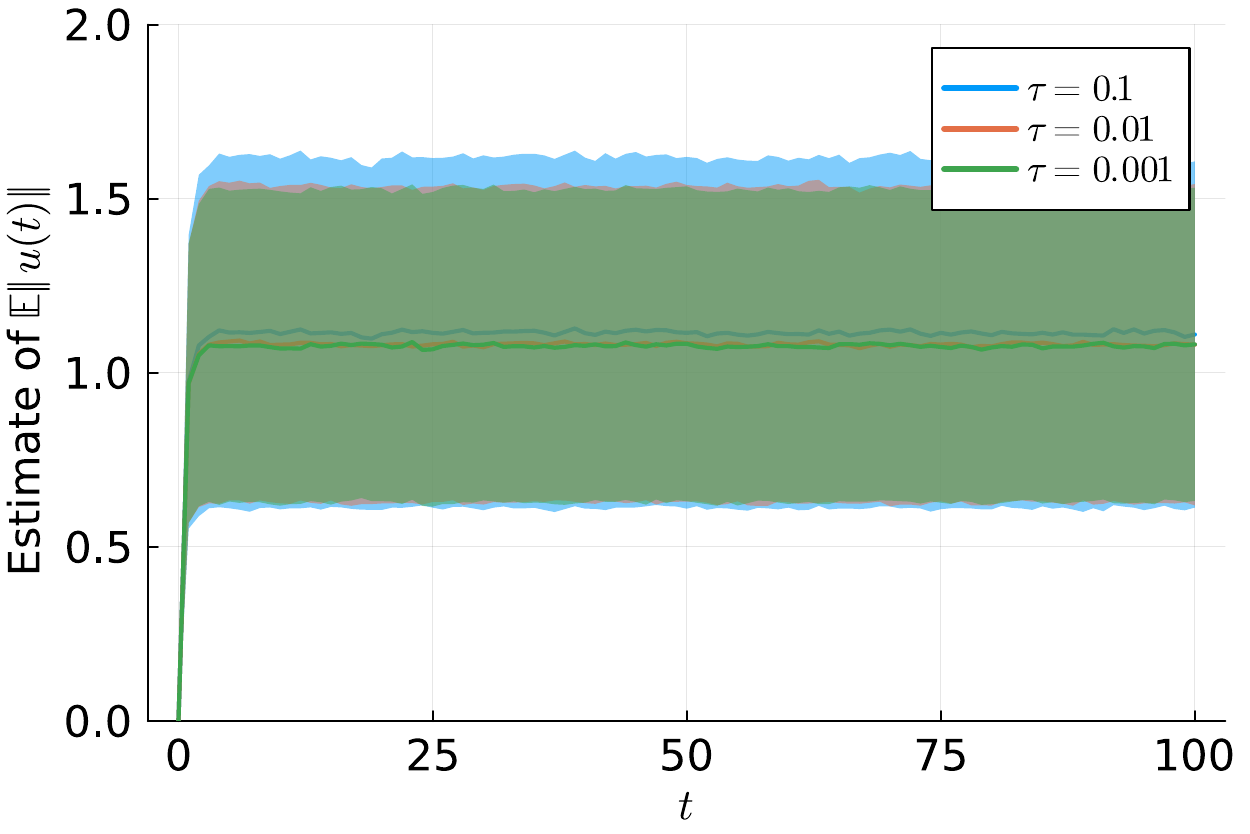}}
    \subfigure[Global gradient taming, \eqref{scheme:taming1}]{\includegraphics[width=6.5cm]{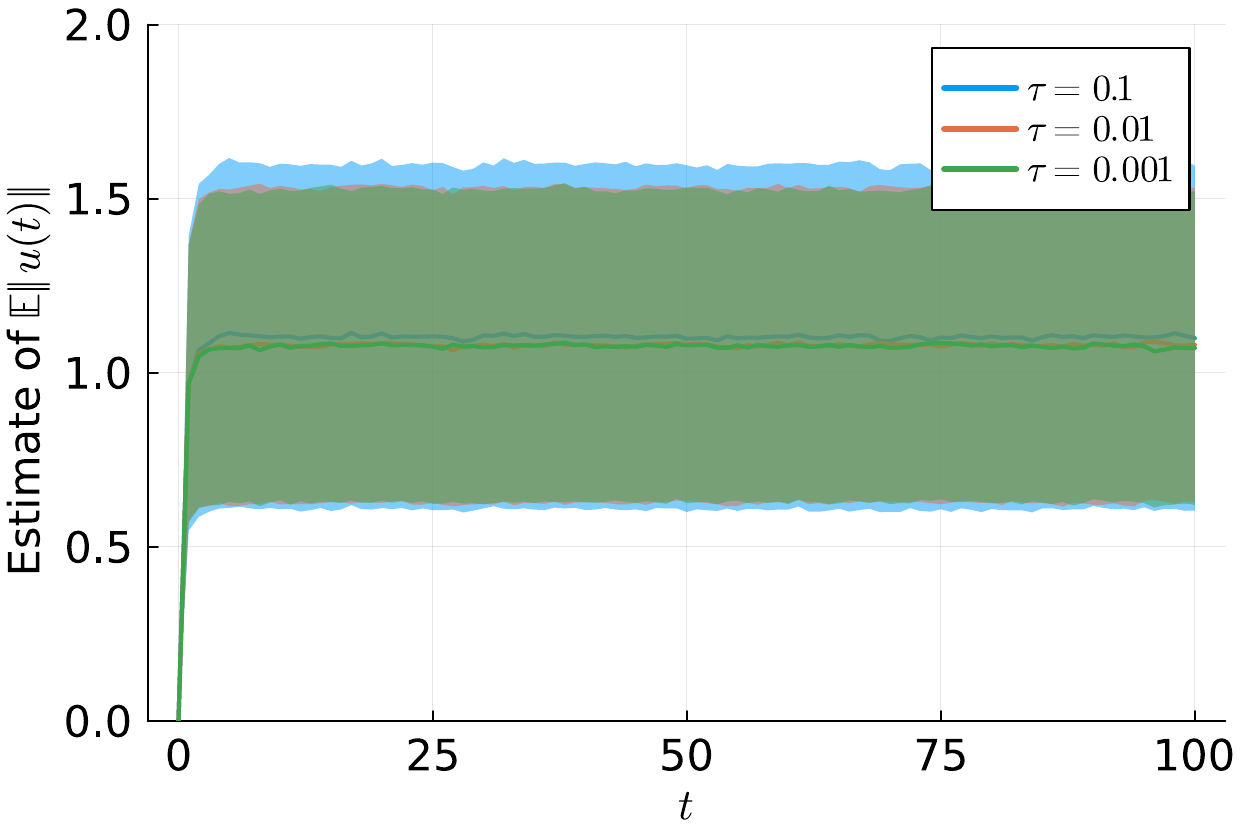}}
    \subfigure[Truncated global taming, \eqref{scheme:tame by norm of f'(u)}]{\includegraphics[width=6.5cm]{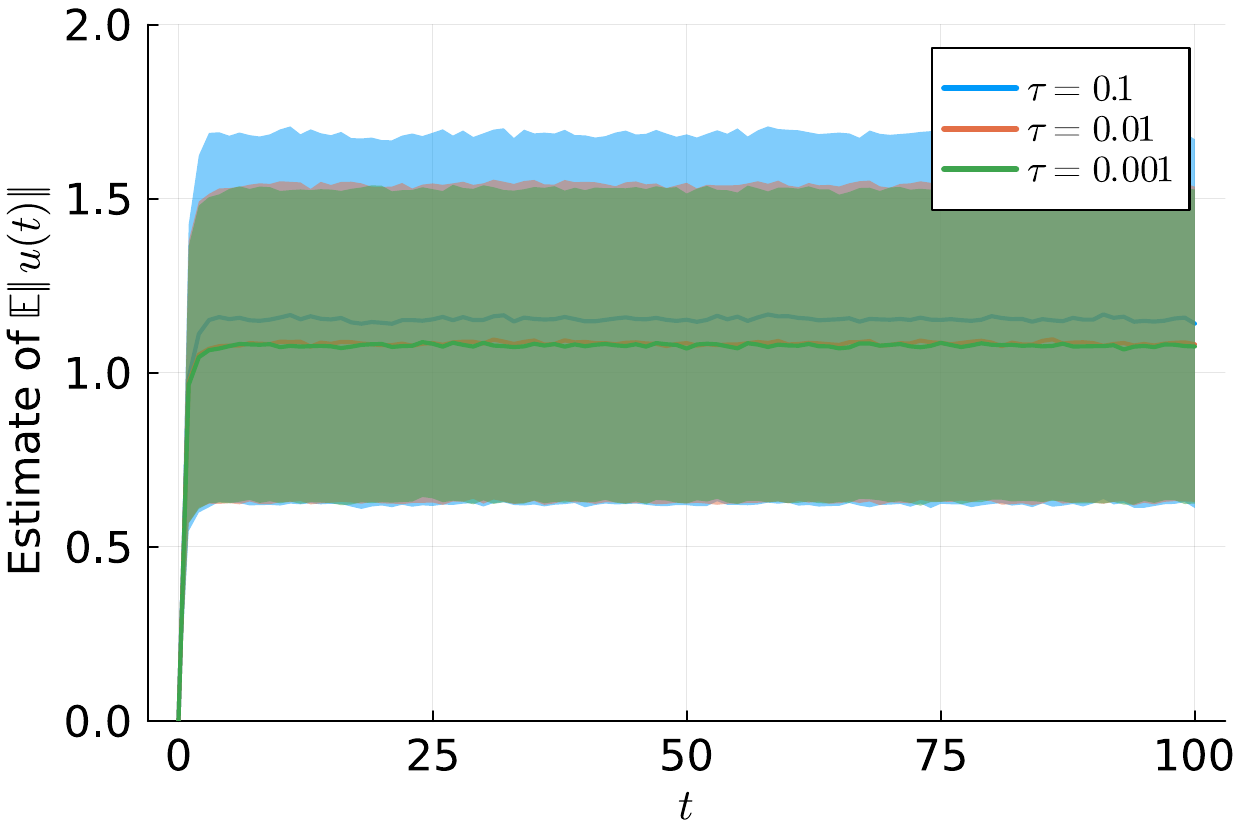}}
    
    \caption{Time series of spectral Galerkin solutions for the first moment for initial condition $u_0(x) = 0$. All methods are quite comparable.  Shaded regions reflect one standard deviation.}
    \label{fig:zerodata_fft}
\end{figure}

\begin{figure}
    \centering

    \subfigure[SIE, \eqref{scheme: general SIE}]{\includegraphics[width=6.5cm]{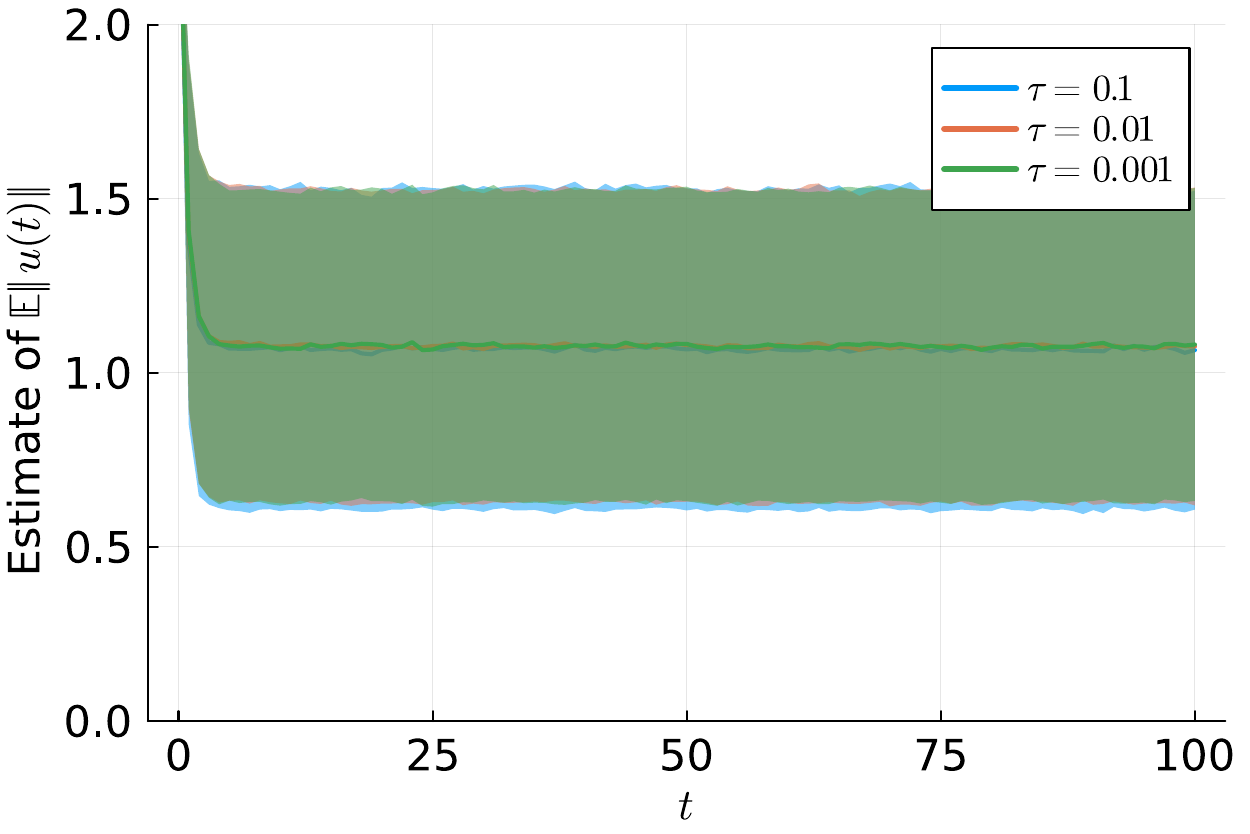}}
    \subfigure[FIE, \eqref{scheme:im}]{\includegraphics[width=6.5cm]{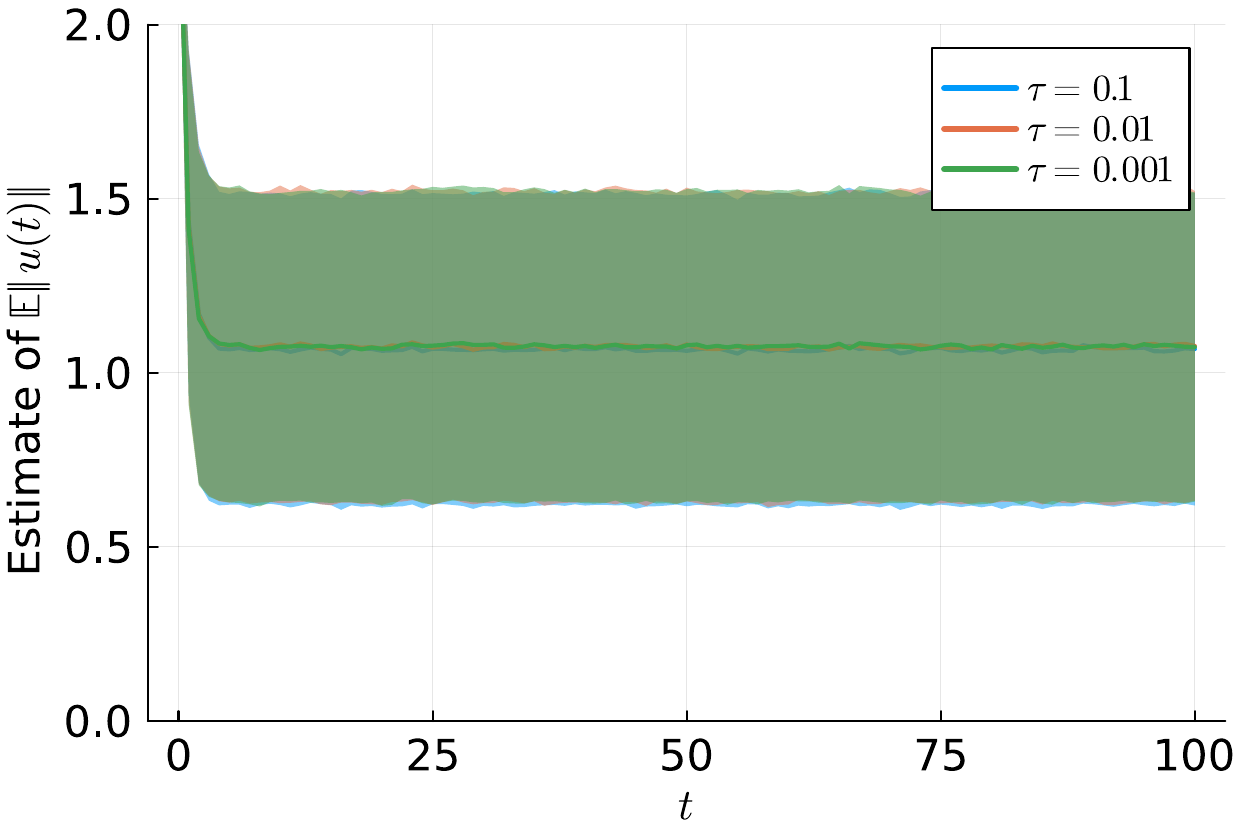}}
    \subfigure[Gy\"ongy's method, \eqref{e:gyongyspde}]{\includegraphics[width=6.5cm]{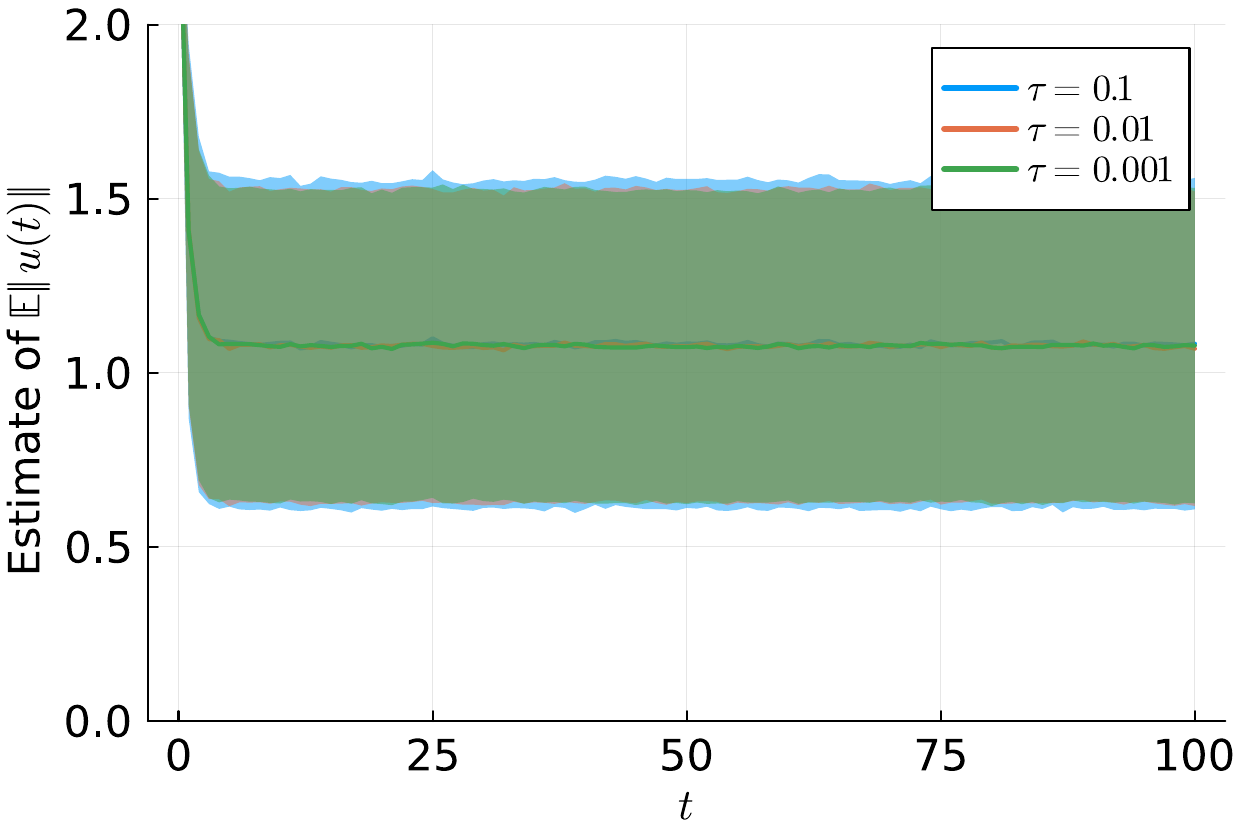}}
    \subfigure[Truncated pointwise taming,  \eqref{scheme:taming by abs val f' pointwise standard form}]{\includegraphics[width=6.5cm]{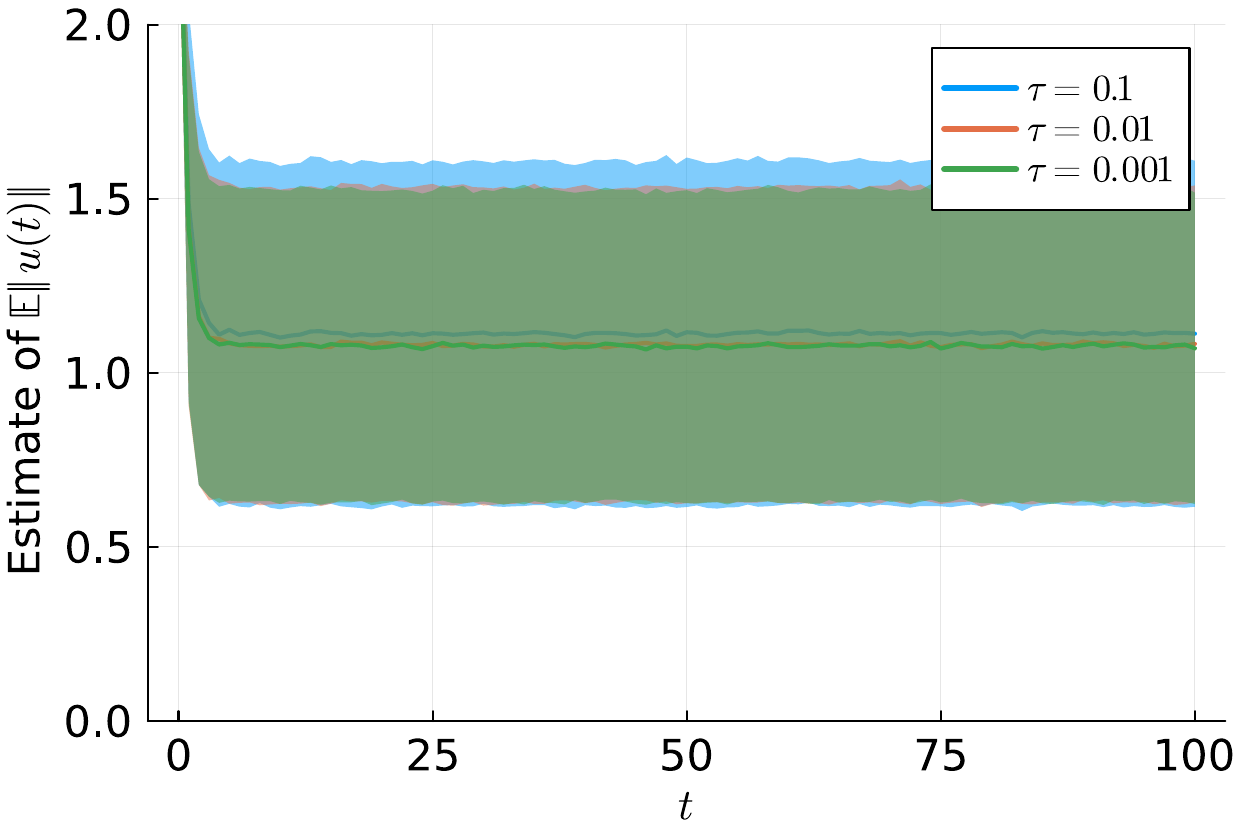}}
    \subfigure[GTEM, \eqref{scheme: GTEM scheme f}]{\includegraphics[width=6.5cm]{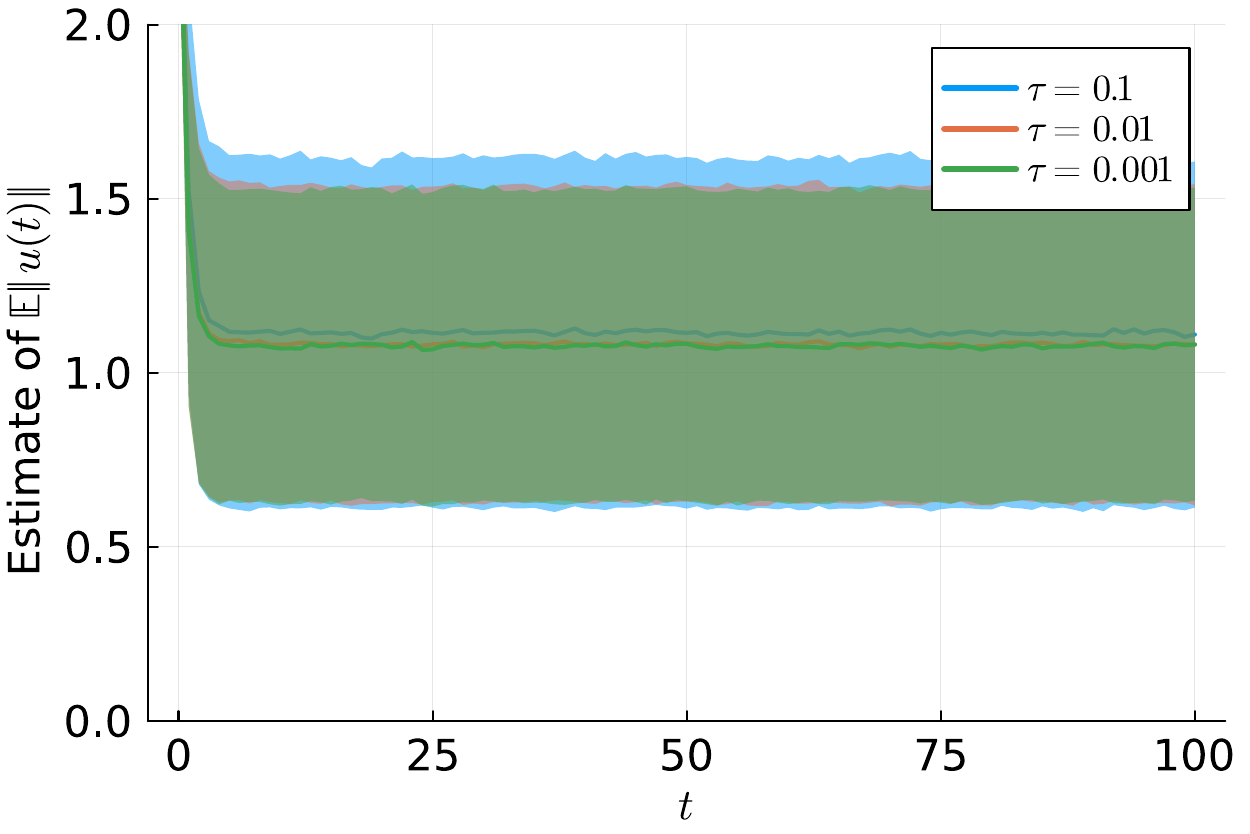}}
    \subfigure[Global gradient taming, \eqref{scheme:taming1}]{\includegraphics[width=6.5cm]{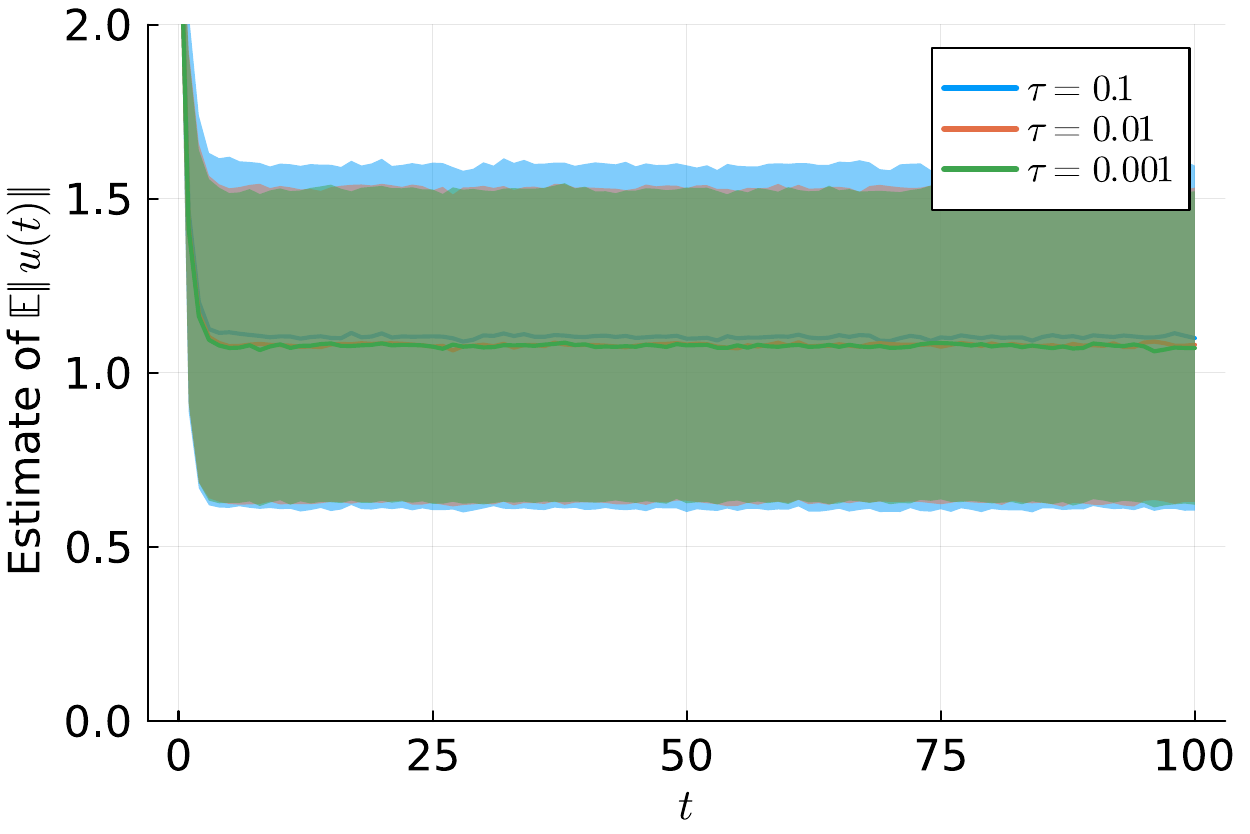}}
    \subfigure[Truncated global taming, \eqref{scheme:tame by norm of f'(u)}]{\includegraphics[width=6.5cm]{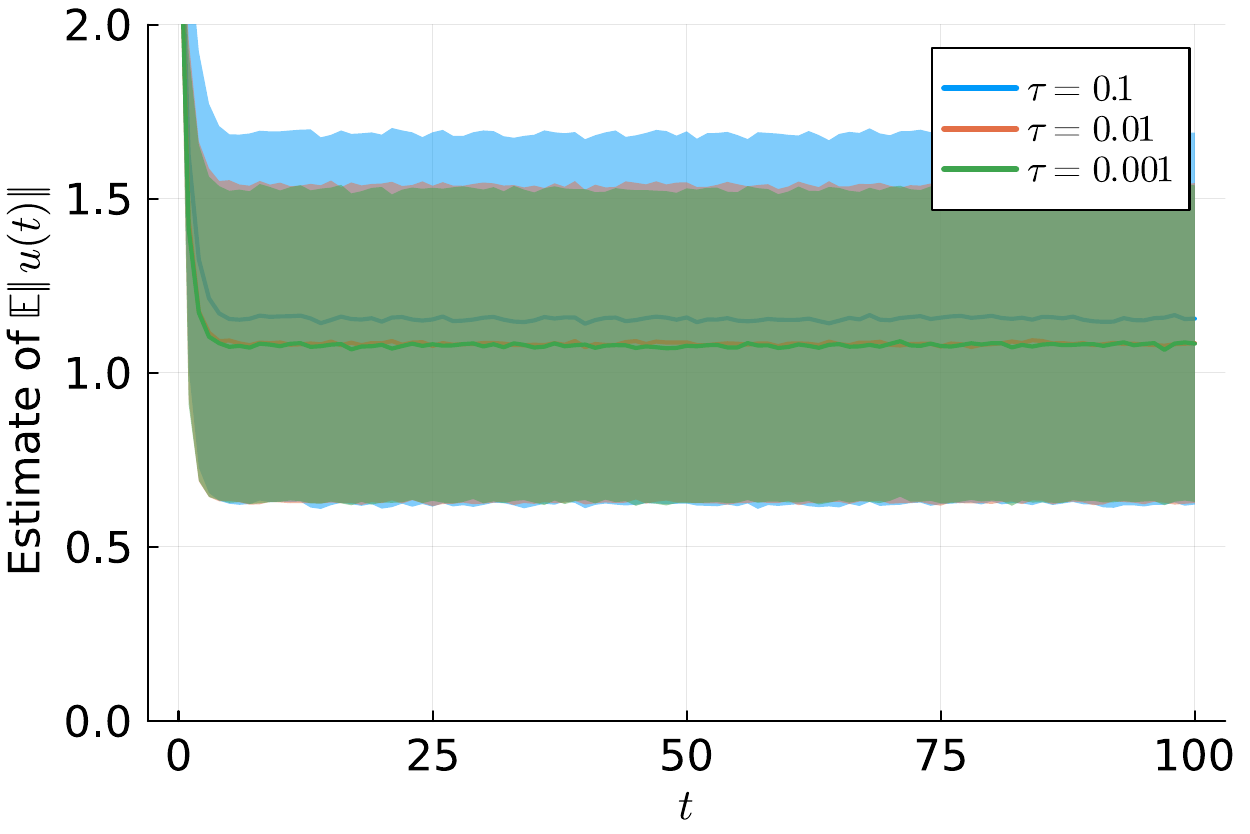}}
    
    \caption{Time series of spectral Galerkin solutions for the first moment for initial condition $u_0(x) = 1$. Results are in good agreement as $\tau$ becomes smaller, though we do see some variation at the larger values of $\tau$.  Shaded regions reflect one standard deviation.}
    \label{fig:onedata_fft}
\end{figure}

\begin{figure}
    \centering

    \subfigure[SIE, \eqref{scheme: general SIE}]{\includegraphics[width=6.5cm]{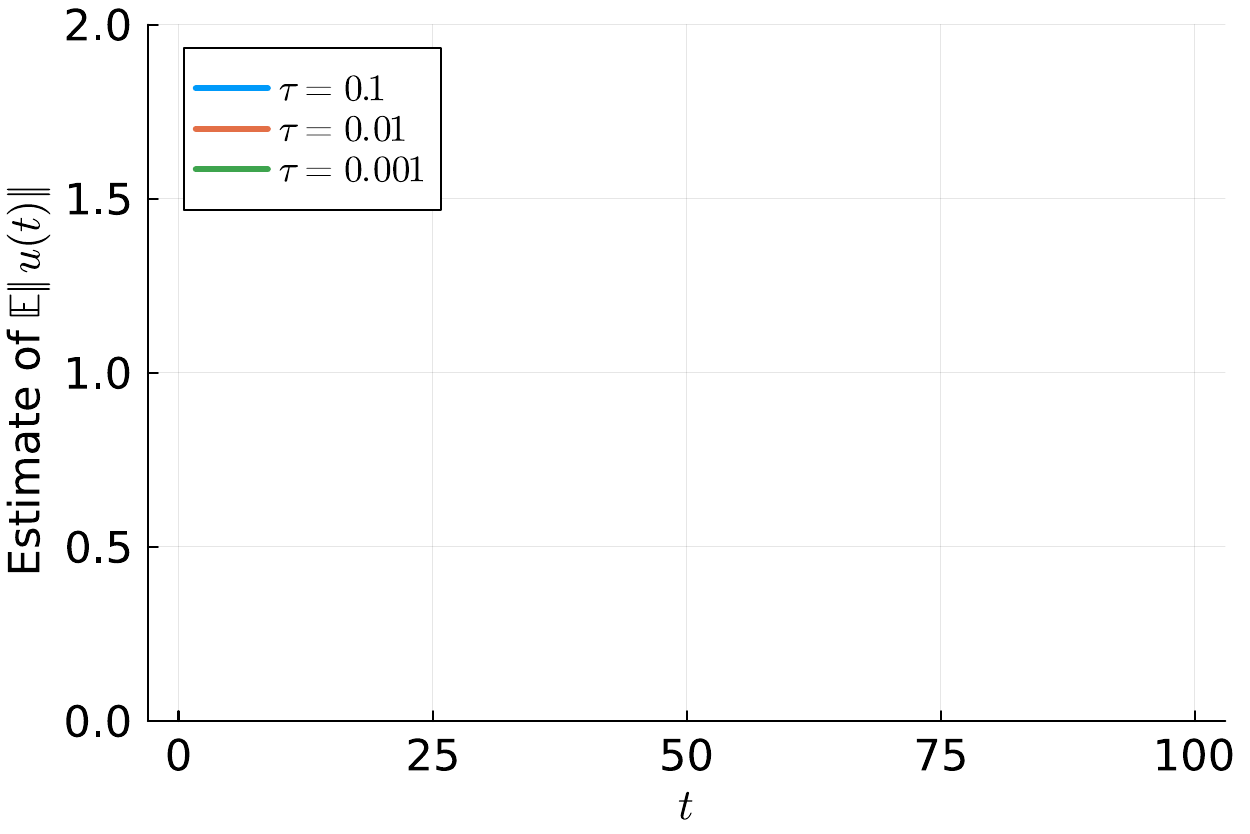}}
    \subfigure[FIE, \eqref{scheme:im}]{\includegraphics[width=6.5cm]{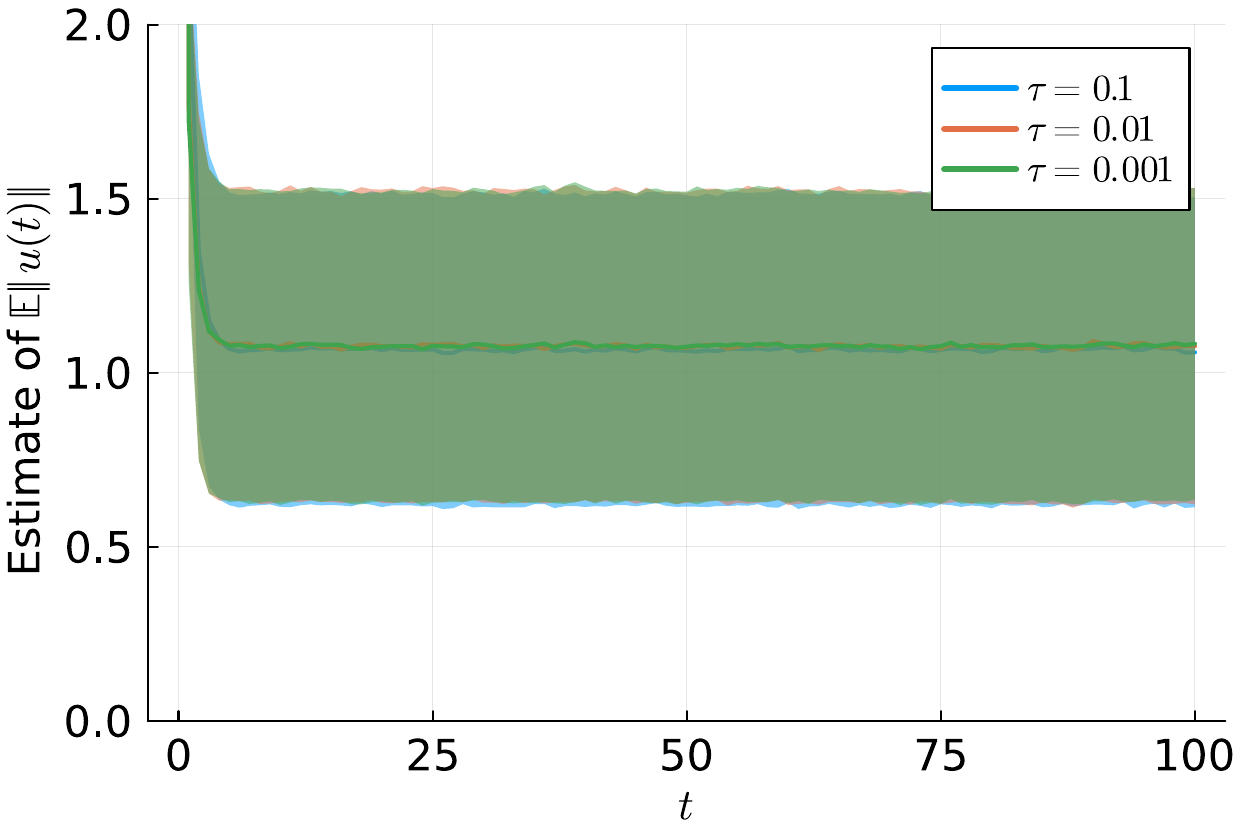}}
    \subfigure[Gy\"ongy's method, \eqref{e:gyongyspde}]{\includegraphics[width=6.5cm]{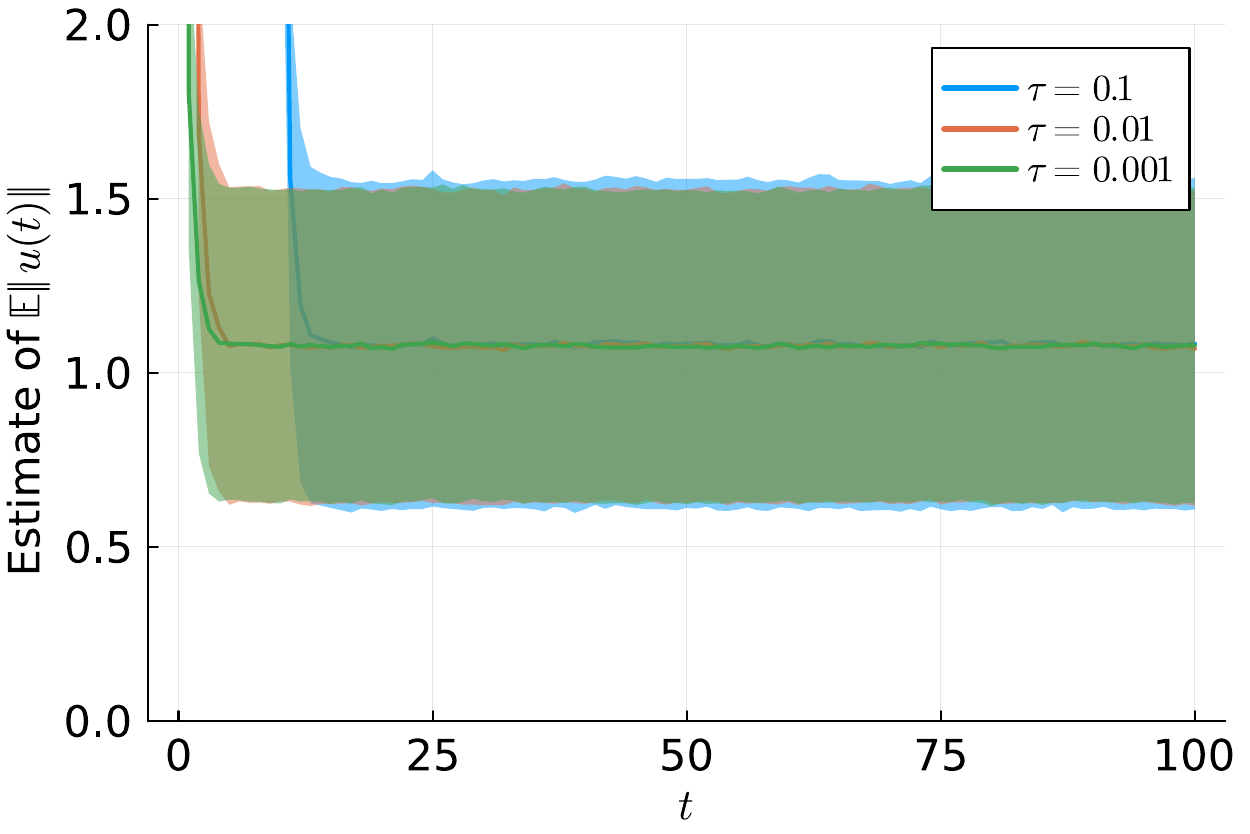}}
    \subfigure[Truncated pointwise taming,  \eqref{scheme:taming by abs val f' pointwise standard form}]{\includegraphics[width=6.5cm]{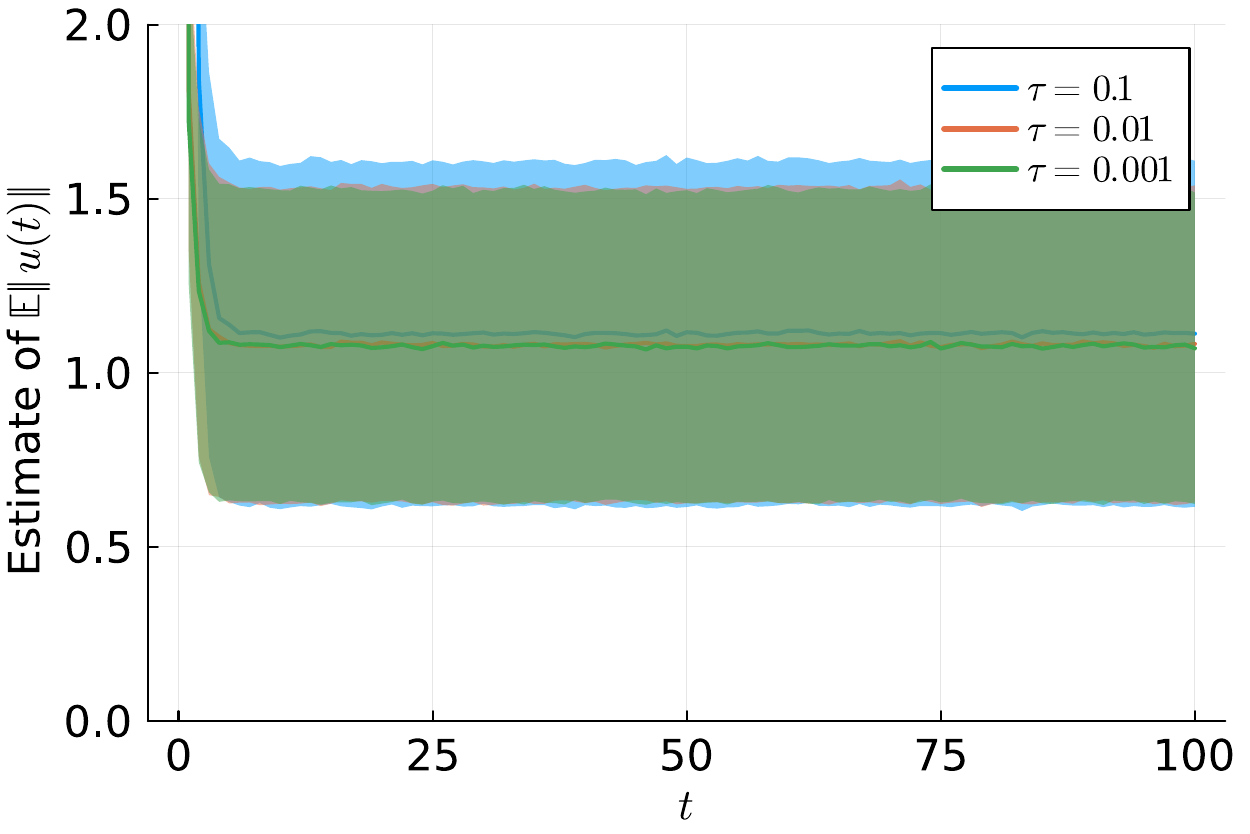}}
    \subfigure[GTEM, \eqref{scheme: GTEM scheme f}]{\includegraphics[width=6.5cm]{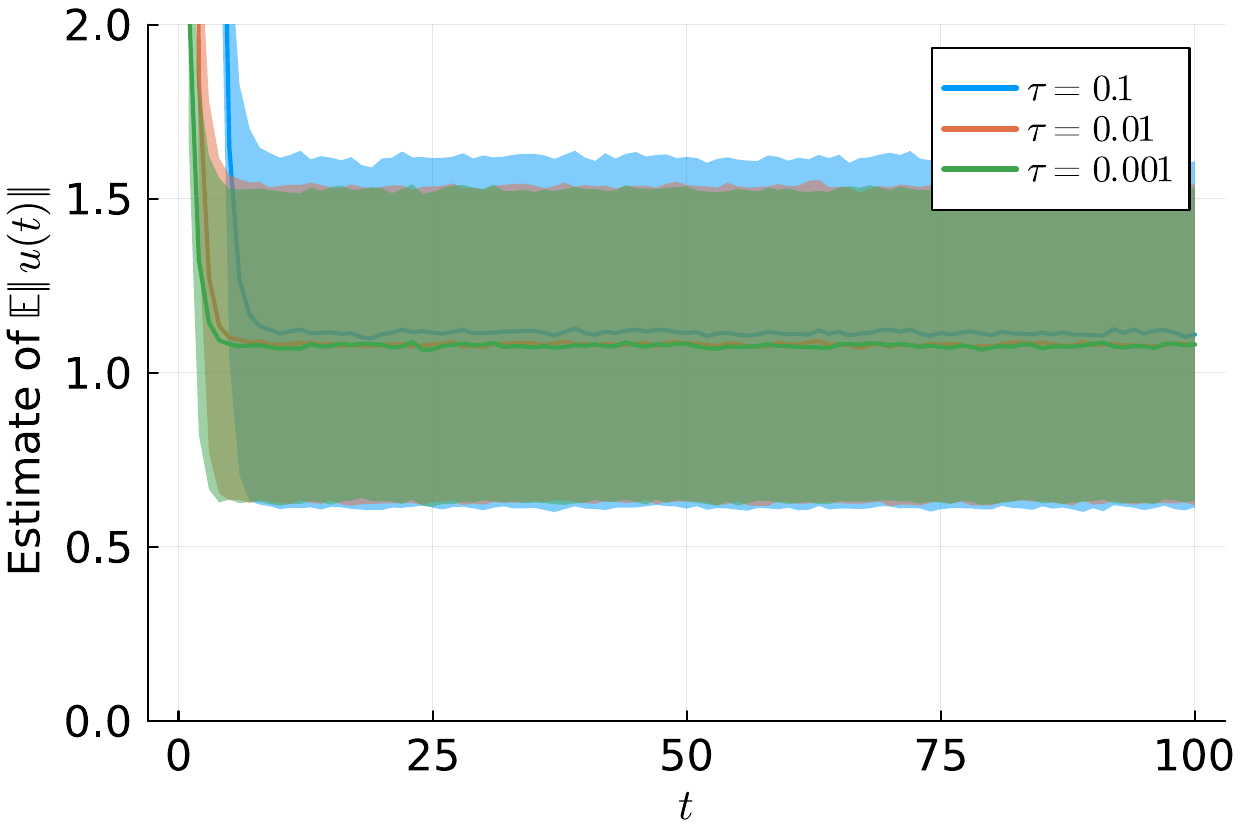}}
    \subfigure[Global gradient taming, \eqref{scheme:taming1}]{\includegraphics[width=6.5cm]{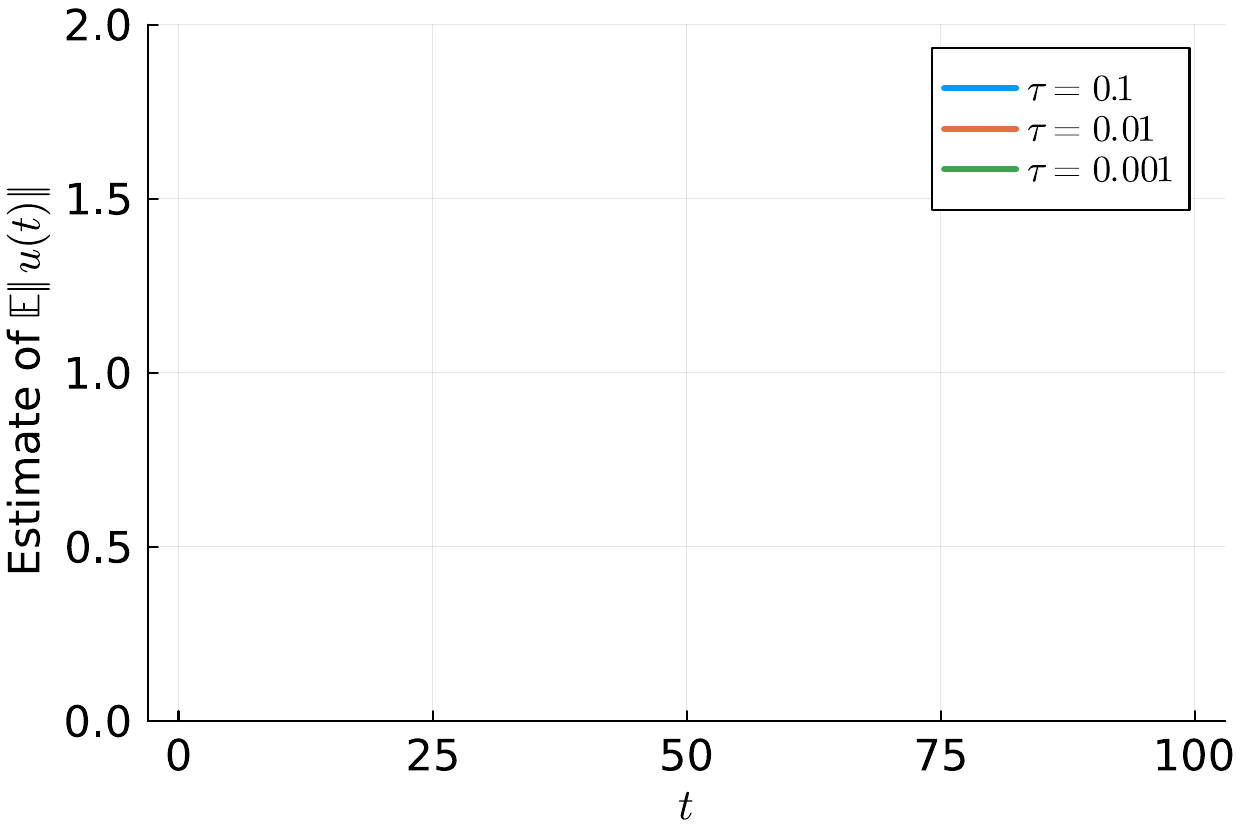}}
    \subfigure[Truncated global taming, \eqref{scheme:tame by norm of f'(u)}]{\includegraphics[width=6.5cm]{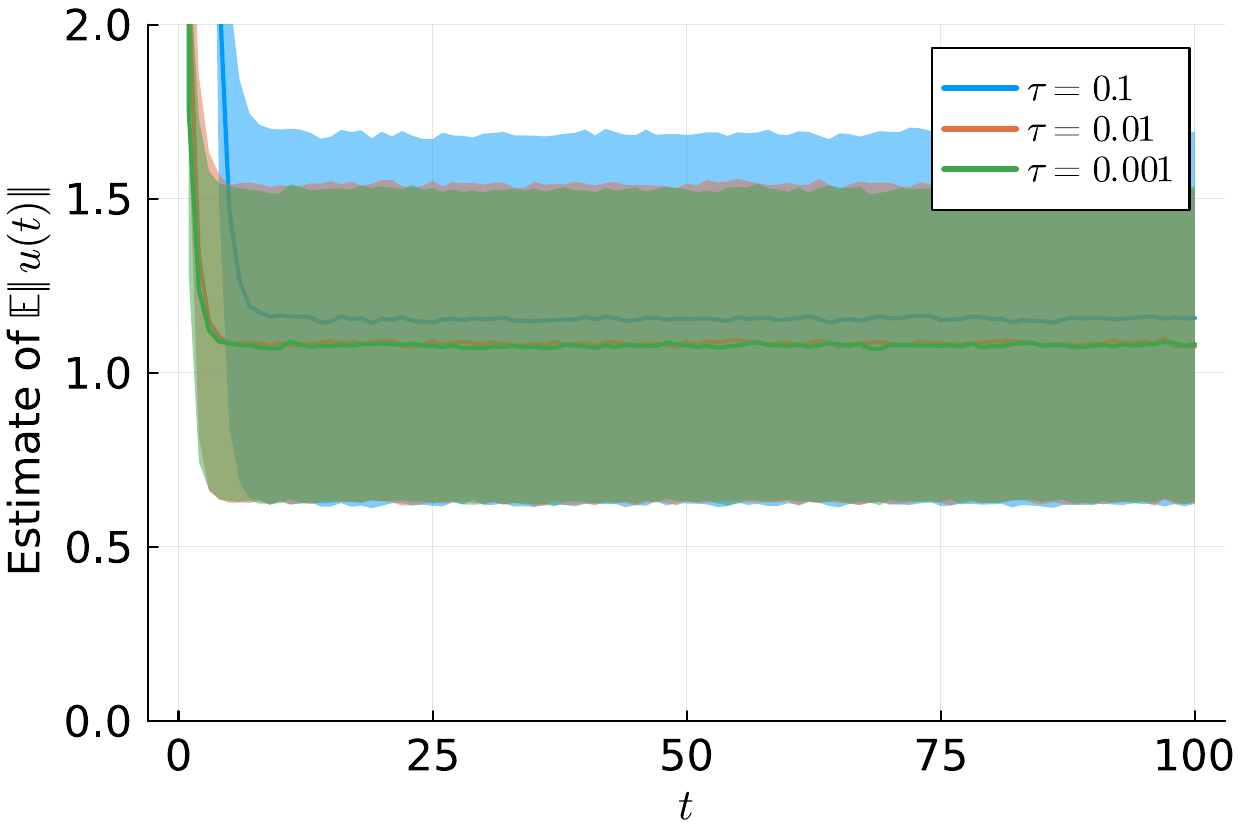}}
    
    \caption{Time series of spectral Galerkin solutions for the first moment for initial condition $u_0(x) = 100$.  SIE experiences blowup while scheme \eqref{scheme:taming1} produces values that are beyond the vertical scale.  Several schemes show visible separation between curves at $\tau = 0.1$. Shaded regions reflect one standard deviation.}
    \label{fig:hundreddata_fft}
\end{figure}



\begin{figure}
    \centering

    \subfigure[SIE, \eqref{scheme: general SIE}]{\includegraphics[width=6.5cm]{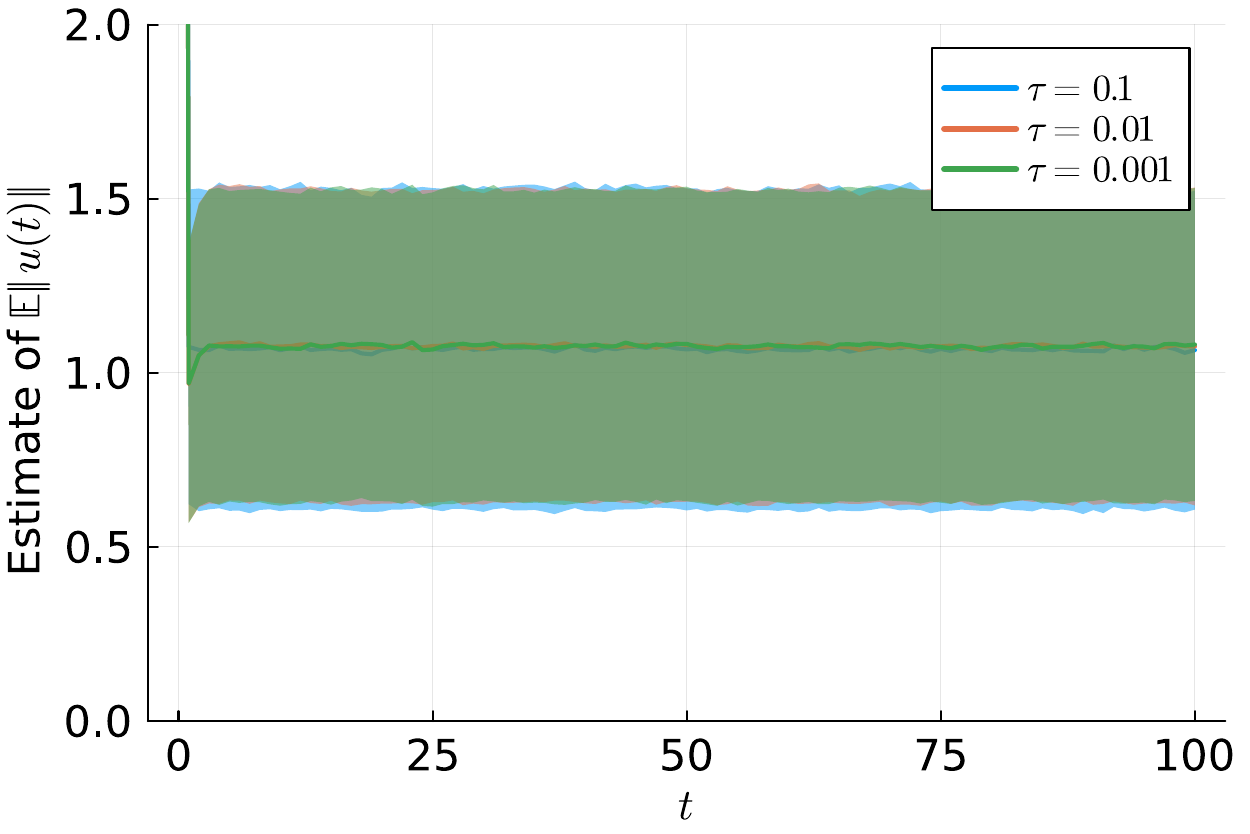}}
    \subfigure[FIE, \eqref{scheme:im}]{\includegraphics[width=6.5cm]{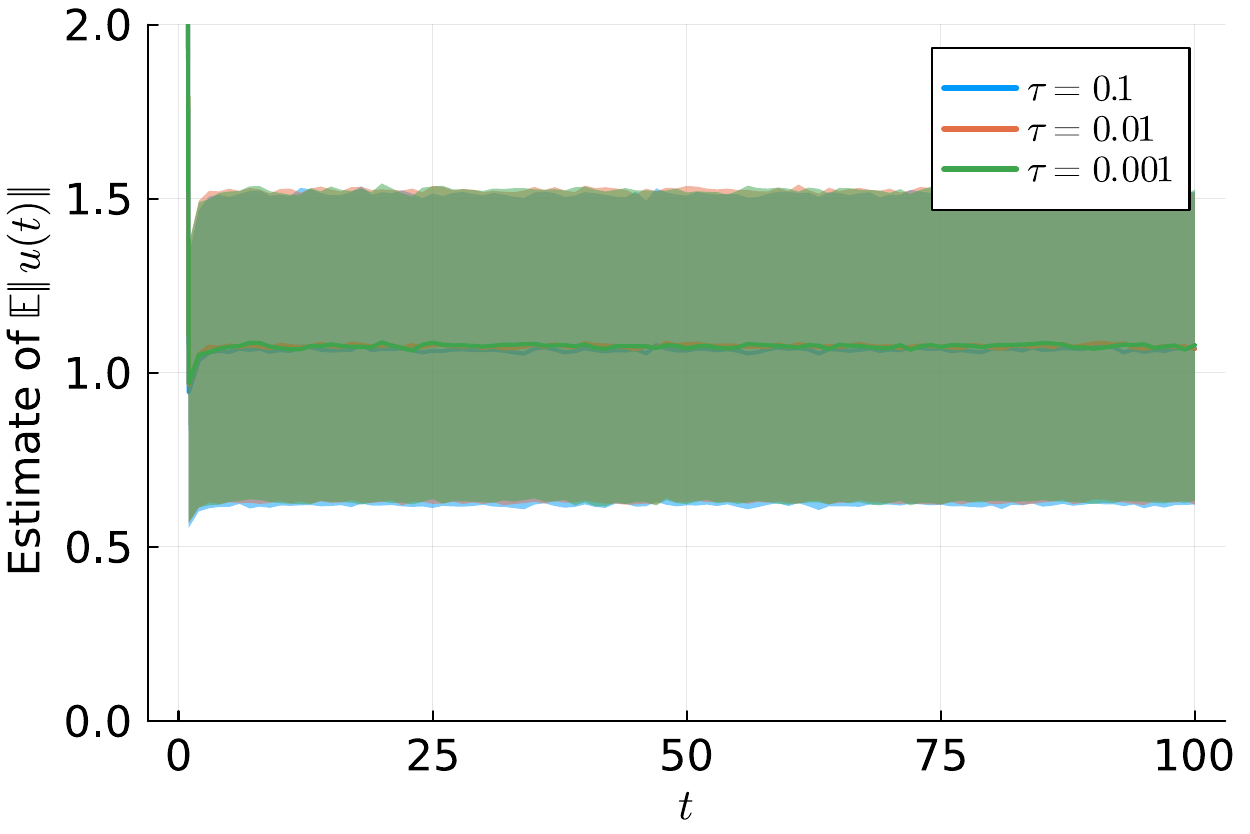}}
    \subfigure[Gy\"ongy's method, \eqref{e:gyongyspde}]{\includegraphics[width=6.5cm]{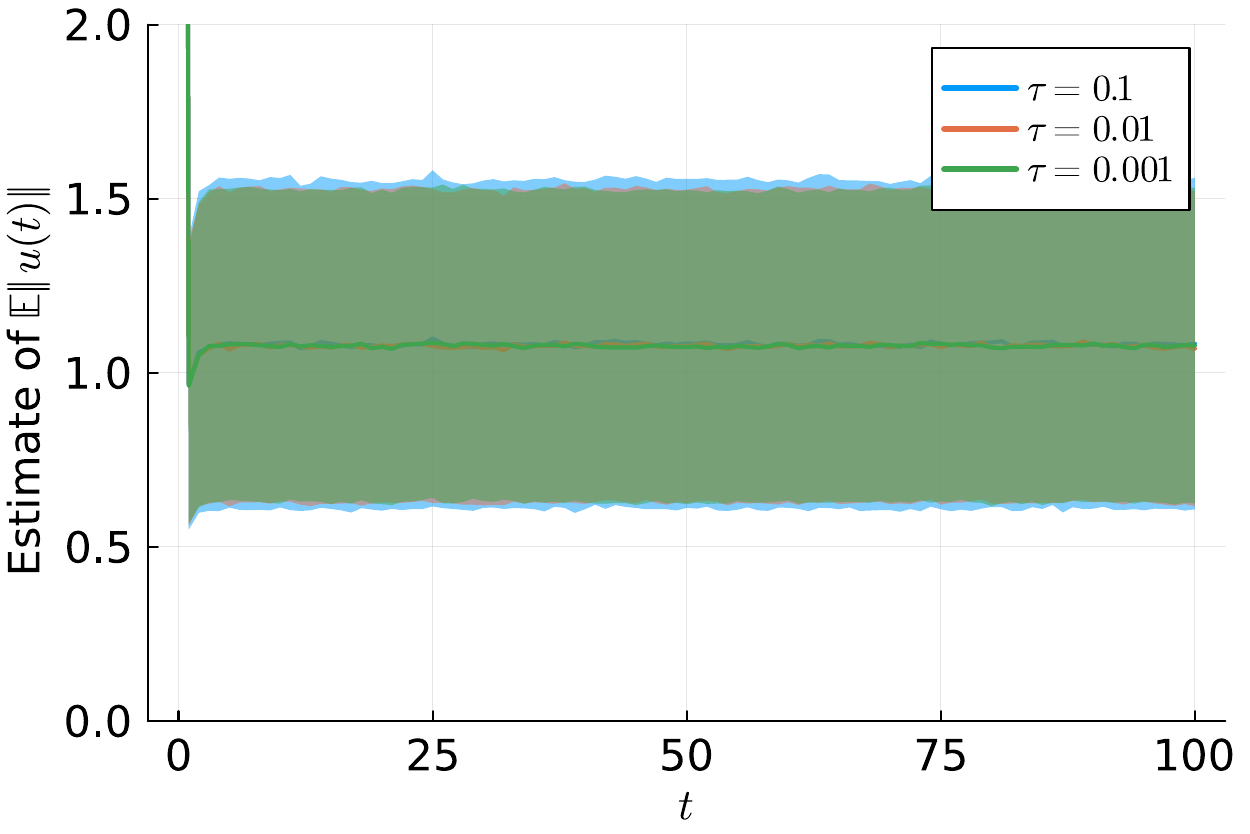}}
    \subfigure[Truncated pointwise taming,  \eqref{scheme:taming by abs val f' pointwise standard form}]{\includegraphics[width=6.5cm]{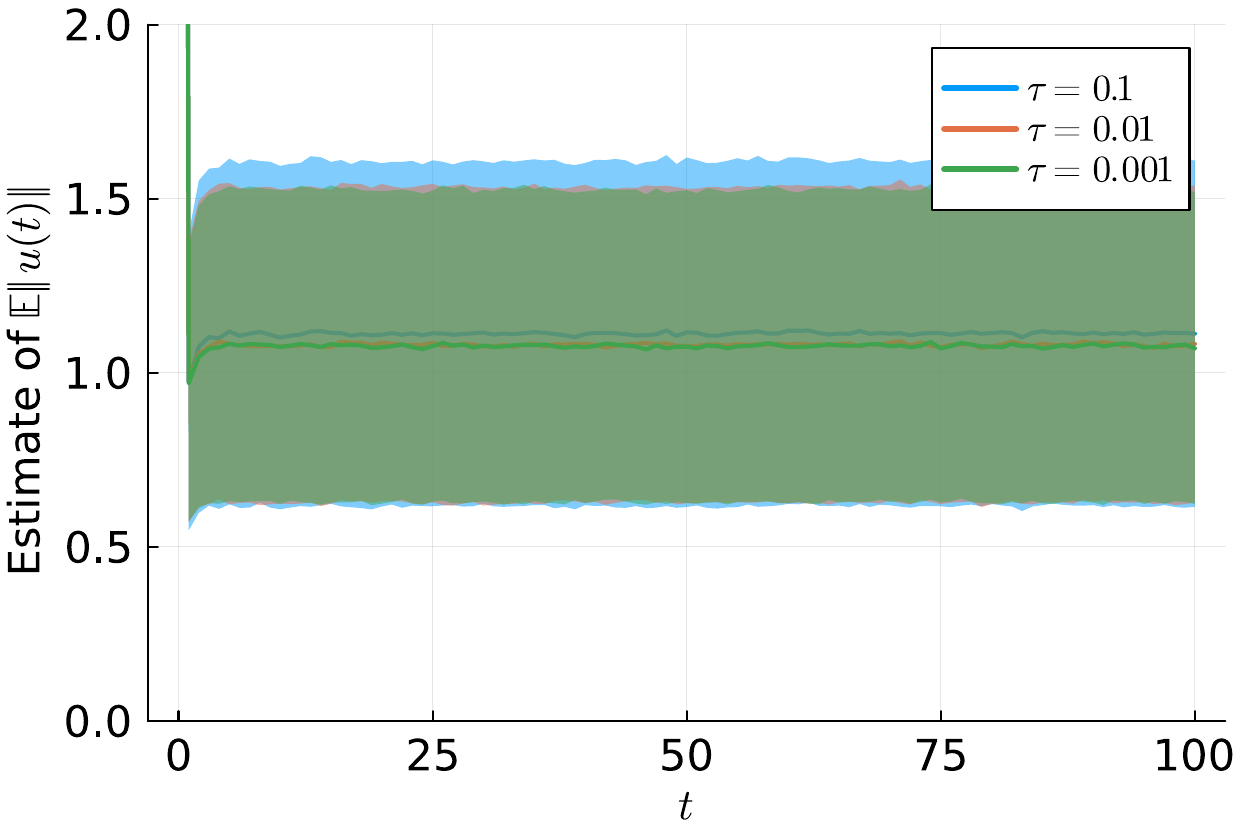}}
    \subfigure[GTEM, \eqref{scheme: GTEM scheme f}]{\includegraphics[width=6.5cm]{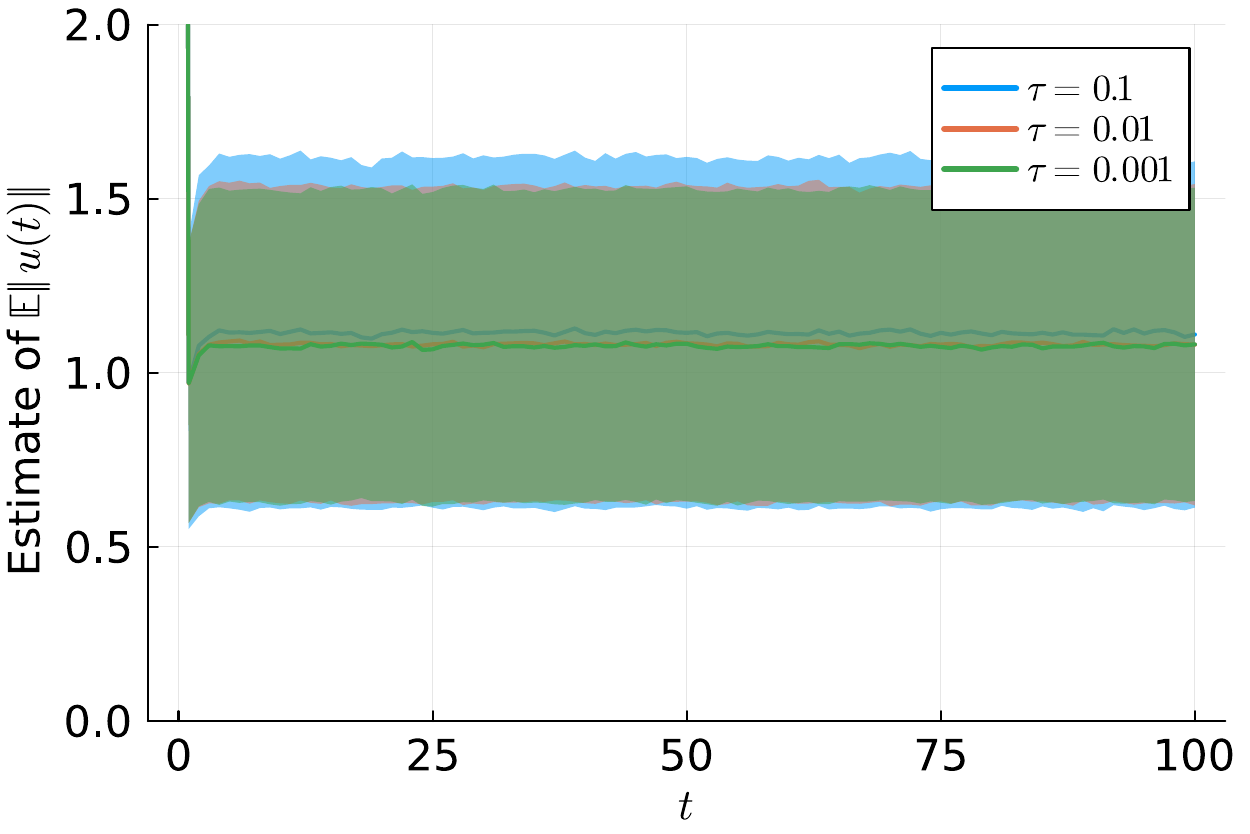}}
    \subfigure[Global gradient taming, \eqref{scheme:taming1}]{\includegraphics[width=6.5cm]{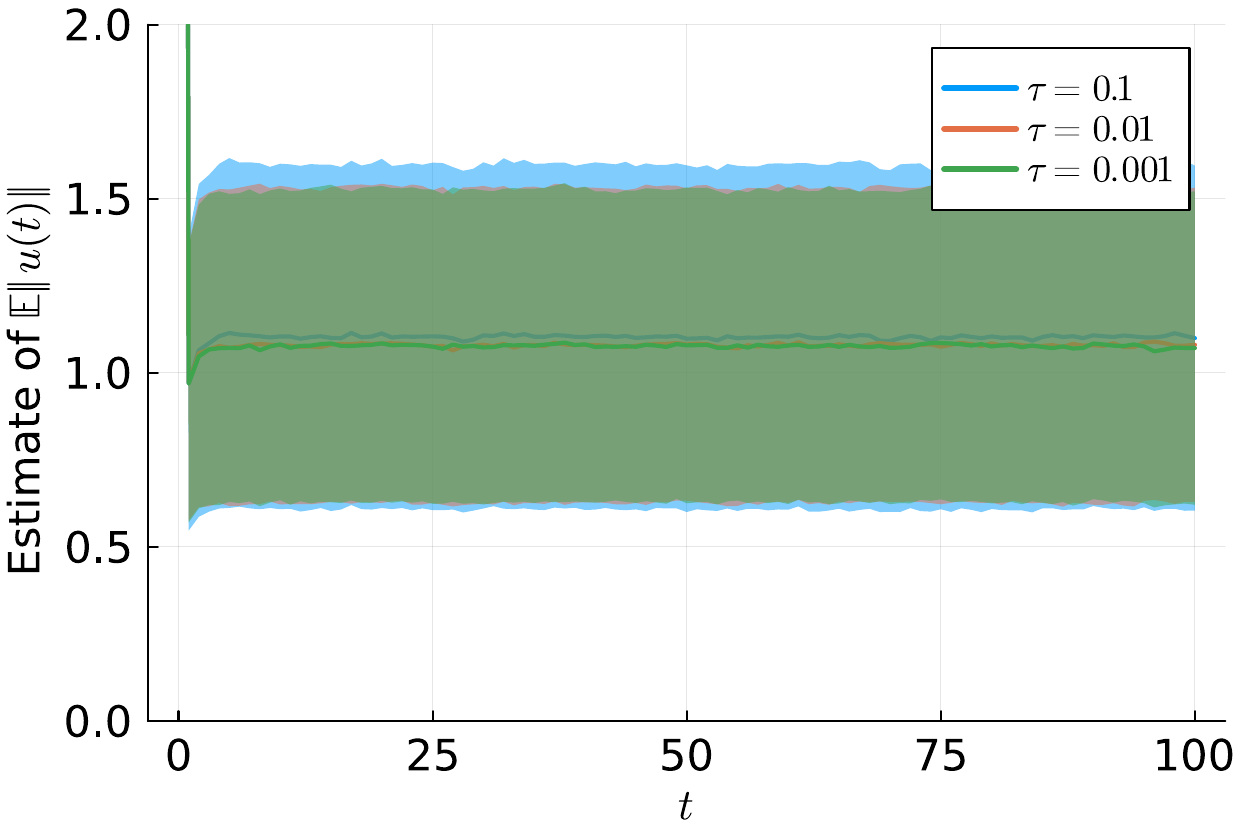}}
    \subfigure[Truncated global taming, \eqref{scheme:tame by norm of f'(u)}]{\includegraphics[width=6.5cm]{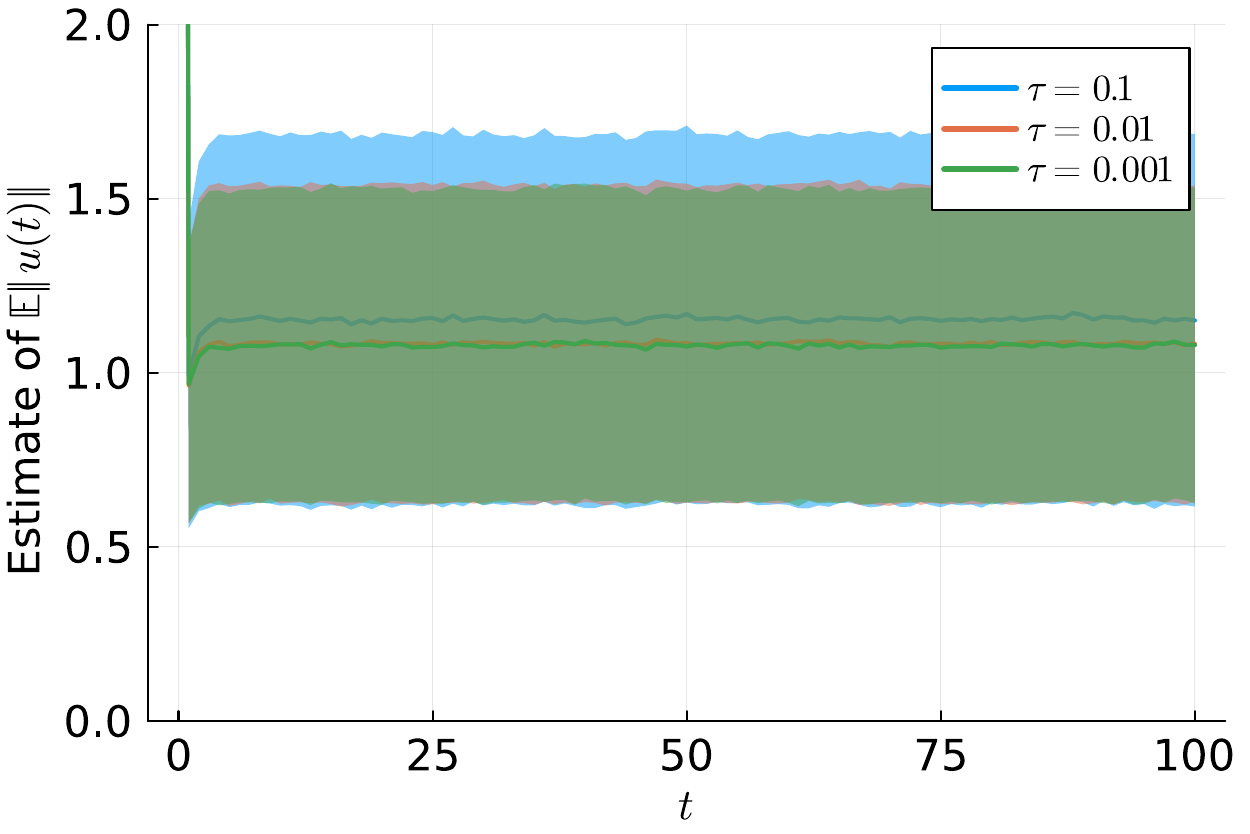}}
    
    \caption{Time series of spectral Galerkin solutions for the first moment for initial condition $u_0(x) = 10\sin(10x)$.  Shaded regions reflect one standard deviation.}
    \label{fig:osc10_10_fft}
\end{figure}

\subsection{Discussion of Numerical Results}
\label{sec:numericdisc}
\subsubsection{Finite Time Blowup without Taming}

By Lemma \ref{lemma: Can explosion lemma},  for any fixed choice of $h$ and $\tau$, for all data sufficiently large, finite time blowup will occur in the SIE scheme.  This explains the results in Figure \ref{fig:hundreddata_fem} (a).  Analogously, based on the calculation from Section \ref{s:periodicblowup}, for initial data $u_0(x) = \text{constant}$, with the constant sufficiently large, $\gtrsim \tau^{-1}$, we should also observe blowup. Indeed, when $u_0(x) = 100$, the simulation is badly behaved for each value of $\tau$; see Figure \ref{fig:hundreddata_fft} (a).



\subsubsection{On Taming by the Gradient Norm}
Next, we revisit global gradient taming, \eqref{scheme:taming1}, for $u_0 = 100$ from Figure \ref{fig:hundreddata_fft} (f), by increasing the vertical limits of the plots; see Figure \ref{fig:hundredzoom} (a).  As now shown, the values are finite for all $\tau$, but they are quite large. 


Our interpretation is as follows.  For initially uniform data, as the solution is continuous in time, on very short time scales, $\nabla f(u)\approx 0$, so there is no taming:
\begin{equation}
    \text{Short Time Scales: }\frac{f(u)}{1 + \tau \|\nabla f(u)\|_2^2}\approx {f(u)}.
\end{equation}
For such a drift, when $u_0 = 100$ we would expect blowup at these values of $\tau$.  Thus, we initially see the $L^\infty$ norm get large; see Figure \ref{fig:hundredzoom} (c).  At the same time, as noise is injected into the system at all modes, 
\begin{equation*}
    \hat{u}_N^{k+1}(j) = \hat{u}_N^k(j) - \tau k^2 \hat{u}^{k+1}_N(j) - 3 \hat{u}_N^k(0)^2 \hat{u}^k_N(j)\tau +\text{Other Terms}.
\end{equation*}
The nonzero modes quickly become large causing growth in $\|\nabla f(u)\|_2^2 = \|3u^2 \nabla u\|_2^2$; see Figure \ref{fig:hundredzoom} (b).  But this is so large that the drift term is obliterated, and
\begin{equation}
    \text{Long Time Scales: }\frac{f(u)}{1 + \tau \|\nabla f(u)\|_2^2}\approx 0,
\end{equation}
and the flow is approximately
\begin{equation*}
    du \approx \Delta u dt + dW.
\end{equation*}
This will struggle to dissipate the spectral energy in mode zero, the mean.


There is {\it not} a corresponding issue in the FEM version of this problem with Dirichlet boundary conditions.  The reason for this is that a Poincaré inequality holds, and $\nabla f(u)$ does not vanish, even if $u_0$ is initially a constant.    Indeed, if we begin with the constant function, $u_0(x) = A$, then for P1 elements,
\begin{equation}
    \|\nabla f(u_0)\|^2 \approx\sum_{i,j}f(A) K_{ij} f(A) = 2 {f(A)^2}/{h}>0
\end{equation}
for stiffness matrix $\mathbf{K} = (K_{ij})$.  For the cubic nonlinearity, this is obviously nonzero for any $A\neq 0$, in contrast to the spectral-Galerkin case.  Indeed, we conjecture that method \eqref{scheme:taming1} is actually suitable for any discretization provided the domain admits a Poincaré. That is to say, we would expect other spatial discretizations of $[0,1]$ with Dirichlet boundary conditions to behave well with taming, while discretiations of $[0,2\pi]$ with periodic boundary conditions might behave poorly, even with taming.




\begin{figure}
    \centering
    \subfigure[]{\includegraphics[width=6.5cm]{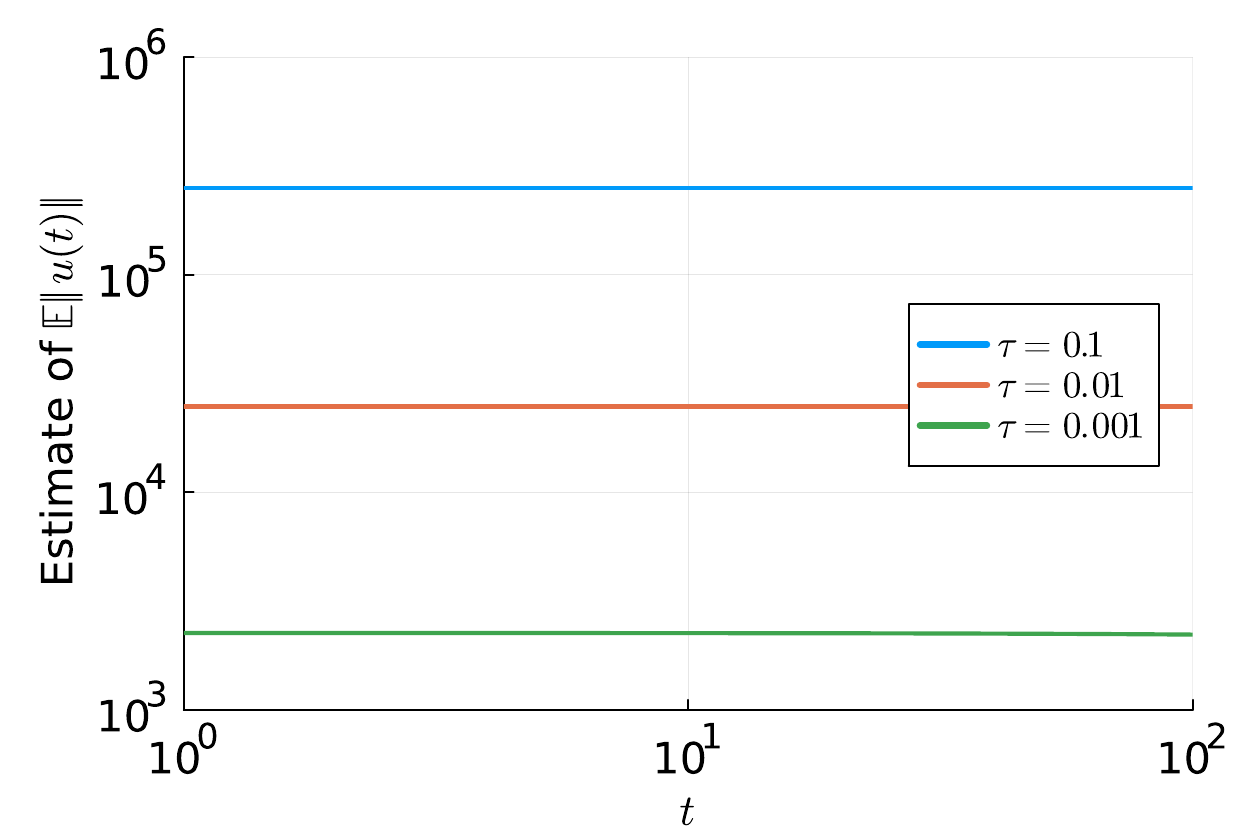}}
    \subfigure[]{\includegraphics[width=6.5cm]{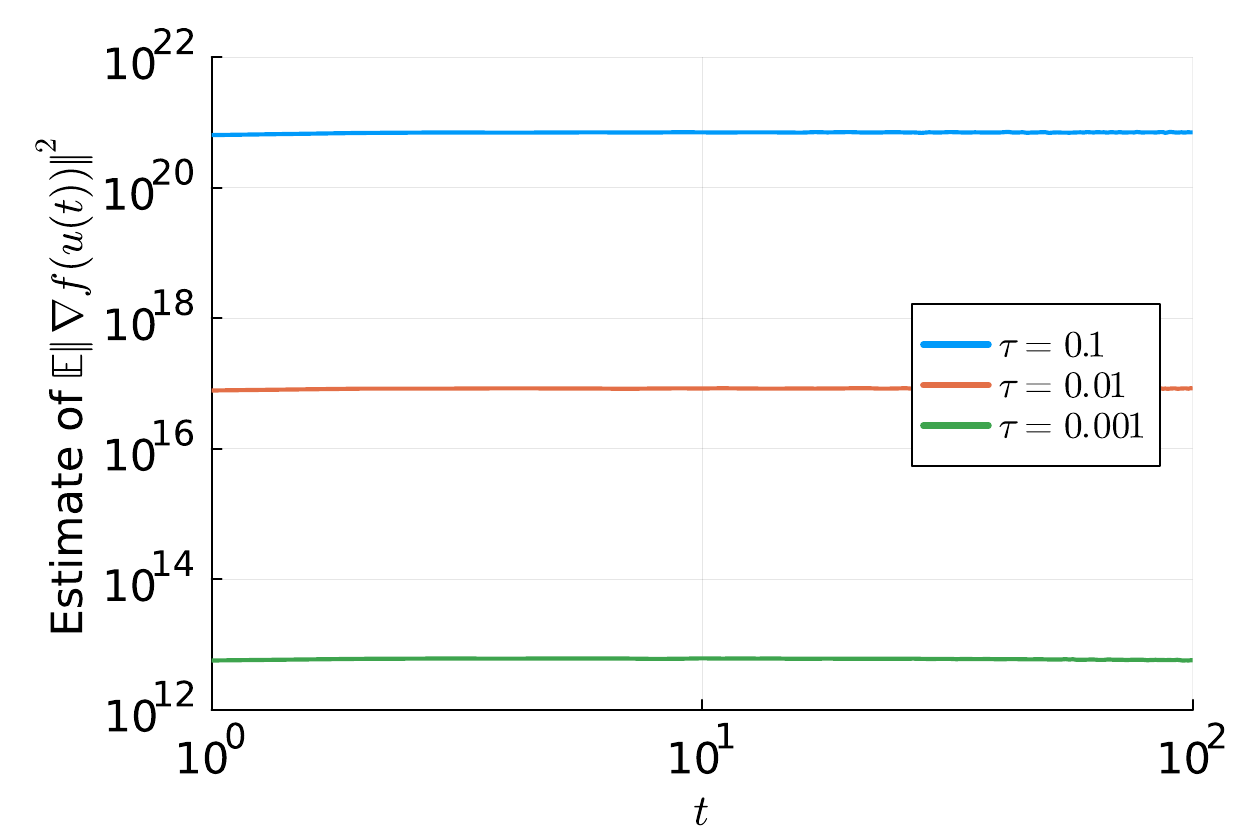}}
    \subfigure[]{\includegraphics[width=6.5cm]{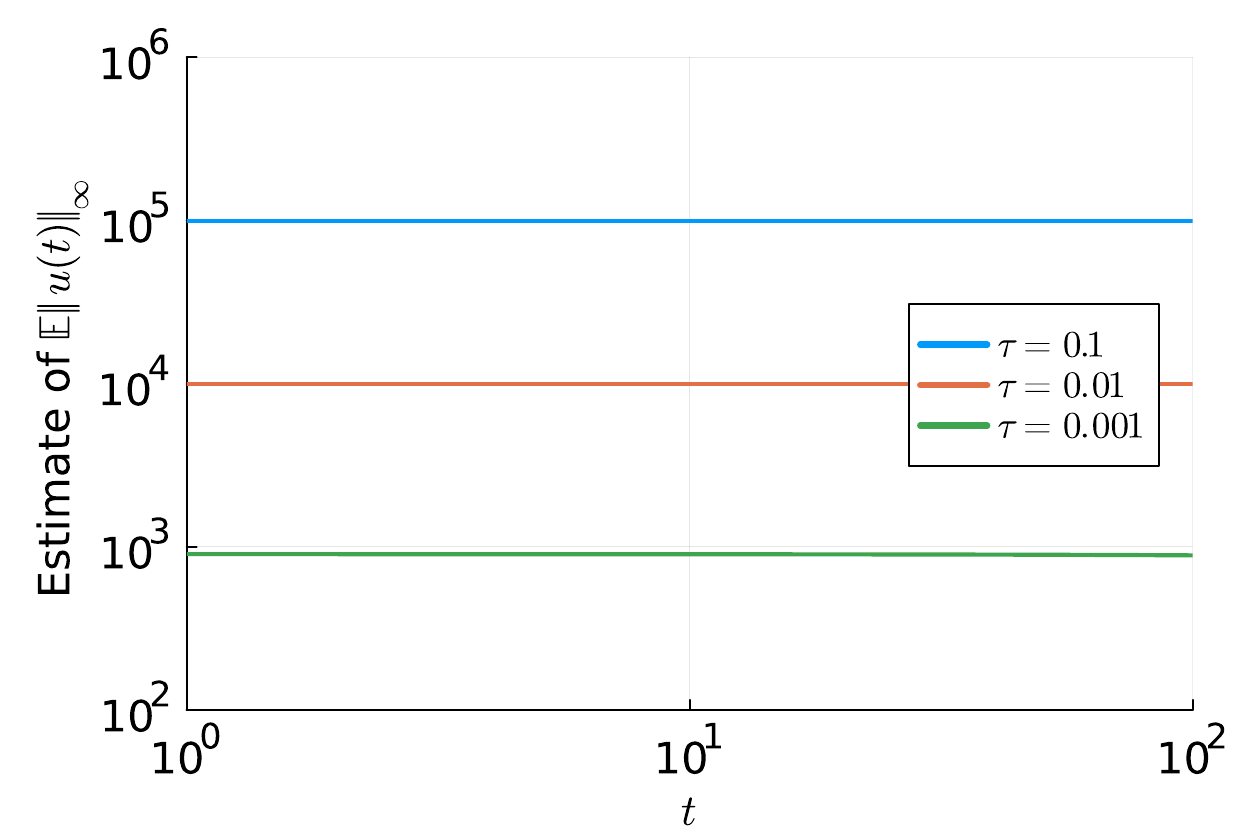}}
    \caption{Observables for Method \eqref{scheme:taming1} with initial condition $u_0(x) =100$ under the spectral Galerkin discretization.}
    \label{fig:hundredzoom}
\end{figure}

\subsubsection{On Transient Dynamics}

Turning to the transient dynamics with the initial condition $u_0(x) = 100$, we see in Figures \ref{fig:hundreddata_short_fem} and  \ref{fig:hundreddata_short_fft} some schemes are more faithful than others.  More variability is present in the spectral Galerkin setting than the FEM case.   Again, the FIE scheme is the most robust, though most costly; the best alternative is \eqref{scheme:taming by abs val f' pointwise standard form}. 


\begin{figure}
    \centering

    \subfigure[FIE, \eqref{scheme:im}]{\includegraphics[width=6.5cm]{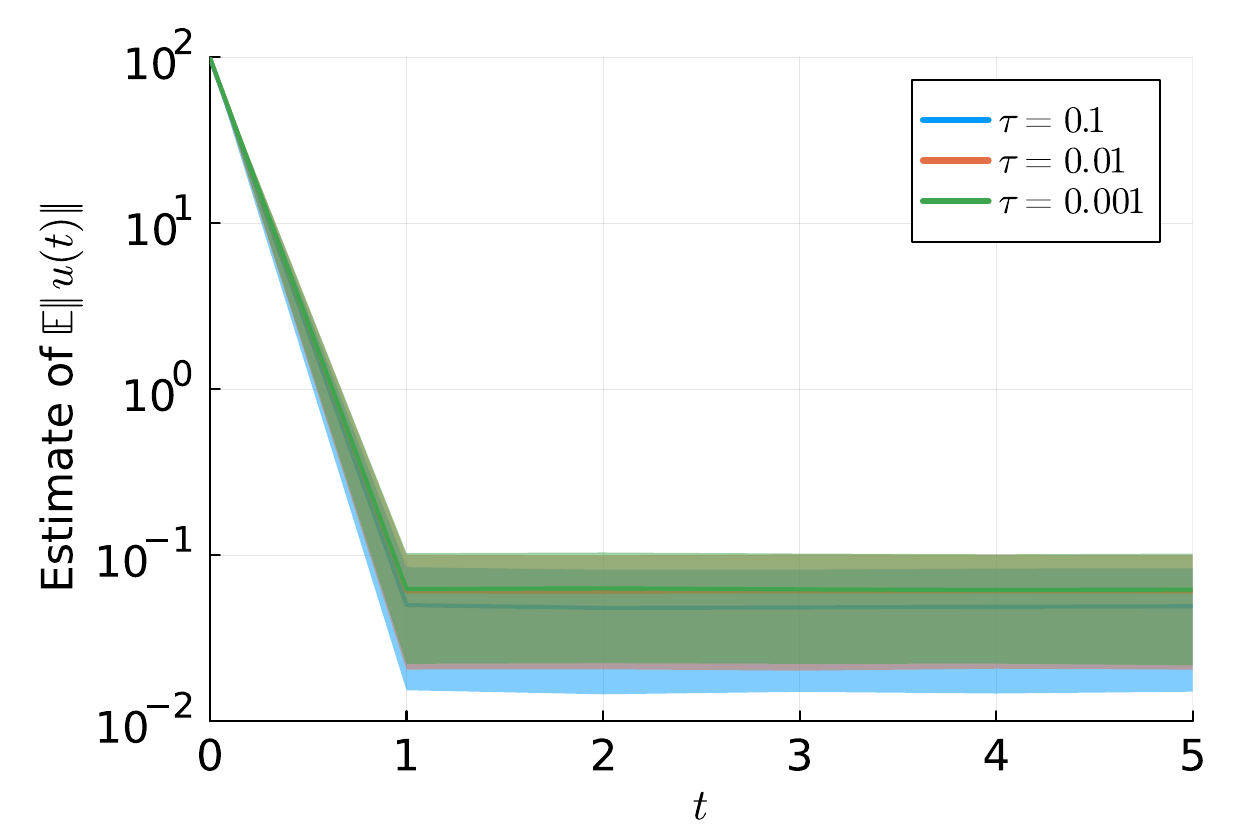}}
    \subfigure[Gy\"ongy's method, \eqref{e:gyongyspde}]{\includegraphics[width=6.5cm]{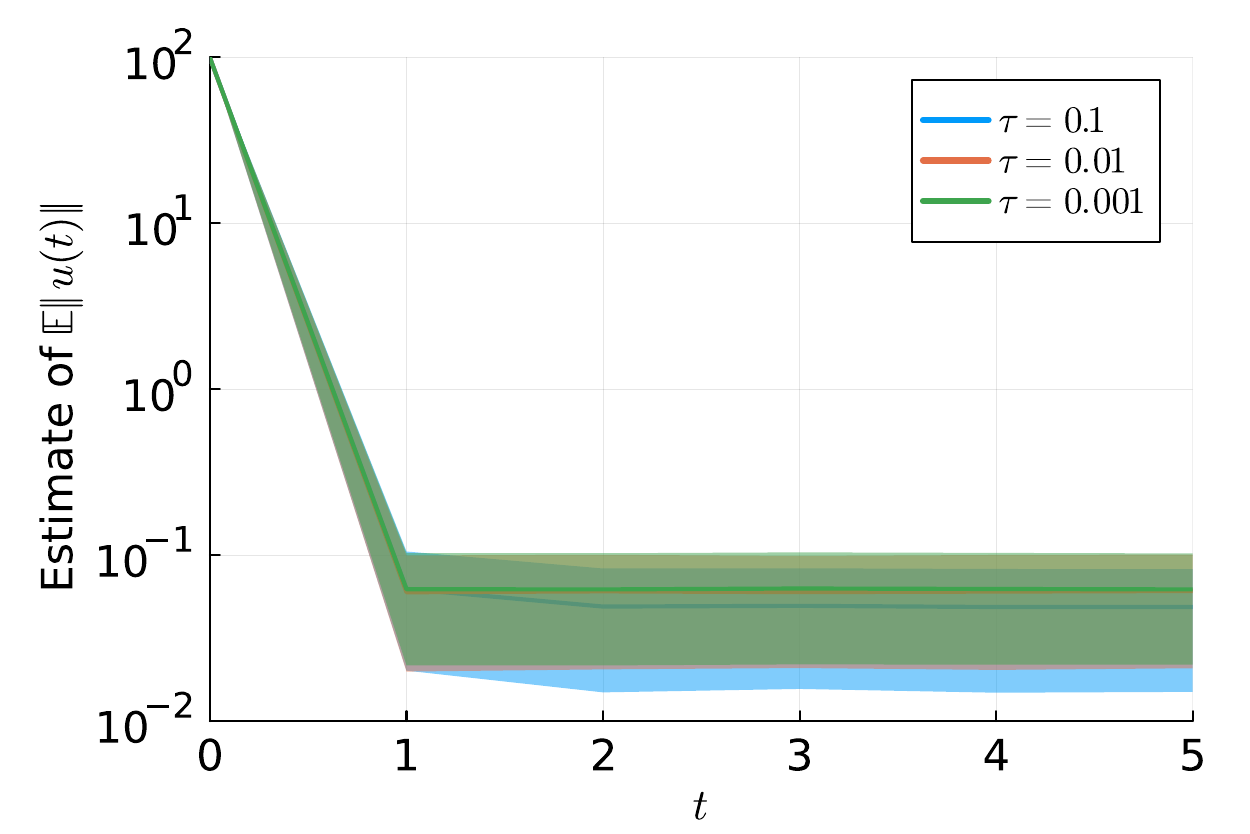}}
    
    \subfigure[Truncated pointwise taming,  \eqref{scheme:taming by abs val f' pointwise standard form}]{\includegraphics[width=6.5cm]{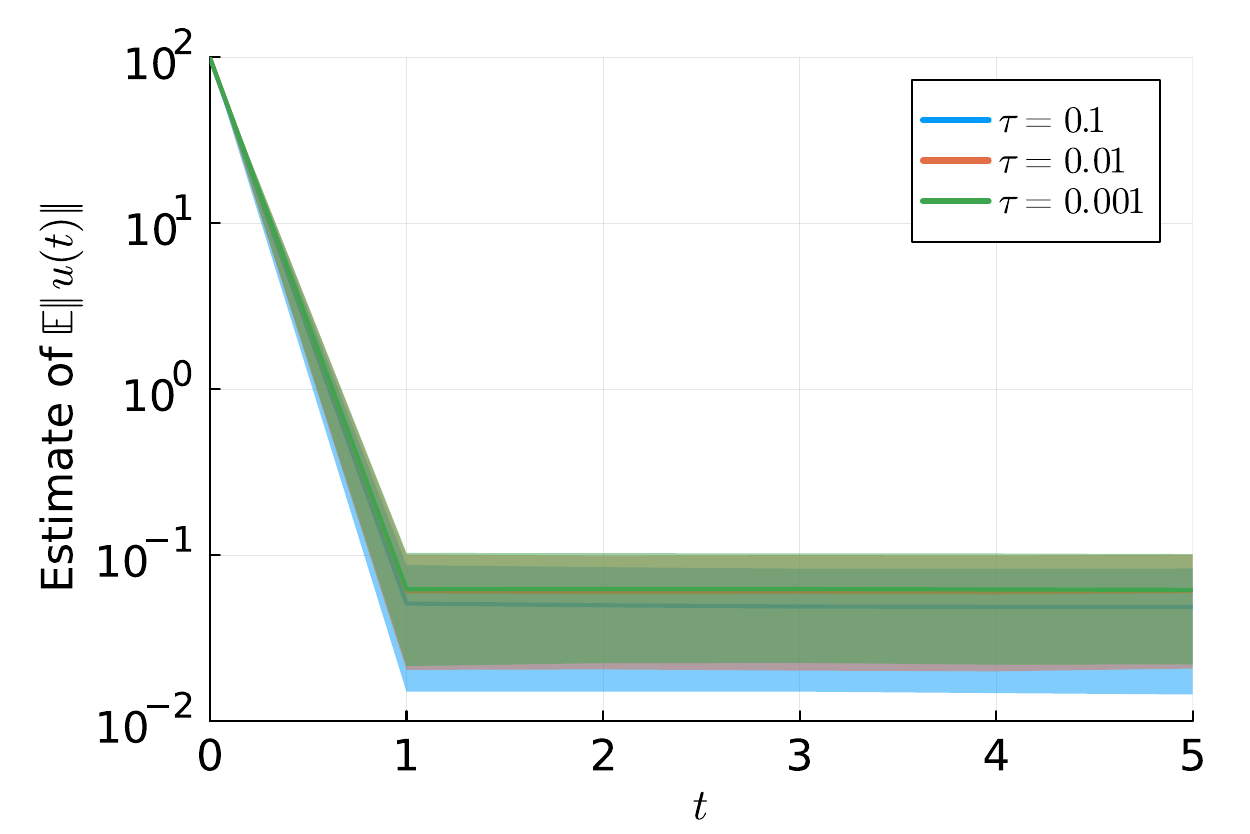}}
    \subfigure[GTEM, \eqref{scheme: GTEM scheme f}]{\includegraphics[width=6.5cm]{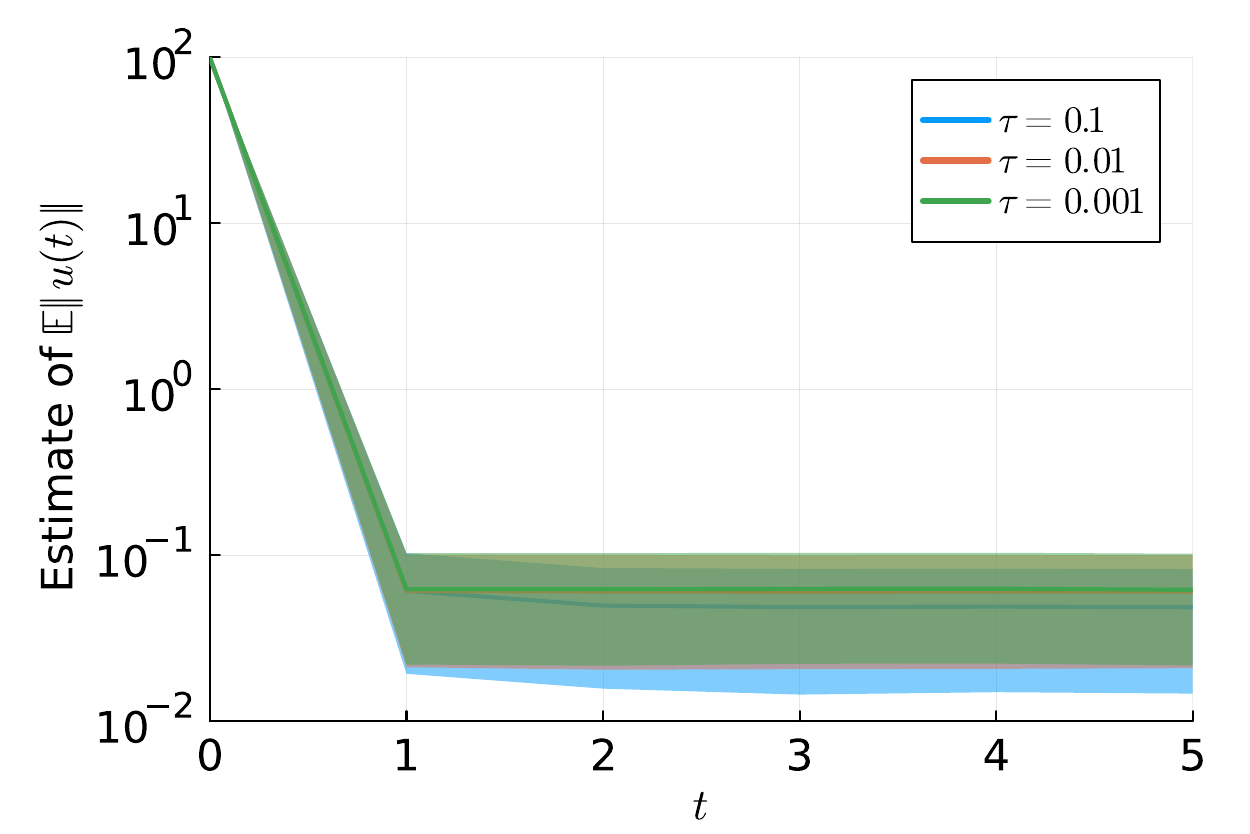}}

    \subfigure[Truncated global taming, \eqref{scheme:tame by norm of f'(u)}]{\includegraphics[width=6.5cm]{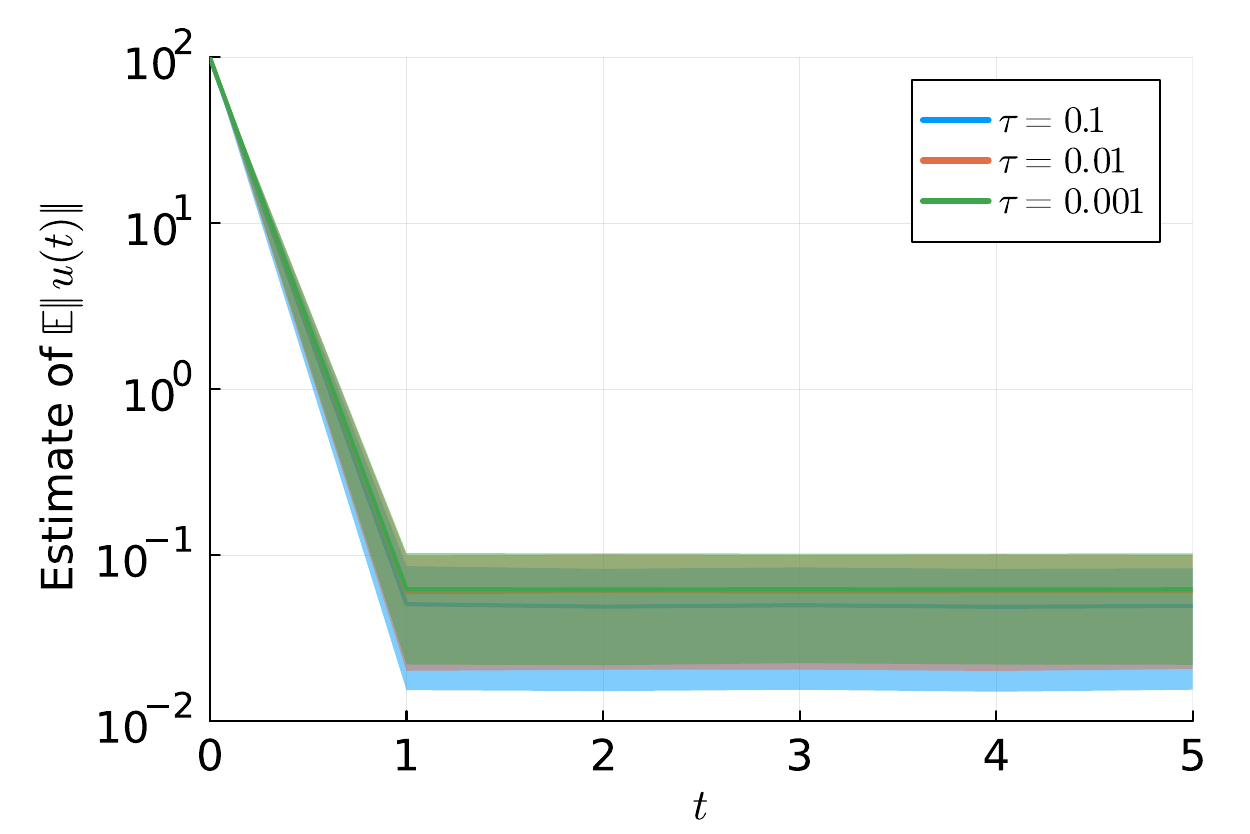}}
    
    \caption{Short time scale data for our methods under the FEM discretization in the case $u_0(x) = 100$.  The fully implicit method and Truncated pointwise taming,  \eqref{scheme:taming by abs val f' pointwise standard form} are the most faithful for larger values of $\tau$.  Shaded regions reflect one standard deviation.}
    \label{fig:hundreddata_short_fem}
\end{figure}


    

    

\begin{figure}
    \centering

    \subfigure[FIE, \eqref{scheme:im}]{\includegraphics[width=6.5cm]{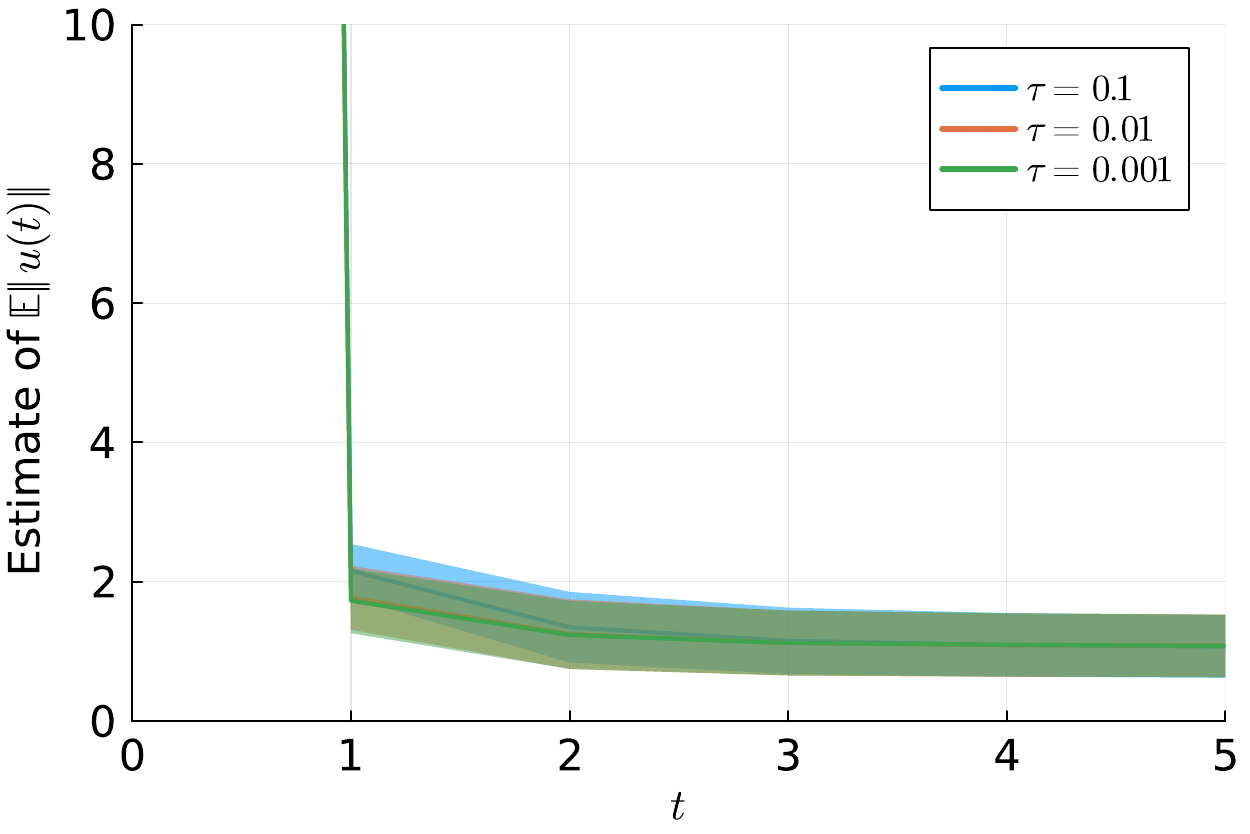}}
    \subfigure[Gy\"ongy's method, \eqref{e:gyongyspde}]{\includegraphics[width=6.5cm]{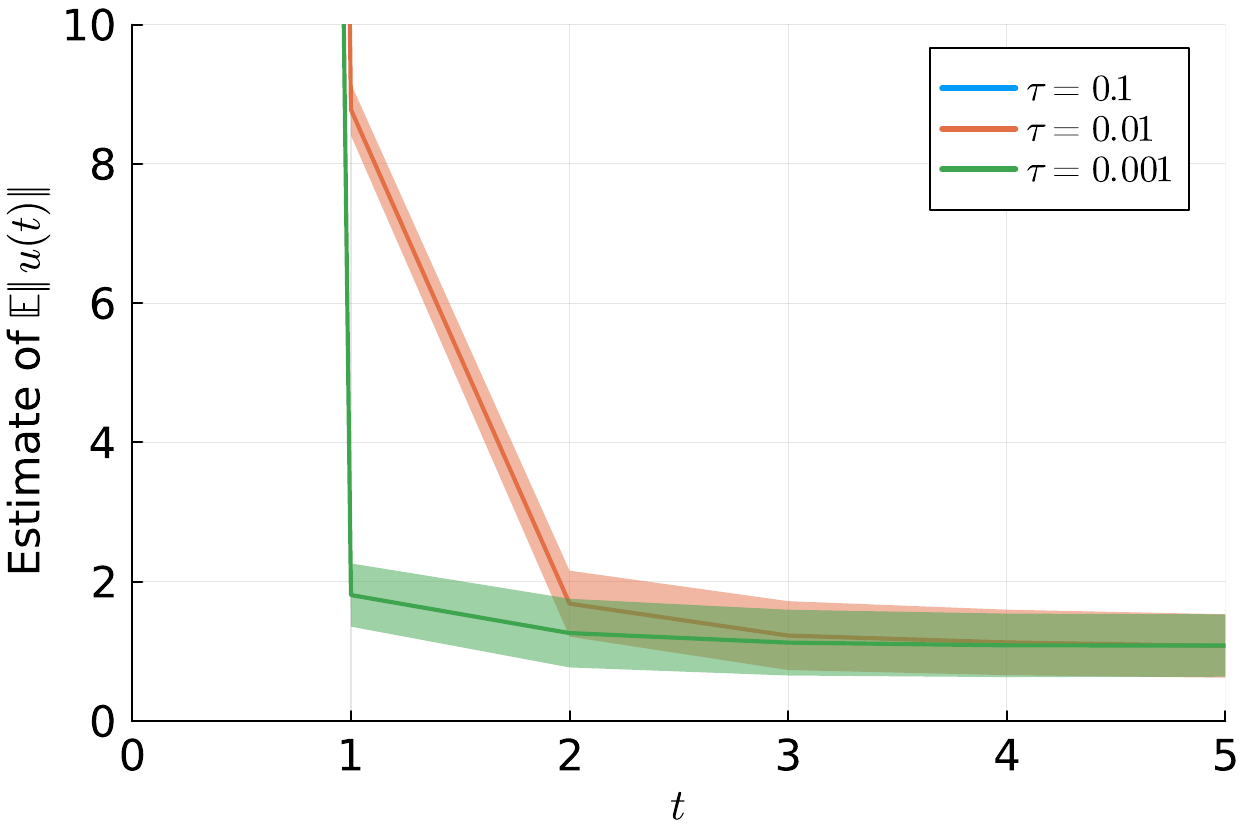}}
    
    \subfigure[Truncated pointwise taming,  \eqref{scheme:taming by abs val f' pointwise standard form}]{\includegraphics[width=6.5cm]{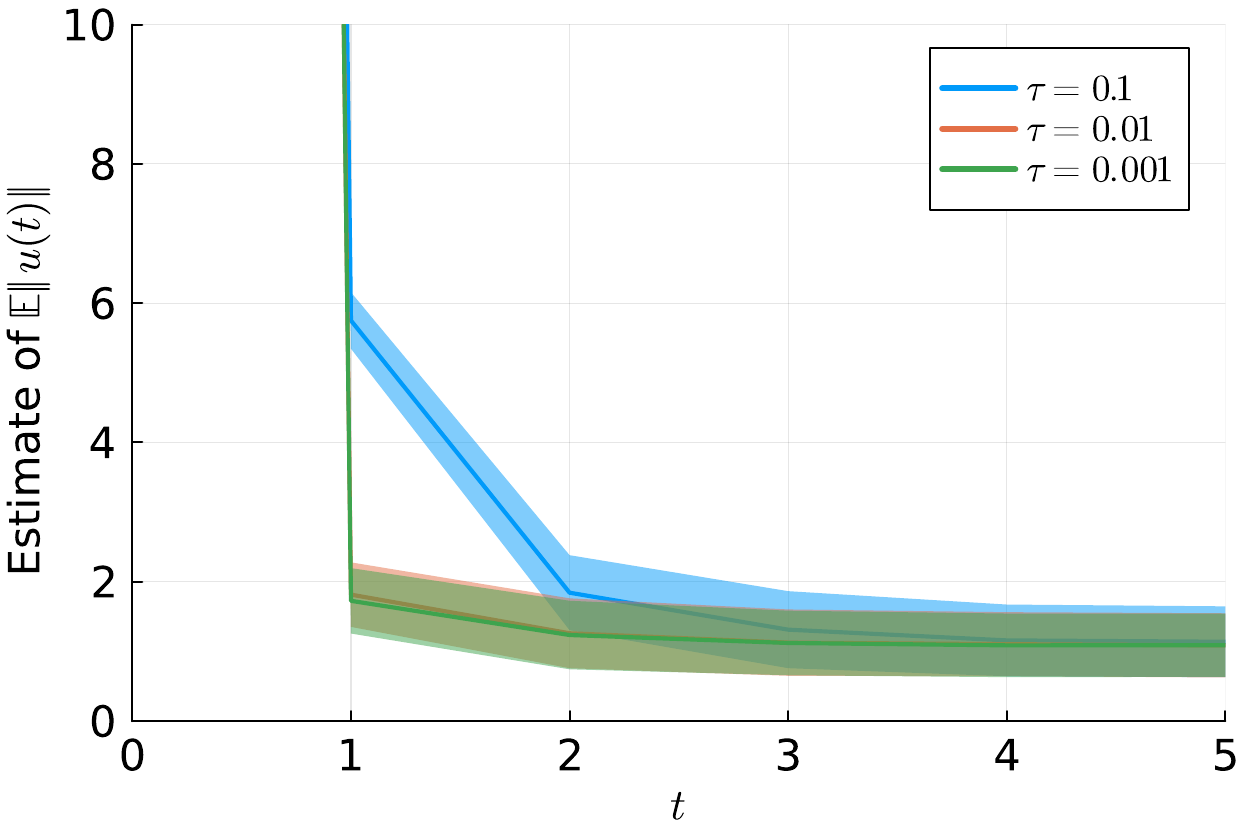}}
    \subfigure[GTEM, \eqref{scheme: GTEM scheme f}]{\includegraphics[width=6.5cm]{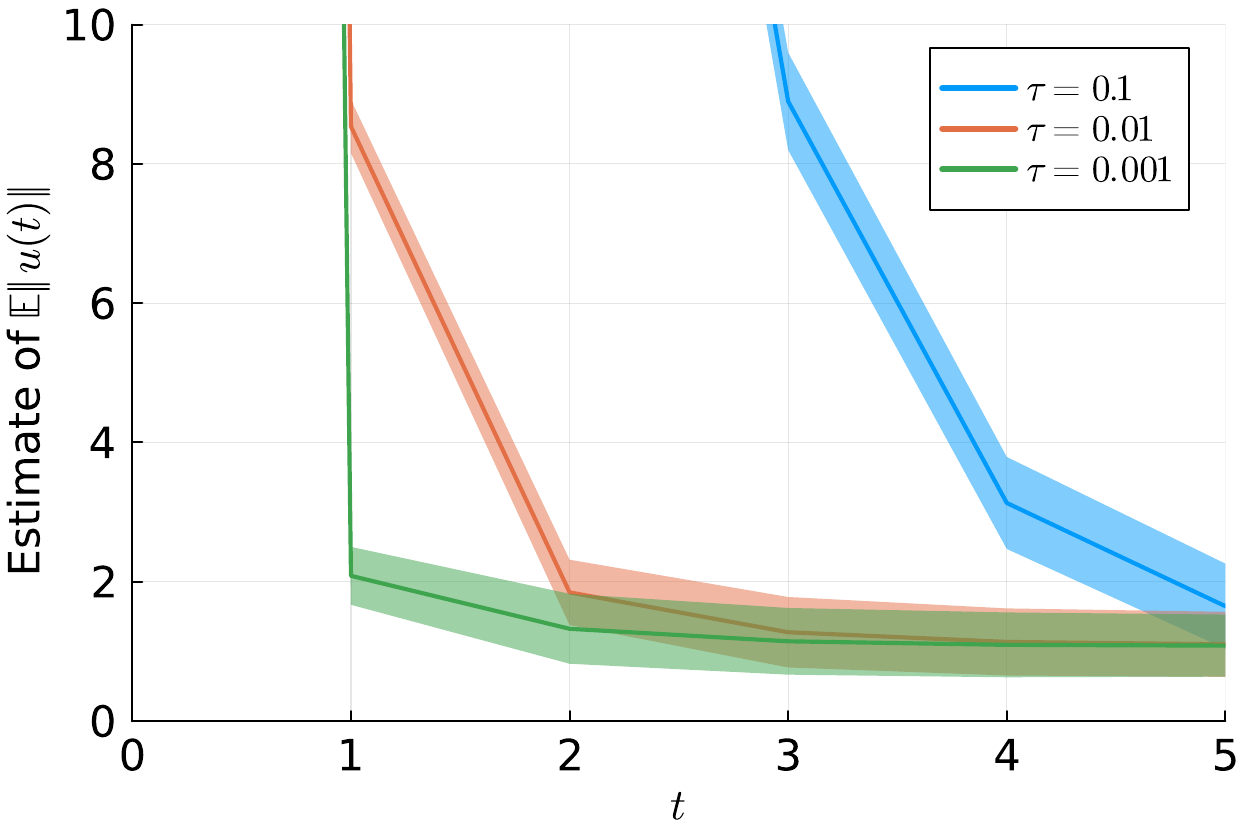}}

    \subfigure[Truncated global taming, \eqref{scheme:tame by norm of f'(u)}]{\includegraphics[width=6.5cm]{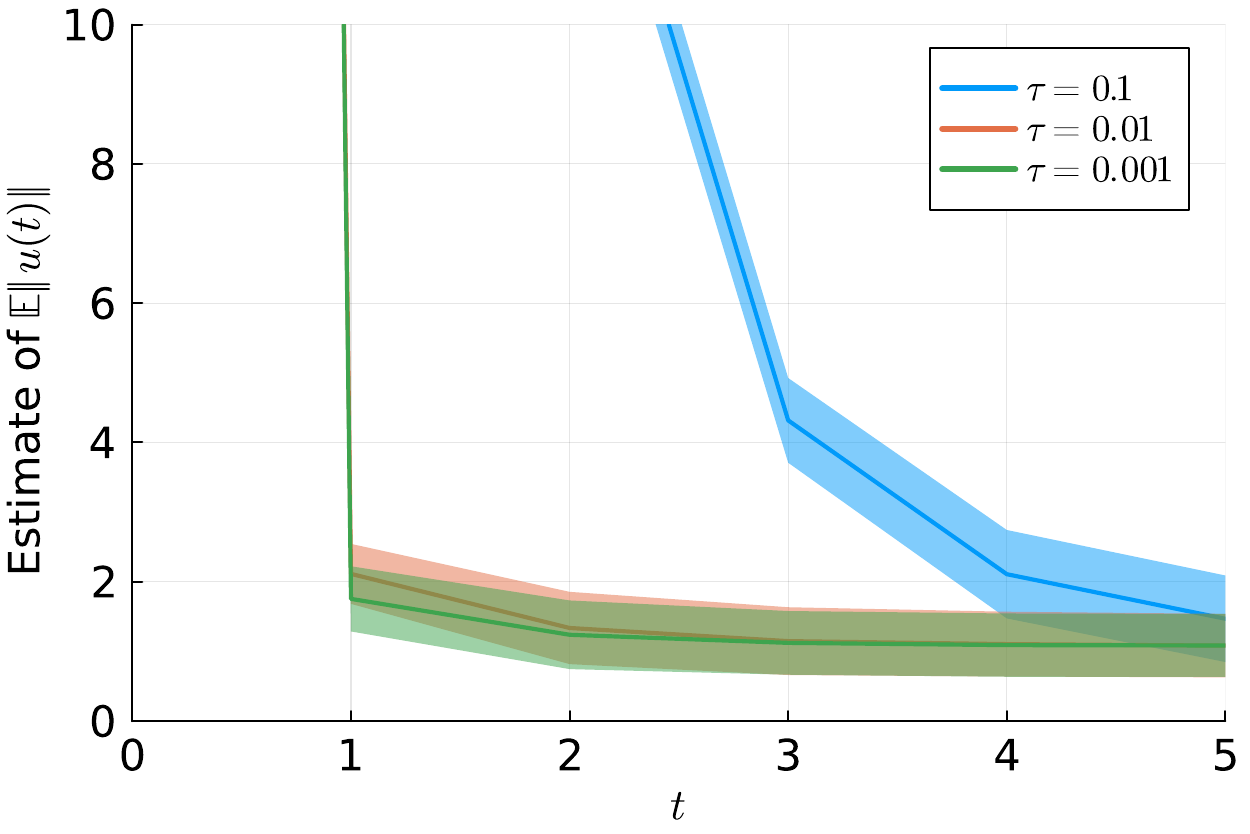}}
    
    \caption{Short time scale data for our methods under the spectral Galerkin discretization in the case $u_0(x) = 100$.  Shaded regions reflect one standard deviation.}
    \label{fig:hundreddata_short_fft}
\end{figure}

\subsubsection{Computational Cost}

Whether we make use of FEM or Spectral-Galerkin, let us take $N$ to be our measurement of the number of degrees of freedom in the system.  This is the number of resolved modes in Spectral-Galerkin, and in the case of an FEM problem in $\R^d$, $N\asymp h^{-d}$.  Furthermore, assume that the FEM method uses local elements, such that the mass and stiffness matrices are sparse, matrix-vector multiplication by either of them can be performed in $\bigo(N)$ operations.

Setting aside the FIE method for a moment, the drift can be evaluated for all methods in either $\bigo(N)$ or $\bigo(N\log N)$ time, assuming we take evaluations of $f$ and $f'$ as costing $\bigo(1)$.  One might initially think that the methods only involving pointwise evaluation, like \eqref{scheme:taming2derivativeabsolute value f'}, would be superior in computational cost, in contrast, to, say method \eqref{scheme:tame by norm of f'(u)}.  Indeed, the norm evaluation itself is $\bigo(N)$.  For FEM, this can be computed as
\begin{equation}
    \|f'(u_h)\|_2^2 = \sum_{ij} f'(u_{h,i})M_{ij} f'(u_{h,j})
\end{equation}
for mass matrix $\mathbf{M} = (M_{ij})$.  As we are using P1 elements, the matrix is sparse, and this is a $\bigo(N)$ computation.  For Spectral-Galerkin in 1D, it corresponds to:
\begin{equation}
    \|f'(u_h)\|_2^2  = \frac{2\pi}{N}\sum_i |f'(u_{h,i})|^2
\end{equation}
In either case, this value only needs to be computed once, per time step, and then can be reused at all evaluations of the modified drift term.  Similar estimates hold for the other quantities.

In our analysis of the tamed methods, there are  constants that appear in inequalities, which we do not fully track.  We attribute some of the variation in performance, particularly in the spectral-Galerkin case, to some methods having larger constants than others, and these constants are likely to change depending on the spatial discretization.  {\it Our recommendation remains to use the pointwise taming method, \eqref{scheme:taming by abs val f' pointwise standard form}, as it shows the best behavior across problems and spatial discretizations, while not requiring root finding, as FIE does.}







\section{Proof of the uniform in time bounds \ref{Uniform in time moment bound}}\label{sec:proof of uniform moment bounds}
In this section we prove that the schemes we consider satisfy assumption \ref{Uniform in time moment bound}. That is, we show that the moments of the schemes we consider are bounded, uniformly in time. 
\subsection{Proof of bound \ref{Uniform in time moment bound} for the fully implicit Euler scheme \eqref{scheme:im}}\label{subsec:proof of H3 for implicit}
\begin{proposition}\label{prop:moment bound for implicit scheme}
(Time-uniform moment bounds for the FIE scheme \eqref{scheme:im}) 
Suppose the nonlinearity $f$ satisfies \eqref{eqn:coercive drift} and \eqref{eqn:assump on gradient of drift} and furthermore that Assumption \ref{assump2} and Assumption \ref{assump3} hold. 
Then the scheme  \eqref{scheme:im} is unconditionally stable, uniformly in time; that is, for $\tau$ small enough,  it satisfies the following bound
\begin{align}
\mathbb{E} \| u_h^{k+1}\|_1^{2n}\leq \mathbb E \|u_h^0\|_1^{2n}+D_n,
\end{align}
for any initial datum $u_0$ satisfying Assumption \ref{assump3},  \footnote{Note that by standard FEM approximation properties if $u^0 \in H^1$ then also its (space) discretization $u^0_h \in H^1$ (cf. \cite[Page 96]{QuarteroniV97}) }
where $D_n$ is a constant  independent of $\tau, h$ and $k$. 
\end{proposition}
\begin{proof}
We first consider the case $n=1$. First, let $k=0$ and choose $\psi=2 u_h^1$ in \eqref{scheme:im}.
This results in
$$
2(u_h^{1}, u_h^1-u_h^0) =  2\tau (\Delta u_h^{1}, u_h^{1})  +
2\tau(f(u_h^{1}), u_h^{1}) + 2(\delta W^1, u_h^{1}) \,.
$$
Writing
\begin{equation}\label{eqn: AA}
2(u_h^{1}, u_h^1-u_h^0) = \|u_h^{1}\|^2+\|u_h^{1}-u_h^{0}\|^2-\|u_h^{0}\|^2
\end{equation}
and using \eqref{eqn:coercive drift} gives
\begin{align}
\| u_h^1\|^2+\|u_h^1-u_h^0\|^2+2\tau\|\nabla u_h^1\|^2 & =\| u_h^0\|^2+2\tau ( f(u_h^1), u_h^1)+2(\delta W^1, u_h^1)\notag\\
&\leq \| u_h^0\|^2+2b_1\tau+2(\delta W^1, u_h^1) \,.\notag
\end{align}
To work on the noise term we use the fact that $2(c,a)\leq 2(c,b)+2c^2 + |a-b|^2/2$, obtaining 
\begin{align}\label{eq:imL2}
\| u_h^1\|^2+\|u_h^1-u_h^0\|^2+2\tau\|\nabla u_h^1\|^2 &\leq \| u_h^0\|^2+2b_1\tau +2( \delta W^1, u_h^0)+2\|\delta W^1\|^2+\frac{1}{2}\| u_h^1- u_h^0\|^2 \,.
\end{align}
{Next, to obtain an estimate on the gradient of $u_h^1$,}
we let $k=0$ and choose $\psi=-2\Delta u_h^1$ in \eqref{scheme:im}. 
 Then, using \eqref{eqn:assump on gradient of drift} and similar calculations as above gives
\begin{align}\label{eq:imH1}
\|\nabla u_h^1\|^2&+\| \nabla u_h^1-\nabla u_h^0\|^2+2\tau\|\Delta u_h^1\|^2=\|\nabla u_h^0\|^2+2\tau (\nabla f(u_h^1), \nabla u_h^1)+2(\nabla\delta W^1, \nabla u_h^1)\notag\\
&\leq \|\nabla u_h^0\|^2+2b_2\tau\|\nabla u_h^1\|^2+2(\nabla\delta W^1, \nabla u_h^1)\notag\\
&\leq \|\nabla u_h^0\|^2+2b_2\tau\|\nabla u_h^1\|^2+2(\nabla \delta W^1, \nabla u_h^0)+2\|\nabla\delta W^1\|^2+\frac{1}{2}\|\nabla u_h^1-\nabla u_h^0\|^2, 
\end{align}
where we observe that,  thanks to Assumption \ref{assump2} , the variable   $\nabla \delta W^1$ 
is a Gaussian random variable with covariance  operator $\tilde Q$(cf. \cite{KovacsLL15}) given by
\begin{align}\label{eq:tildeQ}
\tilde{Q}:=\tau (-\Delta)^{\frac{1}{2}}Q^{\frac{1}{2}} ((-\Delta)^{\frac{1}{2}}Q^{\frac{1}{2}})^*,
\end{align}
so that 
$$\mathbb{E}\|\delta W^j\|_1^2=\tau (\text{Tr}(\tilde{Q})+\text{Tr}(Q)):=\tau |Q|_1^2.
$$
Let us point out that in the above and throughout $\nabla \delta W$ is a shorthand notation for $\sqrt{-\Delta} \delta W$ (i.e. it is scalar-valued, not vector valued, as the short notation would suggest). 
Adding up \eqref{eq:imL2} and \eqref{eq:imH1} and applying the inequality $\lambda_1\|v\|\leq \|\Delta v\|$ (where we recall $\lambda_1$ is the smallest eigenvalue of $-\Delta$ on $\mathcal{O}$) to the term $\|\Delta u_h^1\|^2$, then gives \footnote{The inequality $\lambda_1 \|v\|\leq \|\Delta v\|$ is obtained by first writing $v=\sum_jv_j e_j$, where $\{e_j\}$ is an orthonormal basis of $L^2$, with respect to which $\Delta$ is diagonal, and $v_j = (v, e_j)$. Then, since the eigenvalues of $-\Delta$ (with either Dirichlet or periodic boundary conditions) form a positive, increasing sequence, we have  $\|\Delta v\|^2 = \sum_j v_j^2 \lambda_j^2 \geq \lambda_1^2 \sum_j v_j^2 =\lambda_1^2\|v\|^2$ \,.}
\begin{align}
b_p\|u_h^1\|_1^2&+\frac{1}{2}\| u_h^1- u_h^0\|_1^2
\leq \|u_h^0\|_1^2+2b_1\tau+2(\delta W^1, u_h^0)+2(\nabla \delta W^1, \nabla u_h^0)+2\|\delta W^1\|_1^2,
\end{align}
where $b_p=1+2\tau\min\{\lambda_1^2,1-b_2\}$. That is,
\begin{align}\label{eq:im2}
\|u_h^1\|_1^2+\frac{1}{2b_p}\|u_h^1-u_h^0\|_1^2\leq \frac{1}{b_p}\|u_h^0\|_1^2+\frac{2}{b_p}(\delta W^1, u_h^0)+\frac{2}{b_p}(\nabla\delta W^1, \nabla u_h^0)+\frac{2}{b_p}\|\delta W^1\|_1^2+\frac{2b_1\tau}{b_p}.
\end{align}
Because of our assumption on $b_2$ contained in \eqref{eqn:assump on gradient of drift}, $b_2<1$, so that $b_p$ is strictly greater than one and $1/b_p<1$. Using the Burkholder-Davies-Gundy (BDG) inequality (\cite[Page 18]{Kruse}),  we have
\begin{align}\label{eq:BDG}
\mathbb{E}\|W(t)-W(s)\|_1^{2n}\leq K_{2n} |t-s|^n |Q|_1^n, \ \ n\in \mathbb{N},
\end{align}
with $K_2=1$ and $K_{2n}=\bigg[n(2n-1)\bigg(\frac{2n}{2n-1}\bigg)^{2n-2}\bigg]^n, \ n>1$,  and 
taking expectation of \eqref{eq:im2}, 
 we obtain
\begin{align}
\mathbb{E}\| u_h^1\|_1^2+\frac{1}{2b_p}\mathbb{E}\| u_h^1-u_h^0\|_1^2\leq \frac{1}{b_p}\mathbb{E}\|u_h^0\|_1^2+\frac{2|Q|_1^2+2b_1}{b_p}\tau.
\end{align}
Therefore, by recursion, we obtain
\begin{align}
\mathbb{E}\| u_h^k\|_1^2\leq {\frac{1}{b_p^k}}\mathbb{E}\|u^0\|_1^2+{(2|Q|_1^2+2b_1)}\tau \sum_{j=1}^k\frac{1}{b_p^j}, \quad \mbox{for every } k\geq 1
\end{align}
which gives
\begin{align}
 \sup_{k \in \N}\mathbb{E}\| u_h^{k}\|_1^2\leq\mathbb{E} \|u^0\|_1^2+\bigg(({2|Q|_1^2+2b_1})\tau\bigg)/\bigg(1-\frac{1}{b_p}\bigg).
\end{align}
For the case $n>1$, from \eqref{eq:im2} we have
\begin{align}\label{eq:pn1}
\mathbb{E}\|u_h^1\|_1^{2n}&\leq \mathbb{E}\bigg(\frac{1}{b_p}\|u_h^0\|_1^2+\frac{2b_1\tau}{b_p}+\frac{2}{b_p}((\nabla\delta W^1, \nabla u_h^0)+(\delta W^1, u_h^0))+\frac{2}{b_p}\|\delta W^1\|_1^2\bigg)^{n}\\
&\leq  \mathbb{E}\bigg(\frac{1}{b_p}\|u_h^0\|_1^2+\frac{2b_1\tau}{b_p}+\frac{2}{b_p}(\|\nabla\delta W^1\| \|\nabla u_h^0\|+\|\delta W^1\|\|u_h^0\|)+\frac{2}{b_p}\|\delta W^1\|_1^2\bigg)^{n}\notag\\
&\leq  \frac{1}{b_p^n}\mathbb{E}\bigg(\|u_h^0\|_1^2+2b_1\tau+2\|\delta W^1\|_1 \| u_h^0\|_1+2\|\delta W^1\|_1^2\bigg)^{n}, \notag
\end{align}
where the following Cauchy's inequality is used in the last step:
$$
a_1d_1+a_2d_2\leq \sqrt{a_1^2+a_2^2} \sqrt{d_1^2+d_2^2}, \quad a_1, a_2, d_1, d_2\in \mathbb{R}. 
$$
Hence, completing squares gives 
\begin{align}
\mathbb E\|u_h^1\|^{2n}&\leq \frac{1}{b_p^n} \mathbb E\bigg(\|u_h^0\|_1+\sqrt{2}\|\delta W^1\|_1+\sqrt{2b_1\tau}\bigg)^{2n}  \,. 
\end{align}
Applying the following inequality (see \cite[Lemma B.2]{Angeli}), 
\begin{equation}\label{eq:binom}
\begin{aligned}
&(a+b)^k \leq\left(1+2^{k-1} \alpha\right) a^k+\left(1+2^{k-1} \alpha^{-1}\right) b^k,
\end{aligned}
\end{equation}
which holds for every  $a, b>0, k \in \mathbb{N}, k>1,$  and  $\alpha>0$, we have
\begin{align}
\mathbb E\|u_h^1\|^{2n}_1&\leq \frac{1}{b_p^n} \mathbb E \left[(1+2^{2n-1}\alpha)\|u_h^0\|_1^{2n}+(1+2^{2n-1}\alpha^{-1})(\sqrt{2}\|\delta W^1\|_1+\sqrt{2b_1\tau})^{2n}\right].
\end{align}
Choosing $\alpha=2^{-2n+2}\min\{\lambda_1,1-b_2\}\tau$ such that
$$
1+2^{2n-1}\alpha=b_p,
$$
we obtain
\begin{align}
\mathbb E\|u_h^1\|_1^{2n}&\leq \frac{1}{b_p^{n-1}}\mathbb E \|u_h^0\|_1^{2n}+\frac{1+2^{2n-1}\alpha^{-1}}{b_p^n}\mathbb E(\sqrt{2}\|\delta W^1\|_1+\sqrt{2b_1\tau})^{2n}\notag\\
&\leq \frac{1}{b_p^{n-1}}\mathbb E \|u_h^0\|_1^{2n}+\frac{1+2^{2n-1}\alpha^{-1}}{b_p^n}(1+2^{2n-1})(2^n\mathbb E\|\delta W^1\|_1^{2n}+(2b_1)^n\tau^n)\notag\\
&\leq \frac{1}{b_p^{n-1}}\mathbb E \|u_h^0\|_1^{2n}+\frac{1+2^{2n-1}\alpha^{-1}}{b_p^n}(1+2^{2n-1})(2^n K_{2n}|Q|_1^n\tau^n+(2b_1)^n\tau^n).
\end{align}
Here, we have used \eqref{eq:binom} with $\alpha=1$ and then \eqref{eq:BDG} .  
Looking at the last addend in the above, we now observe that there exists a constant $\hat D_n$ so that
\begin{align}
\frac{1+2^{2n-1}\alpha^{-1}}{b_p^n}(1+2^{2n-1})(2^n K_{2n}|Q|_1^n\tau^n+(2b_1)^n\tau^n)\leq \hat D_n\tau^{n-1}, 
\end{align}
where $\hat D_n$ depends on $n$ and $|Q|_1$ only. Consequently,
\begin{align}
\mathbb{E}\|u_h^1\|_1^{2n}\leq \frac{1}{b_p^{n-1}}\mathbb{E}\| u_h^0\|_1^{2n}+\hat D_n\tau^{n-1}. 
\end{align}
By recursion, we obtain
\begin{align}
\mathbb{E}\| u_h^k\|_1^{2n}\leq \frac{1}{(b_p)^{k(n-1)}}\mathbb{E}\| u_h^0\|_1^{2n}+D_n \tau^{n-1} \sum_j \frac{1}{(b_p^{n-1})^j}, \ \ n\geq 2,
\end{align}
for every $k$. This concludes the proof.
\end{proof}

\subsection{Proof of bound \ref{Uniform in time moment bound} for the scheme \eqref{scheme:taming1}}\label{subsec: tamed H3}

\begin{proposition}\label{them: uniform in time moment bound for tamed}
(Time-uniform moment bound for the tamed scheme \eqref{scheme:taming1}).  Suppose the nonlinearity $f$ satisfies \eqref{eqn:coercive drift} and \eqref{eqn:assump on gradient of drift} and furthermore that Assumption \ref{assump2} and Assumption \ref{assump3} hold. Then there exists a constant  $D_n>0$, independent of $k$ and $h$,  such that for $\tau$ small enough, the following bound holds
\begin{align}
\mathbb{E} \| u_h^{k+1}\|_1^{2n}\leq  \mathbb{E} \| u_h^{0}\|_1^{2n}+D_n \,.
\end{align}
\end{proposition}
\begin{proof}
This proof is in places similar to the proof of Proposition \ref{prop:moment bound for implicit scheme}, so we partly give a scketch and provide more detail when the two proofs differ. 
We first prove the case $n=1$.  
To this end,  let $k=0$ and $\psi=2u_h^{k+1}$ in \eqref{scheme:taming1}. Using \eqref{eqn: AA} and \eqref{eqn:coercive drift}, we then obtain 
\begin{align}\label{first block}
&\|u_h^1\|^2+\|u_h^1-u_h^0\|^2+2\tau\|\nabla u_h^1\|^2\notag\\
&\ =\|u_h^0\|^2+\frac{2\tau(f(u_h^{0}), u_h^0)}{1+\alpha\tau \|\nabla f(u_h^0)\|^2}
+\frac{2\tau(f(u_h^{0}), u_h^1-u_h^0)}{1+\alpha\tau \|\nabla f(u_h^0)\|^2}+2(\delta W^1, u_h^0)+2(\delta W^1, u_h^1-u_h^0) \,.\notag
\end{align}
Then, by Young inequality with $\theta=4$ (which we use twice in the next step), we get
\begin{align}
&\|u_h^1\|^2+\|u_h^1-u_h^0\|^2+2\tau\|\nabla u_h^1\|^2\notag\\
& \leq\|u_h^0\|^2+\frac{2\tau b_1}{1+\alpha\tau \|\nabla f(u_h^0)\|^2}
+\frac{1}{4}\|u_h^1-u_h^0\|^2+\frac{4\tau^2\|f(u_h^{0})\|^2}{(1+\alpha\tau\|\nabla f(u_h^0)\|^2)^2}\notag\\
&+2(\delta W^1, u_h^0)+\frac{1}{4}\|u_h^1-u_h^0\|^2+4\|\delta W^1\|^2\notag\\
&\ \leq \|u_h^0\|^2+\frac{2\tau b_1}{1+\alpha\tau\|\nabla f(u_h^0)\|^2}+\frac{1}{2}\|u_h^1-u_h^0\|^2
+\frac{4\tau^2\|f(u_h^{0})\|^2}{(1+\alpha\tau\|\nabla f(u_h^0)\|^2)^2}+2(\delta W^1, u_h^0)+4\|\delta W^1\|^2.
\end{align}
Next, using the inequality $(1+c)^2\geq 4c$ (which holds for any $c\in\R$) and the Poincare inequality $\|\nabla f(u_h^0)\|\geq c_{PC}\|f(u_h^0)\|$, we obtain
\begin{equation}
\label{e:PC1}    
\frac{4\tau^2\|f(u_h^0)\|^2}{(1+\alpha\tau\|\nabla f(u_h^0)\|^2)^2}\leq \frac{\tau\|f(u_h^0)\|^2}{\alpha \|\nabla f(u_h^0)\|^2}\leq \frac{\tau}{\alpha c_{PC}^2}. 
\end{equation}
Hence,
we have
\begin{align}\label{eq:sta1}
\|u_h^1\|^2+2\tau\|\nabla u_h^1\|^2+\frac{1}{2}\|u_h^1-u_h^0\|^2
&\leq \|u_h^0\|^2
+2(\delta W^1, u_h^0)+\bigg(2b_1+\frac{1}{\alpha c_{PC}^2}\bigg)\tau+4\|\delta W^1\|^2.
\end{align}
Now let $k=0$ and $\psi=-2\Delta u_h^{k+1}$ in \eqref{scheme:taming1}, and we obtain
\begin{align}
\|\nabla u_h^1\|^2+\|\nabla u_h^1-\nabla u_h^0\|^2+2\tau\|\Delta u_h^1\|^2 & =\|\nabla u_h^0\|^2+\frac{2\tau(\nabla f(u_h^{0}), \nabla u_h^0)}{1+\alpha\tau \|\nabla f(u_h^0)\|^2}\notag\\
&+\frac{2\tau(\nabla f(u_h^{0}), \nabla u_h^1-\nabla u_h^0)}{1+\alpha\tau \|\nabla f(u_h^0)\|^2}+2(\nabla\delta W^1, \nabla u_h^1).
\end{align}
With calculations similar to those in the above we then land with
\begin{align}\label{eq:sta2}
&\|\nabla u_h^1\|^2+2\tau\|\Delta u_h^1\|^2+\frac{1}{2}\|\nabla u_h^1-\nabla u_h^0\|^2\notag\\
&\leq \|\nabla u_h^0\|^2
+2(\nabla \delta W^1, \nabla u_h^0)+\frac{2b_2\tau \|\nabla u_h^0\|^2}{1+\alpha\tau\|\nabla f(u_h^0)\|^2}
+\frac{4\tau^2\|\nabla f(u_h^0)\|^2}{(1+\alpha\tau\|\nabla f(u_h^0)\|^2)^2}+4\|\nabla \delta W^1\|^2\notag\\
&\leq (1+2b_2\tau)\|\nabla u_h^0\|^2
+2(\nabla\delta W^1, \nabla u_h^0)+4\|\nabla\delta W^1\|^2+\frac{\tau}{\alpha} \,.
\end{align}
The rest is similar to the proof of Proposition \ref{prop:moment bound for implicit scheme}, the main difference is in tracking the constants. 
Indeed, summing \eqref{eq:sta1} and \eqref{eq:sta2} up, and using the Poincare's inequality $\|\nabla u_h^1\|\geq c_{PC}\| u_h^1\|$  and $\lambda_1\|u_h^1\|\leq \|\Delta u_h^1\|$ leads to
\begin{align}\label{eq:pnH2}
b_p\|u_h^1\|_1^2+\frac{1}{2}\|u_h^1-u_h^0\|_1^2\leq (1+2b_2\tau)\|u_h^0\|_1^2+2(\delta W^1, u_h^0)+2(\nabla \delta W^1, \nabla u_h^0)+4\|\delta W^1\|_1^2+c_1\tau,
\end{align}
where 
\begin{align*}
&b_p=1+2\tau\min\{1, \lambda_1^2\}, \\
 &c_1:=2b_1+\frac{1}{\alpha c_{PC}^2}+\frac{1}{\alpha}.
 \end{align*}
Taking expectation on both sides of \eqref{eq:pnH2} and noting that both $\mathbb{E}(\delta W^1, u_h^0)$ and $\mathbb{E}(\nabla \delta W^1, \nabla u_h^0)$ vanish,  
leads to
\begin{align}
b_p\mathbb{E}\|u_h^1\|_1^2+\frac{1}{2}\mathbb{E}\|u_h^1-u_h^0\|_1^2
&\leq \mathbb{E}(1+2b_2\tau)\|u_h^0\|_1^2+4\mathbb{E}\|\delta W^1\|_1^2+c_1\tau\notag\\
&\leq (1+2b_2\tau)\mathbb{E}\|u_h^0\|_1^2+\big(4|Q|_1^2+c_1\big)\tau.
\end{align}
That is,
\begin{align}
&\mathbb{E}\|u_h^1\|_1^2
\leq \frac{1+2b_2\tau}{b_p}\mathbb{E}\|u_h^0\|_1^2+\frac{1}{b_p}\big(4|Q|_1^2+c_1\big)\tau.
\end{align}
Now recall  that by Assumption \ref{assump1}, $b_2< \min\{1, \lambda_1^2\}$. Hence
\begin{align}\label{eqn: CC}
\frac{(1+2b_2 \tau)}{b_p}<1 \,. 
\end{align}
So, again be recursion, we obtain
\begin{align}
\sup\limits_{k\in \N}\mathbb{E} \| u_h^k\|_1^2\leq \mathbb E\|u_0\|_1^2+\frac{\big(4|Q|_1^2+c_1\big)}{\min\{\lambda_1^2, 1\}}. 
\end{align}
Next, we consider the case $n>1$.  
Taking $n$-th power on both sides of \eqref{eq:pnH2} and then  taking expectation yields
\begin{align}\label{eq:pn2}
&\mathbb{E}\| u_h^1\|_1^{2n}
\leq \mathbb E\bigg(\frac{1+2b_2\tau}{b_p}\|u_h^0\|_1^2+\frac{2(\delta W^1, u_h^0)}{b_p}+\frac{2(\nabla W^1,\nabla u_h^0)}{b_p}+\frac{4\|\delta W^1\|_1^2+c_1\tau}{b_p}\bigg)^n.
\end{align}
Comparing \eqref{eq:pn1} with \eqref{eq:pn2}, and  following the same steps as in the proof of \eqref{eq:pn1} we arrive at
\begin{align}
\mathbb{E}\| u_h^k\|_1^{2n}\leq \bigg(\frac{1+2b_2\tau}{b_p}\bigg)^{k(n-1)}\|u_h^0\|_1^{2n}+D_n \tau^{n-2}, \ \ n\geq 2,
\end{align}
for every $k$, where $D_n$ depends on $n$ and $|Q|_1$ and $c_1$. So, after recalling \eqref{eqn: CC},  a recursion argument again allows to conclude. 
\end{proof}

{\bf Acknowledgments. } C. Huang thanks the support of NSFC grant 12271457. M. Ottobre  gratefully acknowledges the support of
the Leverhulme grant RPG–2020–095 and of the EPSRC Standard grant EP Z534225/1. G. Simpson  was supported in part by NSF grant DMS–2111278 and ARO grant  W911NF-19-1-0243. Work
reported here was run on hardware supported by Drexel’s University Research Computing
Facility.

\appendix
\section{Proofs of the finite time error bounds \ref{Finite time estimate}}\label{sec: proofs of finite time error bounds}
Throughout this section, for any $0<\nu\leq 2$, $\| \cdot \|_{\dot{H}^{\nu}}:=\|(-\Delta)^{\frac{\nu}{2}}\cdot\|$ denotes norm in $\dot{H}^{\nu}$. Before starting the proofs, we make some preliminary observations that will be used later. Firstly,  as a consequence of \eqref{assump eqn: growth of f}
for any $u,v \in L^\infty$ there exists a constant $C>0$ such that\footnote{To see this just use Taylor's theorem and write $f(x)-f(y) = f'(y)(x-y)+ f''(x\theta+y)(x-y)^2$, $\theta \in [0,1]$, and then use the growth of $f'$ and $f''$}

\begin{equation}\label{eqn: f polynomial}
\|f(u)-f(v)\|\leq C \|u-v\|(1+\|u\|_\infty^{p+1}+\|v\|_\infty^{p+1}) \,.
\end{equation}
Let now  
\begin{align}
&S(t)=e^{t\Delta}, \,\, S_{h,\tau}=(I-\tau \Delta_h)^{-1}, \notag\\
&S_{h,\tau}(t):=S_{h,\tau}^j, \quad \text { if } t \in\left[t_{j-1}, t_j\right) \text { for } j=1,2, \ldots \notag\\
&E_{h,\tau}(t)=S(t)-S_{h,\tau}(t)P_h,
\end{align}
where $\Delta_h$ is defined in \eqref{def:dislap}.
We have the following error estimates for $E_{h,\tau}(t)$ (cf. \cite[Lemma 3.12 and 3.13]{Kruse}).
\begin{lemma}\label{lem:serrop}
Let $t>0$. 
\begin{enumerate}
\item For $0\leq \nu\leq 1$,  $\|(-\Delta)^\nu S(t)\|\leq Ct^{-\nu},\;\;\|(-\Delta)^{-\nu}(S(t)-\mathbb{I})\|\leq Ct^\nu.$
\item For $0\leq \nu\leq \mu\leq 2$ and $x\in \dot{H}^\nu$, 
\begin{align}
\|E_{h,\tau}(t)x\|\leq C(h^\mu+\tau^{\frac{\mu}{2}})t^{-\frac{\mu-\nu}{2}}\|x\|_{\dot{H}^{\nu}}. 
\end{align}
\item For $0\leq \mu\leq 1$ and $x\in \dot{H}^\mu$,
\begin{align}
\bigg(\int_0^t\|E_{h,\tau}(r)x\|^2dr\bigg)^{1/2}\leq C(h^{1+\mu}+\tau^{\frac{1+\mu}{2}})\|x\|_{\dot{H}^\mu}.
\end{align}
\end{enumerate}
\end{lemma}
We also recall that the unique mild solution of  
\eqref{eqn:SPDEintro}, i.e. 
\begin{align}\label{eq:truesol}
u(t)=S(t)u_0+\int_0^t S(t-s)f(u(s))ds+\int_0^t S(t-s)dW(s), 
\end{align}
enjoys the following properties. Firstly, if \eqref{assump eqn: growth of f}, Assumption \ref{assump:coercivity assumption on overall drift}, Assumption \ref{assump3} and Assumption \ref{assumption: on the covariance operator Q} hold,   with the same calculations as in the proof of \cite[Proposition 3.1]{brehier2014approximation}, one has that for every $m \in \N$ and for every $q\geq 1$, 
\begin{align}\label{eq:p61u0}
\sup\limits_{t\geq 0}\|u(t)\|_{L^{m}(\Omega;L^{q})} \leq C (1+ \|u_0\|_{L^{m}(\Omega;L^{q})}) \,.
\end{align}
In the above and throughout this section $L^m(\Omega;L^q)$ denotes the space of random functions v such that $\mathbb E \|v\|_{L^q}^m$ is finite.  

If also our Assumption \ref{assump2} holds, from \cite[Equation (3.60)]{WangW25}, \footnote{Our \eqref{eq:p61u1} is \cite[Equation (3.60)]{WangW25} when in the latter we take $\gamma=1$. Notice also that \cite[Equation (3.60)]{WangW25} is inside the proof of \cite[Theorem 3.8]{WangW25}. Such a theorem requires \cite[Assumption 2.1, 2.2, 3.6]{WangW25} to hold. The first and last assumptions are easily checked to hold under our assumptions. As for \cite[Assumption 2.2]{WangW25}, this is not used to prove \cite[Equation (3.60)]{WangW25}. 
} 
we have 
\begin{align}\label{eq:p61u1}
 \sup\limits_{t\in [0,T]} \|u(t)\|_{L^{m}(\Omega; \dot{H}^1)}\leq C(T)\big[\|u_0\|_{L^{m}(\Omega; \dot{H}^1)}+\sup\limits_{t\in [0,T]}\|f(u(t))\|_{L^{m}(\Omega;L^2)}+Tr(Q)\big].
\end{align}
Since  
\begin{align}\label{eq:p61f}
\sup\limits_{t\geq 0}\|f(u(t))\|_{L^{m}(\Omega;L^2)}&
\leq C \sup_{ t\geq 0}  \left[ 1+ \left(\mathbb{E} \|u(t)^p\|^m\right)^{1/m}\right] \nonumber \\
& = C \sup_{ t \geq 0}  \left[ 1+ \left(\mathbb{E} \|u(t)\|_{L^{2p}}^{pm}\right)^{1/m}\right] \nonumber\\
& \leq  C \left[ 1+ \left(\mathbb{E} \|u_0\|_{L^{2p}}^{pm}\right)^{1/m}\right]
= C (1+ \|u_0\|^p_{L^{m}(\Omega; L^{2p}) }) \, ,
\end{align} 
 substituting the above \eqref{eq:p61f} into \eqref{eq:p61u1} and using \eqref{eq:p61u0},  we obtain
\begin{align}\label{eq:p61u2}
\sup\limits_{t\in [0,T]} \|u(t)\|_{L^{m}(\Omega; H^1)}\leq C(T)(1+\|u_0\|_{L^{m}(\Omega; H^1)}+\|u_0\|^p_{L^m(\Omega; L^{2p})}) \,.
\end{align}
If also \eqref{C22} holds,  by the embedding $H^1\hookrightarrow L^\infty$ ( which holds for $d=1$), we can write  
\begin{align}\label{eq:nablaf}
\|\nabla f(u(t))\|&=\|f'(u(t))\nabla u(t)\|\leq C( 1+\|u(t)\|_\infty^{p-1})\|\nabla u(t)\|
\leq C(1+\|u(t)\|_1^p) \, ,
\end{align}
so that, by \eqref{eq:p61u2}, 
we also have 
\begin{align}\label{eq:fH1}
\sup\limits_{t\in[0,T]}\|f(u(t))\|_{L^{m}(\Omega;{H}^1)}\leq
C (T) (1+ \|u_0\|^p_{L^m(\Omega; {H}^1)}+ \|u_0\|^{p^2}_{L^m(\Omega; L^{2p})}) \,.
\end{align}
Finally, \eqref{eq:p61f} and Burkholder-Davis-Gundy inequality imply that for every $m \in \N$
\begin{align}\label{eq:contu}
&\mathbb{E}\|u(t)-u(s)\|^{2m}\\
&\leq C\mathbb{E}\bigg\|\int_s^t S(t-\sigma)f(u(\sigma))d\sigma\bigg\|^{2m}+C\mathbb{E}\bigg\|\int_s^t S(t-\sigma)dW(\sigma)\bigg\|^{2m}\notag\\
&\leq C(t-s)^{2m}\sup\limits_{t\in [0,T]}\mathbb{E}\|f(u(t))\|^{2m}+C\mathbb{E}\bigg\|\int_s^t S(t-\sigma)dW(\sigma)\bigg\|^{2m}\notag\\
&\leq C(t-s)^{2m}C(T)(1+\|u_0\|^{pm}_{L^{2m}(\Omega; L^{2p})})+C(T)|t-s|^m  \, . 
\end{align}


Throughout this section $\con(u_0)$ denotes a constant which depends on the initial datum and that may change from line to line. More precisely, it is a quantity of the form $\|u_0\|_{L^m(\Omega ; H^1)}$ for some $m$ (which will change from line to line), the important thing being that it depends only on the  moments of the $H^1$ norm (as opposed to other norms) of the initial datum. 
We also emphasize that the proofs of this section are rather standard, so we only give the main steps. 
\subsection{Finite time error bounds for the fully implicit scheme \eqref{scheme:im}} \label{subsec:proof of H2 for fully implicit} This bound has already been proved in \cite{liu2021strong} for deterministic $u_0\in \dot{H}^1$. Proposition \ref{Prop: finite tim error implicit} below  is substantially a restatement such results, in a notation and setting coherent to the one of our paper (and under the slight extension of random initial data).  

\begin{proposition}\label{Prop: finite tim error implicit}(Finite time error bound for FIE scheme  \eqref{scheme:im})
 Let $u_h^n$ be as in \eqref{scheme:im} and $u(t)$ be as in \eqref{eqn:SPDEintro}.  Suppose  \eqref{C21} holds for some constant $K_1 \in \R$.  Let also  \eqref{C22} and \eqref{assump eqn: growth of f} hold, together with  Assumption \ref{assump:coercivity assumption on overall drift}, Assumption \ref{assump3}, Assumptions \ref{assump2} and Assumption \ref{assumption: on the covariance operator Q}. Then 
\begin{align}\label{cimu0}
\|u(T)-u_h^n\|_{L^2(\Omega;L^2)}&\leq C(T) \con(u_0)(h+\tau^{\frac{1}{2}}) \, ,
\end{align}
where we recall that $\con(u_0)$ is a constant which depends on some moments of the $H^1$-norm of the initial datum $u_0$, while $C(T)$ is a generic constant which depends on the final time $T$. 
\end{proposition}

\begin{proof}[Sketch of proof of Proposition \ref{Prop: finite tim error implicit}] The statements of this proposition are a direct consequence of  the proof of \cite[Lemma 4.2 and Theorem 1.1]{liu2021strong}. We give minimal detail.  
First, we introduce an auxiliary process: 
\begin{align}\label{eq:uhatdef}
\breve{u}_h^{k+1}=S_{h,\tau}^{k+1}P_hu_0+\tau\sum\limits_{\ell=0}^kS_{h,\tau}^{k+1-\ell}P_hf(u(t_{\ell+1}))+\sum\limits_{\ell=0}^k S_{h,\tau}^{k+1-\ell}P_h[\delta W^{\ell+1}].
\end{align}
By the triangle inequality
\begin{equation}\label{eqn: use split implicit}
\left\|u\left(t_k\right)-u_h^k\right\|_{L^{2}(\Omega ; L^2)} \leq\left\|u\left(t_k\right)-\breve{u}_h^k\right\|_{L^{2}(\Omega ; L^2)}+\left\|\breve{u}_h^k-u_h^k\right\|_{L^{2}(\Omega ; L^2)}.
\end{equation}
The first term on the RHS of the above is studied in \cite[Lemma 4.2]{liu2021strong}. When comparing with that paper, recall that in that paper $q-1$ is what here we call $p$ (i.e. the growth of $f$) and take everywhere $\gamma=0$.

When the initial datum is deterministic, from \cite[equation (4.14) and bound after (4.19)]{liu2021strong} we have 
$$
\|u(t_k)-\breve{u}_h^k\|_{L^q(\Omega; L^2)}\leq C (h+\tau^{1/2}) (1+\|u_0\|_{\dot{H}^1}^{p} + \|u_0\|_{\dot{H}^1}^{p(p-1)}) \,.
$$
When the initial datum is random we just need to adapt slightly the calculations in \cite{liu2021strong}. So,  acting as in \cite[Proof of Lemma 4.2]{liu2021strong}, for any $q\geq 2$ we have
\begin{align}
\|u(t_k)-\breve{u}_h^k\|_{L^q(\Omega; L^2)}
&\leq \|E_{h,\tau}(t_k)u_0\|_{L^q(\Omega;L^2)}\\
&+\bigg\|\int_0^{t_k}S(t_k-s)f(u(s))ds-\tau\sum\limits_{\ell=0}^{k-1} S_{h,\tau}^{k-\ell}P_hf(u(t_{\ell+1}))\bigg\|_{L^q(\Omega;L^2)}\notag\\
&+\bigg\|\int_0^{t_k}S(t_k-s)dW^Q(s)-\sum\limits_{\ell=1}^k \int_{t_{\ell-1}}^{t_\ell}S_{h,\tau}^{k-\ell}P_h\delta W(s)\bigg\|_{L^q(\Omega;L^2)}\notag\\
&:=J_1+J_2+J_3.
\end{align}
We are using here the same notation $J_{1,2,3}$ as in \cite{liu2021strong} to help the reader track the proof. 
Using Lemma \ref{lem:serrop} (2) with $\mu=\nu=1$ and Assumption \ref{assump3} give 
\begin{align}
J_1\leq C(h+\tau^{\frac{1}{2}})\|u_0\|_{L^q(\Omega;H^1)}.
\end{align}
The term $J_2$ is then decomposed as follows.
\begin{align}
J_2&=\bigg\|\int_0^{t_k}S(t_k-s)f(u(s))ds-\tau\sum\limits_{\ell=0}^{k-1}S_{h,\tau}^{k-\ell}P_hf(u(t_{\ell+1}))\bigg\|_{L^q(\Omega;L^2)}\notag\\
&\leq \bigg\|\sum\limits_{\ell=0}^{k-1}\int_{t_{\ell}}^{t_{\ell+1}}S(t_k-s)(f(u(s))-f(u(t_{\ell+1})))ds\bigg\|_{L^q(\Omega;L^2)}\notag\\
&\quad+\bigg\|\sum\limits_{\ell=0}^{k-1} \int_{t_{\ell}}^{t_{\ell+1}} (S(t_k-s)-S_{h,\tau}^{k-\ell}P_h)f(u(t_{\ell+1}))ds\bigg\|_{L^q(\Omega;L^2)}\notag\\
&:=J_{21}+J_{22}.
\end{align}
From the estimate of the terms that are denoted $J_{21}^{m+1}$ and $J_{22}^{m+1}$ in \cite[Lemma 4.2]{liu2021strong}, we have
\begin{align}
J_{21}&\leq C(T)\big[1+\|u_0\|_{L^{2q(p-1)}(\Omega;\dot{H}^1)}^{p-1}\big][1+\|u_0\|_{L^{pq}(\Omega;\dot{H}^1)}^p]\tau^{\frac{1}{2}}
\end{align}
and 
\begin{align}
J_{22}&\leq C(T)(h+\tau^{\frac{1}{2}})\sup\limits_{t\in [0,T]}\|f(u(t))\|_{L^{q}(\Omega;L^2)}\notag\\
&\leq C(T)\con(u_0)(h+\tau^{\frac{1}{2}}),
\end{align}
where we have used \eqref{eq:p61f} and the embedding of $L^q$ into $H^1$. 
Therefore, 
\begin{align}
J_2\leq C(T) \con(u_0)(h+\tau^{\frac{1}{2}}) .
\end{align}
Similarly, by the estimate of $J_3^{m+1}$ of \cite[Lemma 4.2]{liu2021strong},  \eqref{eq:p61u2} and again  the embedding of $L^q$ into $H^1$, we have
\begin{align}
J_3&\leq C(T)(h+\tau^{\frac{1}{2}})(1+\sup\limits_{t\in [0,T]}\|u(t)\|_{L^q(\Omega;\dot{H}^1)})\leq C(T) \con(u_0)(h+\tau^{\frac{1}{2}}).
\end{align}
Therefore, 
\begin{align}\label{eq:erruhat1}
\|u(t_k)-\breve{u}_h^k\|_{L^q(\Omega;L^2)}&\leq C(T) \con(u_0)(h+\tau^{\frac{1}{2}})
\end{align}
{We now consider the second term in \eqref{eqn: use split implicit}, denoting the difference inside the norm as
\begin{equation}
\label{e:ekchk}
    \breve{e}^k:=u_h^k-\breve{u}_h^k.
\end{equation}
This} satisfies
\begin{align}
\breve{e}^k-\breve{e}^{k-1}=\tau\Delta_h\breve{e}^k+\tau P_h(f(u_h^k)-f(u(t_k))),\ \ \ \breve{e}^0=0.
\end{align}

 From the proof of \cite[Theorem 1.1]{liu2021strong} (see equation above \cite[equation (4.27)]{liu2021strong}), 
\begin{equation}
\begin{aligned}
 \mathbb{E}\|\breve{e}_h^{k}\|^2-\mathbb{E}\|\breve{e}_h^{k-1}\|^2 
 &\leq 2|K_1| \mathbb{E}\|\breve{e}_h^{k}\|^2 \tau+C\tau\sqrt{\mathbb{E}\left\|u\left(t_{k}\right)-\breve{u}_h^{k}\right\|^4} \\
&  \times\left(1+\left(\mathbb{E}\left\|u\left(t_{k}\right)\right\|_1^{2(p-1)}\right)^{\frac{1}{2}}+\left(\mathbb{E}\|\breve{u}_h^{k}\|_1^{2(p-1)}\right)^{\frac{1}{2}}\right)\,.
\end{aligned}
\end{equation}
Therefore, summing over $k$ gives
\begin{align}
\left(1-2 |K_1| \tau\right) \mathbb{E}\|\breve{e}_h^k\|^2 &\leq C\tau\sum\limits_{m=0}^{k-1}\sqrt{\mathbb{E}\left\|u\left(t_{m}\right)-\breve{u}_h^{m}\right\|^4}\left(1+\left(\mathbb{E}\left\|u\left(t_{m}\right)\right\|_{\dot{H}^1}^{2(p-1)}\right)^{\frac{1}{2}}+\left(\mathbb{E}\|\breve{u}_h^{m}\|_{\dot{H}^1}^{2(p-1)}\right)^{\frac{1}{2}}\right)\notag\\
&+2 |K_1| \tau \sum_{m=0}^{k-1} \mathbb{E}\left\|\breve{e}_h^m\right\|^2.\notag
\end{align}
Hence, the discrete Gronwall's inequality implies
\begin{align}\label{eq:ehat}
\mathbb{E}\|\breve{e}_h^k\|^2&\leq C\tau\sum\limits_{m=0}^{k-1}\sqrt{\mathbb{E}\left\|u\left(t_{m}\right)-\breve{u}_h^{m}\right\|^4}\left(1+\left(\mathbb{E}\left\|u\left(t_{m}\right)\right\|_{\dot{H}^1}^{2(p-1)}\right)^{\frac{1}{2}}+\left(\mathbb{E}\|\breve{u}_h^{m}\|_{\dot{H}^1}^{2(p-1)}\right)^{\frac{1}{2}}\right)e^{4|K_1|T}. 
\end{align}
Now, by \eqref{eq:uhatdef},
\begin{align}\label{eq:stabuhat}
\|\breve{u}_h^{m}\|_{L^{2(p-1)}(\Omega;H^1)}&\leq \|S_{h,\tau}^{m}P_hu_0\|_{L^{2(p-1)}(\Omega;H^1)}+\tau\sum\limits_{j=0}^{m-1}\big\|S_{h,\tau}^{m-j}P_hf(u(t_{m}))\big\|_{L^{2(p-1)}(\Omega;H^1)}\\
&+\sum\limits_{j=0}^{m-1} \|S_{h,\tau}^{m-j}P_h[\delta W^{j+1}] \|_{L^{2(p-1)}(\Omega;H^1)}\notag\\
&\leq \|u_0\|_{L^{2(p-1)}(\Omega;H^1)}+C(T)(1+\|u_0\|_{L^{2(p-1)p^3}(\Omega; \dot{H}^1)}^{p^2})+C(T)\|(-\Delta)^{\frac{1}{2}}Q^{\frac{1}{2}}\|_{HS}. \notag
\end{align}
Therefore, substituting the estimates of \eqref{eq:erruhat1}, \eqref{eq:p61u2} and \eqref{eq:stabuhat} into \eqref{eq:ehat} leads to
\begin{align}\label{eq:erruhat2}
\|\breve{e}_h^k\|_{L^2(\Omega;H)}&\leq C(T)(h+\tau^{\frac{1}{2}}) \con(u_0) e^{2|K_1|T}.
\end{align}
Collecting the estimates of \eqref{eq:erruhat1} and \eqref{eq:erruhat2}, we have
\begin{align}
\|u(t_k)-u_h^k\|_{L^2(\Omega;H)}&\leq C(T) \con(u_0)(h+\tau^{\frac{1}{2}}).
\end{align}
\end{proof}

\subsection{Finite time error bound for the scheme \eqref{scheme:taming1}}\label{subsec: tamed H2}
\begin{proposition}\label{Prop: one time  step estimate taming1}
Let $u_h^n$ and $u(t)$ be as in \eqref{scheme:taming1} and \eqref{eqn:SPDEintro}.  Suppose  \eqref{C21} holds for some constant $K_1 \in \R$.  Let also \eqref{C22} hold, together with Assumptions \ref{assump1}, Assumption \ref{assump:coercivity assumption on overall drift} (for $m=2$), Assumption \ref{assump3}, Assumptions \ref{assump2} and Assumption \ref{assumption: on the covariance operator Q}. Then, the following error estimate holds:
\begin{equation}\label{eqn:finite time bound for our scheme}
\mathbb E \|u(T) -u_h^n \|^2 \leq C(T) \con(u_0)(h^2+\tau),
\end{equation}
  where $n=T/\tau$. 
\end{proposition}
Before starting the proof, similarly to \eqref{eq:uhatdef}, we introduce an auxiliary process: 
\begin{align}\label{aux}
\breve{u}_h^{k+1}=S_{h,\tau}^{k+1}P_hu_0+\tau\sum\limits_{j=0}^k \frac{S_{h,\tau}^{k+1-j}P_hf(u(t_j))}{1+\tau\|\nabla f(u(t_j))\|^2}+\sum\limits_{j=0}^k S_{h,\tau}^{k+1-j}P_h[\delta W^{j+1}].
\end{align}
Hence, using the equivalence \cite{anderssonL16}
$$c_1\|(-\Delta_h)^\gamma v\|\leq \|(-\Delta)^\gamma v\|\leq c_2\|(-\Delta_h)^\gamma v\|,\;\; v\in V_h, \; -\frac{1}{2}\leq\gamma\leq \frac{1}{2}, $$
 where $c_1,c_2$ are two positive constants independent of $h$, and the Burkholder-Davis-Gundy inequality, we have
\begin{align}\label{eq:hatu}
\|\breve{u}_h^{k+1}\|_{L^{2q}(\Omega;\dot{H}^1)}&\leq \|u_0\|_{L^{2q}(\Omega;H^1)}+\tau \sum\limits_{j=0}^k \|(-\Delta_h)^{\frac{1}{2}}S_{h,\tau}^{k+1-j}P_hf(u(t_j))\|_{L^{2q}(\Omega;L^2)}\notag\\
&+\left\|\sum\limits_{j=0}^k(-\Delta_h)^{\frac{1}{2}}S_{h,\tau}^{k+1-j}P_h[\delta W^{j+1}]\right\|_{L^{2q}(\Omega;L^2)} \notag\\
&\leq \|u_0\|_{L^{2q}(\Omega;H^1)}+C\tau\sum\limits_{j=0}^k t_{k+1-j}^{-\frac{1}{2}}\|f(u(t_j))\|_{L^{2q}(\Omega;L^2)}\notag\\
&+C\bigg(\tau\sum\limits_{j=0}^k \|(-\Delta_h)^{\frac{1}{2}}S_{h,\tau}^{k+1-j}P_hQ^{\frac{1}{2}}\|^2_{HS}\bigg)^{\frac{1}{2}}\notag\\
&\leq C(T) (\con(u_0) +\|(-\Delta)^{\frac{1}{2}}Q^{\frac{1}{2}}\|_{HS}), 
\end{align}
where \eqref{eq:p61f} has been used in the last step.

\begin{lemma}\label{lem:erruhat}
With the notation introduced so far and with the same assumption as in Proposition \ref{Prop: one time  step estimate taming1}, the following holds: 
\begin{align}
\|u(t_k)-\breve{u}_h^k\|_{L^q(\Omega;L^2)}
\leq C(T)\con(u_0)(h+\tau^{\frac{1}{2}}) \, ,\notag
\end{align}
for $0\leq t_k \leq T$
or\  $k\leq T/\tau$. 
\end{lemma}

\begin{proof} Subtracting \eqref{aux} from \eqref{eq:truesol}, and taking associated norms, we have
\begin{align}
&\|u(t_k)-\breve{u}_h^k\|_{L^q(\Omega; L^2)}\notag\\
&\leq \|E_{h,\tau}(t_k)u_0\|_{L^q(\Omega;L^2)}+\bigg\|\int_0^{t_k}S(t_k-s)f(u(s))ds-\tau\sum\limits_{\ell=1}^{k} \frac{S_{h,\tau}^{k-\ell}P_hf(u(t_{\ell-1}))}{1+\tau\|\nabla f(u(t_{\ell-1}))\|^2}\bigg\|_{L^q(\Omega;L^2)}\notag\\
&+\bigg\|\int_0^{t_k}S(t_k-s)dW^Q(s)-\sum\limits_{\ell=1}^k \int_{t_{\ell-1}}^{t_\ell}S_{h,\tau}^{k-\ell}P_h\delta W(s)\bigg\|_{L^q(\Omega;L^2)}\notag\\
&:=I_1+I_2+I_3.
\end{align}
Using the first bound  of Lemma \ref{lem:serrop}  with $\mu=\nu=1$ and Assumption \ref{assump3} gives 
\begin{align}
I_1\leq C(h+\tau^{\frac{1}{2}})\|u_0\|_{L^q(\Omega;H^1)}.
\end{align}
The term $I_2$ can be decomposed in the following way:
\begin{align}
I_2 & =\bigg\|\sum\limits_{\ell=1}^k\int_{t_{\ell-1}}^{t_\ell}\bigg[S(t_k-s)f(u(s))-\frac{S_{h,\tau}^{k-\ell}P_hf(u(t_{\ell-1}))}{1+\tau\|\nabla f(u(t_{\ell-1}))\|^2}\bigg]ds\bigg\|_{L^q(\Omega;L^2)}\notag\\
&\leq \bigg\|\sum\limits_{\ell=1}^k\int_{t_{\ell-1}}^{t_\ell}S(t_k-s)[f(u(s))-f(u(t_{\ell-1}))]ds\bigg\|_{L^q(\Omega;L^2)}\notag\\
&+\bigg\|\sum\limits_{\ell=1}^k\int_{t_{\ell-1}}^{t_\ell}[S(t_k-s)-S_{h,\tau}^{k-\ell}P_h]f(u(t_{\ell-1}))ds\bigg\|_{L^q(\Omega;L^2)}\notag\\
&+\bigg\|\sum\limits_{\ell=1}^k\int_{t_{\ell-1}}^{t_\ell}\bigg[S_{h,\tau}^{k-\ell}P_hf(u(t_{\ell-1}))-\frac{S_{h,\tau}^{k-\ell}P_hf(u(t_{\ell-1}))}{1+\tau\|\nabla f(u(t_{\ell-1}))\|^2}\bigg]ds\bigg\|_{L^q(\Omega;L^2)}\notag\\
&\quad:=I_{21}+I_{22}+I_{23}.
\end{align}

To work on $I_{21}$,  let us first compute the continuity of $f$. To do so, using  \eqref{eqn: f polynomial} we have
\begin{align}
\mathbb{E}\|f(u(s))-f(u(t_{\ell-1}))\|^q &\leq 
C\mathbb{E}[\|u(s)-u(t_{\ell-1})\|(1+\|u(s)\|_\infty^{p+1}+\|u(t_{\ell-1})\|_\infty^{p+1})]^q\notag\\
&\leq C\sqrt{\mathbb{E}\|u(s)-u(t_{\ell-1})\|^{2q}}\sqrt{\mathbb{E}(1+\|u(s)\|_\infty^{p+1}+\|u(t_{\ell-1})\|_\infty^{p+1})^{2q}} \,.
\end{align}
Hence, using again the embedding $H^1\hookrightarrow L^\infty$ ($d=1$), 
\begin{align}
&\|f(u(s))-f(u(t_{l-1}))\|_{L^q(\Omega;L^2)} \\
&\leq C\|u(s)-u(t_{\ell-1})\|_{L^{2q}(\Omega;L^2)}\big(1+\|u(s)\|_{L^{2q}(\Omega; H^1)}^{p+1}+\|u(t_{\ell-1})\|_{L^{2q}(\Omega; H^1)}^{p+1}\big). \notag
\end{align}
Therefore,  \eqref{eq:contu} and \eqref{eq:p61u2} imply
\begin{align*}
I_{21}&=\bigg\|\sum\limits_{\ell=1}^k\int_{t_{\ell-1}}^{t_\ell} S(t_k-s)(f(u(s))-f(u(t_{\ell-1})))ds\bigg\|_{L^q(\Omega;L^2)}\notag\\
 &\leq \sum\limits_{\ell=1}^k\int_{t_{\ell-1}}^{t_\ell} \|f(u(s))-f(u(t_{l-1}))\|_{L^q(\Omega;L^2)}ds\leq C(T) \con(u_0) \,. 
\end{align*}
As for  $I_{22}$,  first  
by Lemma \ref{lem:serrop} (1) with $\nu=1/2$, we  have
\begin{align}\label{eq:estfu1}
&\|[S(t_k-s)-S(t_k-t_{\ell-1})]f(u(t_{\ell-1}))\|_{L^q(\Omega;L^2)} \notag\\
&=\|(-\Delta)^{\frac{1}{2}}S(t_k-s)(-\Delta)^{-\frac{1}{2}}(\mathbb{I}-S(s-t_{\ell-1}))f(u(t_{\ell-1}))\|_{L^q(\Omega;L^2)}\notag\\
&\leq \|(-\Delta)^{\frac{1}{2}}S(t_k-s)\| \|(-\Delta)^{-\frac{1}{2}}(\mathbb{I}-S(s-t_{\ell-1}))\| \|f(u(t_{\ell-1}))\|_{L^q(\Omega;L^2)}\notag\\
&\leq C(t_k-s)^{-\frac{1}{2}}\tau^{\frac{1}{2}}\|f(u(t_{\ell-1}))\|_{L^q(\Omega;L^2)}
\end{align}
and by Lemma \ref{lem:serrop} (2) with $\mu=1, \nu=0$, we obtain
\begin{align}\label{eq:estfu2}
 \|[S(t_k-t_{\ell-1})-S_{h,\tau}^{k-\ell}P_h]f(u(t_{\ell-1}))\|_{L^q(\Omega;L^2)}\leq C(h+\tau^{\frac{1}{2}})(t_k-t_{\ell-1})^{-\frac{1}{2}}\|f(u(t_{\ell-1}))\|_{L^q(\Omega;L^2)}.
\end{align}
Therefore, \eqref{eq:p61f} gives
\begin{align}
I_{22}&\leq \bigg\|\sum\limits_{\ell=1}^k\int_{t_{\ell-1}}^{t_\ell}[S(t_k-s)-S(t_k-t_{\ell-1})]f(u(t_{\ell-1}))ds\bigg\|_{L^q(\Omega;L^2)}\notag\\
&+\bigg\|\sum\limits_{\ell=1}^k\int_{t_{\ell-1}}^{t_\ell}[S(t_k-t_{\ell-1})-S_{h,\tau}^{k-\ell}P_h]f(u(t_{\ell-1}))ds\bigg\|_{L^q(\Omega;L^2)}\notag\\
&\leq \sum\limits_{\ell=1}^k\int_{t_{\ell-1}}^{t_k} \|[S(t_k-s)-S(s-t_{\ell-1})]f(u(t_{\ell-1}))\|_{L^q(\Omega;L^2)}ds\notag\\
&+ \sum\limits_{\ell=1}^k\int_{t_{\ell-1}}^{t_k} \|[S(t_k-t_{\ell-1})-S_{h,\tau}^{k-\ell}P_h]f(u(t_{\ell-1}))\|_{L^q(\Omega;L^2)} ds\notag\\
&\leq C(h+\tau^{\frac{1}{2}}) \sum\limits_{\ell=1}^k\int_{t_{\ell-1}}^{t_\ell} (t_k-s)^{-\frac{1}{2}} \|f(u(t_{\ell-1}))\|_{L^q(\Omega;L^2)}ds
\leq C(T) \con(u_0)(h+\tau^{\frac{1}{2}}). \label{eq:I22}
\end{align}
Similarly,  
\begin{align}
I_{23}&= \bigg\|\sum\limits_{\ell=1}^k\int_{t_{\ell-1}}^{t_\ell}\bigg[S_{h,\tau}^{k-\ell}f(u(t_{\ell-1}))-\frac{S_{h,\tau}^{k-\ell}P_hf(u(t_{\ell-1}))}{1+\tau\|\nabla f(u(t_{\ell-1}))\|^2}\bigg]ds\bigg\|_{L^q(\Omega;L^2)}\notag\\
&\leq \sum\limits_{\ell=1}^k\int_{t_{\ell-1}}^{t_\ell}\bigg\|S_{h,\tau}^{k-\ell}f(u(t_{\ell-1}))-\frac{S_{h,\tau}^{k-\ell}P_hf(u(t_{\ell-1}))}{1+\tau\|\nabla f(u(t_{\ell-1}))\|^2}\bigg\|_{L^q(\Omega;L^2)}ds\notag\\
&\leq \sum\limits_{\ell=1}^k\int_{t_{\ell-1}}^{t_\ell} \bigg\|\frac{(S_{h,\tau}^{k-\ell}-S_{h,\tau}^{k-\ell}P_h)f(u(t_{\ell-1}))
+\tau\|\nabla f(u(t_{\ell-1}))\|^2 S_{h,\tau}^{k-\ell}f(u(t_{\ell-1}))}{1+\tau\|\nabla f(u(t_{\ell-1}))\|^2}\bigg\|_{L^q(\Omega;L^2)}ds\notag\\
&\leq \sum\limits_{\ell=1}^k\int_{t_{\ell-1}}^{t_\ell} \|(S_{h,\tau}^{k-\ell}-S_{h,\tau}^{k-\ell}P_h)f(u(t_{\ell-1}))\|_{L^q(\Omega;L^2)}ds\notag\\
&+\sum\limits_{\ell=1}^k\int_{t_{\ell-1}}^{t_\ell}  \tau \bigg\|\frac{\|\nabla f(u(t_{\ell-1}))\|^2 S_{h,\tau}^{k-\ell}f(u(t_{\ell-1}))}{1+\tau\|\nabla f(u(t_{\ell-1}))\|^2}\bigg\|_{L^q(\Omega;L^2)}ds.
\end{align}
Note that the first term on the RHS can be controlled by using \eqref{eq:estfu2}, giving :
\begin{align}\label{eq:I231}
\sum\limits_{\ell=1}^k\int_{t_{\ell-1}}^{t_\ell} \|(S_{h,\tau}^{k-\ell}-S_{h,\tau}^{k-\ell}P_h)f(u(t_{\ell-1}))\|_{L^q(\Omega;L^2)}\leq C(T)\con(u_0)(h+\tau^{\frac{1}{2}}).
\end{align}
For the  second term:  
\begin{align} \label{eq:I232}
&\sum\limits_{\ell=1}^k\int_{t_{\ell-1}}^{t_\ell}  \tau \bigg\|\frac{\|\nabla f(u(t_{\ell-1}))\|^2 S_{h,\tau}^{k-\ell}f(u(t_{\ell-1}))}{1+\tau\|\nabla f(u(t_{\ell-1}))\|^2}\bigg\|_{L^q(\Omega;L^2)}ds\notag\\
&\leq \tau \sum\limits_{\ell=1}^k\int_{t_{\ell-1}}^{t_\ell} \bigg(\mathbb{E}\|\nabla f(u(t_{\ell-1})) \|^{2q}\|f(u(t_{\ell-1}))\|^q\bigg)^{\frac{1}{q}}ds \leq \tau C(T) \con(u_0) \, ,
\end{align}
where \eqref{eq:nablaf} and  \eqref{eq:fH1} have been used. Clearly, this is a higher-order term.  
Combining the result of \eqref{eq:I231} and \eqref{eq:I232}, we have
\begin{align}
I_{23}&\leq C(T) \con(u_0)(h+\tau^{\frac{1}{2}}).
\end{align}
Collecting the estimates of $I_{21}, I_{22}$ and $I_{23}$ we obtain
\begin{align}
I_2\leq C(T) \con(u_0)(h+\tau^{\frac{1}{2}}).
\end{align}
  To estimate $I_3$, similarly to \cite[Theorem 3.14]{Kruse},   we have
\begin{align}
\sum\limits_{\ell=1}^k S_{h,\tau}^{k-\ell}P_h\delta W^\ell=\int_0^{t_k} S_{h,\tau}(s)P_hdW(s). 
\end{align}
Hence, by Lemma \ref{lem:serrop} (3) (with $\mu=0$)
\begin{align}
\mathbb{E}I_3^q&=\mathbb{E}\bigg\|\int_0^{t_k}(S(t_k-s)-S_{h,\tau}(t_k-s))P_hdW(s)\bigg\|^q\notag\\
&\leq C\bigg(\int_0^{t_k} \|[S(t_k-s)-S_{h\tau}(t_k-s)P_h]Q^{\frac{1}{2}}\|_{HS}^2ds\bigg)^\frac{q}{2}\notag\\
&\leq C\bigg(\sum\limits_{j=1}^\infty \gamma_j\int_0^{t_k} \|[S(t_k-s)-S_{h\tau}(t_k-s)P_h]e_j\|^2ds\bigg)^\frac{q}{2}\notag\\
&\leq C \bigg(\sum\limits_{j=1}^\infty \gamma_j (h^2+\tau)\|e_j\|^2\bigg)^\frac{q}{2}\notag\\
&\leq C(T) Tr(Q)^{\frac{q}{2}}(h^2+\tau)^{\frac{q}{2}},
\end{align}
where we applied the Burkholder-Davis-Gundy inequality in the first step.
The statement of the lemma follows by combining the estimates of $I_1, I_2$ and $I_3$. 
\end{proof}

\begin{proof}[Proof of Proposition \ref{Prop: one time  step estimate taming1}]
Let us write
$$
\|u(t_k)-u_h^k\|_{L^2(\Omega;L^2)}\leq \|u(t_k)-\breve{u}_h^k\|_{L^2(\Omega;L^2)}+\|\breve{u}_h^k-u_h^k\|_{L^2(\Omega;L^2)}, 
$$
where $\breve{u}_h^k$ is defined in \eqref{aux}. The first added on the RHS of the above has already been estimated in  Lemma \ref{lem:erruhat}, so we only need to bound the second term. 

Denote $\breve{e}^k:=u_h^k-\breve{u}_h^k$. Then $\breve{e}_k$ satisfies the equation
\begin{align}
& \breve{e}^k-\breve{e}^{k-1}= \tau \Delta_h \breve{e}^k+\frac{\tau P_hf\left(u_h^{k-1}\right)}{1+\tau\left\|\nabla f\left(u_h^{k-1}\right)\right\|^2}-\frac{\tau P_h f\left(u\left(t_{k-1}\right))\right.}{1+\tau\left\|\nabla f\left(u\left(t_{k-1}\right)\right)\right\|^2} \\
& \breve{e}_0=0 .
\end{align}
Taking inner product of both sides with $\breve{e}^k$,  recalling the definition of $\Delta_h$ in \eqref{def:dislap} and the fact that $(u,v)=(P_hu,v)$ for any $ v\in V_h$, we have 
\begin{align}\label{eq:J}
& \frac{1}{2}\|\breve{e}^k\|^2-\frac{1}{2}\|\breve{e}^{k-1}\|^2+\frac{1}{2}\|\breve{e}^k-\breve{e}^{k-1}\|^2+\tau\|\nabla \breve{e}^k\|^2\notag \\
& =\left(\frac{\tau f\left(u_h^{k-1}\right)}{1+\tau\left\|\nabla f\left(u_h^{k-1}\right)\right\|^2}-\frac{\tau f\left(u\left(t_{k-1}\right)\right.}{1+\tau\left\|\nabla f\left(u\left(t_{k-1}\right)\right)\right\|^2}, \breve{e}^k\right):=J.
\end{align}

Computing further the expression for $J$ gives
$$
\begin{aligned}
J & =\frac{\tau\left(f\left(u_h^{k-1}\right)-f\left(u\left(t_{k-1}\right)\right), \breve{e}^k\right)+\tau^2\left(\left\|\nabla f\left(u\left(t_{k-1}\right)\right)\right\|^2 f\left(u_h^{k-1}\right)-\| \nabla f\left(u_h^{k-1}\right) \|^2 f\left(u\left(t_{k-1}\right)\right), \breve{e}^k\right)}{\left(1+\tau\left\|\nabla f\left(u\left(t_{k-1}\right)\right)\right\|^2\right)\left(1+\tau\left\|\nabla f\left(u_h^{k-1}\right)\right\|^2\right)} \\
& =\frac{\tau^2\left(\left\|\nabla f\left(u\left(t_{k-1}\right)\right)\right\|^2-\left\|\nabla f\left(u_h^{k-1}\right)\right\|^2\right)\left(f\left(u\left(t_{k-1}\right)\right), \breve{e}^k\right)}{\left(1+\tau\left\|\nabla f\left(u\left(t_{k-1}\right)\right)\right\|^2\right)\left(1+\tau\left\|\nabla f\left(u_h^{k-1}\right)\right\|^2\right)}+\frac{\tau\left(f\left(u_h^{k-1}\right)-f\left(u\left(t_{k-1}\right), \breve{e}^k\right)\right.}{1+\tau\left\|\nabla f\left(u_h^{k-1}\right)\right\|^2}\\
&:=J_1+J_2.
\end{aligned}
$$
By Young's inequality
\begin{align}
\mathbb{E} J_1 & \leq C \tau^3 \mathbb{E}
\bigg[\left(\left\|\nabla f\left(u_h^{k-1}\right)\right\|^2
+\|\nabla f(u(t_{k-1}))\|^2 \right)\|f(u(t_{k-1}))\|^2\bigg]+\frac{\tau}{4} \mathbb{E}\| \breve{e}^k\|^2 \notag \\
&\leq C\tau^3\bigg(\sqrt{\mathbb{E}\|\nabla f(u_h^{k-1})\|^4}+\sqrt{\mathbb{E}\|\nabla f(u(t_{k-1}))\|^4} \bigg)\sqrt{\mathbb{E}\|f(u(t_{k-1}))\|^4}
+\frac{\tau}{4} \mathbb{E}\|\breve{e}^k\|^2.
\end{align}
Using \eqref{eq:fH1}  and  our $H^1$ stability result Proposition \ref{them: uniform in time moment bound for tamed}, we land with 
\begin{align}
\mathbb{E}J_1\leq \tau^3 C(T)\con(u_0)+\frac{\tau}{4} \mathbb{E}\|\breve{e}^k\|^2.
\end{align}
Next, let us estimate $J_2$. Using the Young's inequality and \eqref{C21}, 
\begin{align}
J_2 & =\frac{\tau\left(
f\left(u\left(t_{k-1}\right)-f\left(u_h^{k-1}\right), \breve{e}^k\right)\right.}{1+\tau\left\|\nabla f\left(u_h^{k-1}\right)\right\|^2} \\
& =\frac{\tau\left(f\left(u\left(t_{k-1}\right)-f\left(\breve{u}_h^{k-1}\right), \breve{e}^k\right)+\tau\left( f\left(\breve{u}_h^{k-1}\right)-f\left(u_h^{k-1}\right), \breve{e}^{k-1}\right)\right.}{1+\tau\left\|\nabla f\left(u_h^{k-1}\right)\right\|^2} \notag\\
&+\frac{\tau\left(f\left(\breve{u}_h^{k-1}\right)-f\left(u_h^{k-1}\right), \breve{e}^k-\breve{e}^{k-1}\right)}
{1+\tau\left\|\nabla f\left(u_h^{k-1}\right)\right\|^2}\notag\\
& \leq \tau\left[\|f(u(t_{k-1}))-f\left(\breve{u}_h^{k-1}\right)\|^2+\frac{1}{4}\| \breve{e}^k \|^2\right]+ |K_1|\tau\| \breve{e}^{k-1}\|^2+\tau^2\left\| f\left(\breve{u}_h^{k-1}\right)-f\left(u_h^{k-1}\right) \right\|^2 \notag\\
&+\frac{1}{4}\|\breve{e}^k-\breve{e}^{k-1}\|^2 .
\end{align}
From Lemma \ref{lem:erruhat}, the embedding $H^1\hookrightarrow L^\infty$  and the bounds \eqref{eqn: f polynomial}, \eqref{eq:p61u2} and \eqref{eq:hatu} we then deduce
\begin{align}\label{eqn:fuuhat}
& \tau \mathbb{E} \| f\left(u\left(t_{k-1}\right)\right)-f\left(\breve{u}_h^{k-1}\right) \|^2 \\
& \leq C \tau \mathbb{E}\left[\left(1+\left\|u\left(t_{k-1}\right)\right\|_{\infty}^{2p+2}+\|\breve{u}_h^{k-1}\|_{\infty}^{2p+2}\right)\|u\left(t_{k-1}\right)-\breve{u}_h^{k-1}\|^2\right] \notag\\
& \leq C \tau \sqrt{\mathbb{E}\left(1+\left\|u\left(t_{k-1}\right)\right\|_{\infty}^{4p+4}+\|\breve{u}_h^{k-1}\|_{\infty}^{4p+4}\right)}\sqrt{\mathbb{E}\|u\left(t_{k-1}\right)-\breve{u}^{k-1}\|^4}\notag \\
& \leq C \tau \sqrt{\mathbb{E}\left(1+\left\|u\left(t_{k-1}\right)\right\|_{1}^{4p+4}+\|\breve{u}_h^{k-1}\|_{1}^{4p+4}\right)}\sqrt{\mathbb{E}\|u\left(t_{k-1}\right)-\breve{u}^{k-1}\|^4} \notag\\
& \leq C \tau (h^2+\tau) C(T) \con(u_0) 
\end{align}
and, similarly, 
\begin{align}
&\tau^2 \mathbb{E}\|f\left(\breve{u}_h^{k-1}\right)-f\left(u_h^{k-1}\right)\|^2 \notag\\
&\leq 2\tau^2\left[\mathbb{E}\|f(\breve{u}_h^{k-1})\|^2+\mathbb{E}\|f(u_h^{k-1})\|^2\right]\notag\\
&\leq C\tau^2\left[\mathbb{E}\|\breve{u}_h^{k-1}\|_\infty^{2p}+\mathbb{E}\|u_h^{k-1}\|_\infty^{2p}\right]\notag\\
&\leq C\tau^2\left[\mathbb{E}\|\breve{u}_h^{k-1}\|_{{H}^1}^{2p}+\mathbb{E}\|u_h^{k-1}\|_{{H}^1}^{2p}\right]\notag\\
&\leq C(T) \tau^2 \con(u_0).
\end{align}
Hence, we drop high-order terms, and have
\begin{align}\label{eq:estJ}
\mathbb{E} J &\leq \tau (h^2+\tau)C(T) \con(u_0)+ \frac{\tau}{2} \mathbb{E}\|\breve{e}^k\|^2+ K_1\tau \mathbb{E}\|\breve{e}^{k-1}\|^2+\frac{1}{4} \mathbb{E}\|\breve{e}^k-\breve{e}^{k-1}\|^2 .
\end{align}
Taking expectation on both sides of \eqref{eq:J} and applying \eqref{eq:estJ}, we have
\begin{align}\label{eq:hate1}
\mathbb{E}&\|\breve{e}^k\|^2\leq A(\tau)\mathbb{E}\|\breve{e}^{k-1}\|^2+\tau (h^2+\tau)C(T) \con(u_0)\notag
\end{align}
where
$$
A(\tau):=\frac{1+2 |K_1|\tau}{1-\tau}.
$$
Using the elementary inequality $(1+x)^k\leq e^{kx}, x,k>0$, we have
\begin{align}
& A(\tau)^k\leq e^{(2|K_1|+1)T}, \notag\\
&\sum\limits_{\ell=1}^{k-1}A(\tau)^\ell=\frac{1-A(\tau)^k}{1-A(\tau)}\leq \frac{e^{(2|K_1|+1)T}}{(2|K_1|+1)\tau}.
\end{align}
Therefore,
\begin{align}\label{eq:hate2}
\mathbb{E}\|\breve{e}^k\|^2 &\leq A(\tau)^k\mathbb{E}\|\breve{e}^0\|^2+\sum\limits_{\ell=1}^{k-1} A(\tau)^\ell\left[\tau (h^2+\tau)C(T) \con(u_0)\right]\notag \\
&\leq (h^2+\tau)C(T) e^{(2|K_1|+1)T} \con(u_0) \,.\notag
\end{align}
This concludes the proof. 
\end{proof}

\bibliographystyle{abbrv}
\bibliography{bib}

\end{document}